\documentclass{amsart}
\usepackage[utf8]{inputenc}
\usepackage{lmodern}
\usepackage{anyfontsize}
\usepackage{mathrsfs} 
\usepackage{amssymb}
\usepackage{bm}
\usepackage{graphicx}
\usepackage[centertags]{amsmath}
\usepackage{amsfonts}
\usepackage{amsthm}
\usepackage{enumerate}
\usepackage{tocvsec2}
\usepackage{xcolor}
\usepackage[margin=1in]{geometry}

\newtheorem{theorem}{Theorem}[section]
\newtheorem*{theorem*}{Theorem}

\newtheorem{corollary}[theorem]{Corollary}
\newtheorem{lemma}[theorem]{Lemma}
\newtheorem{rem}[theorem]{Remark}

\newtheorem{proposition}[theorem]{Proposition}

\newtheorem{fact}[theorem]{Fact}
\newtheorem*{fact*}{Fact}

\newtheorem{question}{Question}[section]

\theoremstyle{definition}

\newcommand{\ee}{\varepsilon}
\newcommand{\nn}{\mathbb{N}}

\begin{document}
\title{Three and a half asymptotic properties}

\begin{abstract} We define and discuss transfinite asymptotic notions of smoothability, type, and equal norm type. We prove distinctness of these notions for a proper class of ordinals and that each class is an ideal.  We also extend some results from \cite{GKL} to operators and ordinals greater than zero regarding the equivalence of equal norm asymptotic type and uniform renormings with power type smoothness.    Finally, we discuss an extension of a non-linear result for quasi-reflexive, asymptotically $p$-smoothable Banach spaces to quasi-reflexive Banach spaces with asymptotic equal norm type $p$.

\end{abstract}

\author{R.M. Causey}
\address{Department of Mathematics, Miami University, Oxford, OH 45056, USA}
\email{causeyrm@miamioh.edu}


\maketitle

\section{Introduction}

An important theorem in Banach space theory is Enflo's \cite{Enflo} result that every super-reflexive Banach space is isomorphic to a uniformly convex Banach space.  Pisier \cite{Pisier} later proved that any super-reflexive Banach space is isomorphic to a uniformly convex Banach space the modulus of convexity (and by duality, smoothness) of which has some power type. Furthermore, Pisier defines for each $1<p\leqslant 2$ some isomorphic property regarding infinite sequences of Walsh-Paley martingales which exactly characterizes whether a given Banach space admits an equivalent norm which is smooth with power type $p$.  A somewhat more tractable notion is that of Haar type, which involves an inequality similar to that given by Pisier, but which is only assumed to hold uniformly for finite sequences of Walsh-Paley martingales, as opposed to infinite sequences. Yet another distinct notion is that of equal norm type.  Tzafriri \cite{Tzafriri} showed that for every $1<p\leqslant 2$, there exists a Banach space with equal norm Rademacher type $p$, but without Rademacher type $p$.

Analogously, Godefroy, Kalton, and Lancien \cite{GKL} defined the Szlenk power type $\textsf{p}(X)$ of a given Banach space $X$ with $Sz(X)=\omega$.     There it was argued, using previous work of Lancien \cite{Lancien}, that if $Sz(X)=\omega$, $1\leqslant \textsf{p}(X)<\infty$.  A sharp renorming theorem from \cite{GKL} was that if $1/p+1/\textsf{p}(X)=1$, then $p$ is the supremum of those $r\in (1,\infty)$ such that $X$ admits an equivalent $r$-asymptotically uniformly smooth norm. In \cite{C3}, for every ordinal $\xi$, the $\xi$-Szlenk power type $\textsf{p}_\xi(A)$ of an operator $A:X\to Y$ was defined, and the renorming theorem from \cite{GKL} was generalized: The operator $A$ admits an equivalent power type $\xi$-asymptotically uniformly smooth norm if and only if $\textsf{p}_\xi(A)<\infty$, and if $q=\max\{1, \textsf{p}_\xi(A)\}$ and $1/p+1/q=1$, then $p$ is the supremum of those $r\in (1, \infty)$ such that $A$ admits an equivalent, $\xi$-$r$-asymptotically uniformly smooth norm.    The question arises as to when this supremum is attained.   That is, if $\textsf{p}_\xi(A)\in [1, \infty)$ and $1/p+1/\textsf{p}_\xi(A)=1$, does $A$ admit an equivalent, $\xi$-$p$-asymptotically uniformly smooth norm? Already in \cite{GKL}, this question was answered negatively in the $\xi=0$ case, so it is natural to seek for each ordinal $\xi$ and $1<p\leqslant \infty$ some isomorphic property which characterizes $\xi$-$p$-asymptotic uniform smoothability.    This was accomplished in \cite{C4}, giving a transfinite, operator, asymptotic analogue of Pisier's infinite Walsh-Paley martingale property which characterizes $p$-smoothability.  Given the existence of intermediate properties between $p$-smoothability and $r$-smoothability for all $1<r<p$ (namely, having Haar type $p$ and equal norm Haar type $p$), one may ask whether there are analogous asymptotic, transfinite properties, and whether or not they are distinct.  It is the primary goal of this work to answer this question.  For every ordinal $\xi$ and every $1<p\leqslant \infty$, we will define four classes of operators $\mathfrak{T}_{\xi,p}$, $\mathfrak{A}_{\xi,p}$, $\mathfrak{N}_{\xi,p}$, and $\mathfrak{P}_{\xi,p}$.    Roughly speaking, these classes are defined in terms of properties which are transfinite, asymptotic analogues of the notions of $p$-smoothability, Haar type $p$, equal norm Haar type $p$, and $r$-smoothability for each $1<r<p$.    Our primary result is the following.  In what follows, $\mathfrak{D}_\xi$ denotes the class of operators with Szlenk index not exceeding $\omega^\xi$.

\begin{theorem} For any ordinal $\xi$ and $1<p<\infty$,  $$\mathfrak{D}_\xi\subsetneq  \mathfrak{T}_{\xi,\infty}\subset \mathfrak{A}_{\xi,\infty}=\mathfrak{N}_{\xi,\infty}\subset\mathfrak{P}_{\xi,\infty}\subsetneq \mathfrak{T}_{\xi,p}\subset \mathfrak{A}_{\xi,p}\subset \mathfrak{N}_{\xi,p}\subset \mathfrak{P}_{\xi,p}\subsetneq \mathfrak{D}_{\xi+1}.$$

Furthermore, if $\xi$ has countable cofinality, and in particular if $\xi$ is countable, $$\mathfrak{T}_{\xi,p}\subsetneq \mathfrak{A}_{\xi,p}\subsetneq \mathfrak{N}_{\xi,p}\subsetneq \mathfrak{P}_{\xi,p}$$ and $$\mathfrak{T}_{\xi,\infty}\subsetneq \mathfrak{A}_{\xi,\infty}=\mathfrak{N}_{\xi,\infty}\subsetneq \mathfrak{P}_{\xi,\infty}.$$

\end{theorem}

We also provide a characterization of the transfinite, asymptotic equal norm type notions in terms of equivalent norms on the space. This result extends the $\xi=0$ spatial case proved in \cite{GKL} to all ordinals and operators. In what follows, for an ordinal $\xi$, $\sigma>0$, and an operator $A:X\to Y$, $\varrho_\xi(\sigma;A:X\to Y)$ denotes the $\xi$-modulus of asymptotic uniform smoothness.

\begin{theorem} Let $\xi$ be an ordinal,  $p\in (1,\infty)$, and $A:X\to Y$ be an operator.  \begin{enumerate}[(i)]\item $A\in \mathfrak{N}_{\xi,p}$ if and only if there exist constants $C,b\geqslant 1$  such that for any $\sigma>0$, there exists an equivalent norm $|\cdot|$ on $Y$ such that $$b^{-1}B_Y^{|\cdot|} \subset B_Y\subset b B_Y^{|\cdot|}$$ and $\varrho_\xi(\sigma;A:X\to (Y, |\cdot|))\leqslant C\sigma^p$. \item $A\in \mathfrak{N}_{\xi,\infty}$ if and only if there exist $b\geqslant 1$ and  $\sigma>0$ such that for any $\ee>0$, there exists an equivalent norm $|\cdot|$ on $Y$ such that $$b^{-1}B_Y^{|\cdot|}\subset B_Y\subset bB_Y^{|\cdot|}$$ and $$\varrho_\xi(\sigma;A:X\to (Y, |\cdot))\leqslant \ee.$$   \end{enumerate}

\end{theorem}

We also quantitatively extend a result of Lancien and Raja regarding Lipschitz embeddings of certain metric  spaces $G_k$ into quasi-reflexive Banach spaces with nice asymptotic smoothability properties.   To that end, we have the following.

\begin{theorem} If $1<p\leqslant \infty$, $X$ is a quasi-reflexive Banach space, and $I_X\in \mathfrak{N}_{0,p}$, then there exists a constant $C$ such that for any $k\in\nn$ and any Lipschitz map $f:G_k\to X$, there exists an infinite subset $M$ of $\nn$ such that for any $m_1<n_1<\ldots <m_k<n_k$, $m_i, n_i\in M$, $$\|f(m_1, m_2, \ldots, m_k)-f(n_1, n_2, \ldots, n_k)\|\leqslant C\text{\emph{Lip}}(f) k^{1/p}.$$

\end{theorem}

This result extends the theorem of Lancien and Raja, who prove the preceding theorem under the assumption that is quasi-reflexive and $I_X\in \mathfrak{T}_{0,p}$, which is a proper subideal of $\mathfrak{N}_{0,p}$.   Combining the theorem of Lancien and Raja with the renorming theorem from \cite{GKL} would allow one to deduce that if $X$ is quasi-reflexive and $I_X\in \mathfrak{N}_{0,p}$, then for any $1<r<p$, there exists a constant $C_r$ such that for any $k\in\nn$ and any Lipschitz map $f:G_k\to X$, there exists an infinite subset $M$ of $\nn$ such that for any $m_1<n_1<\ldots <m_k<n_k$, $m_i, n_i\in M$, $$\|f(m_1, m_2, \ldots, m_k)-f(n_1, n_2, \ldots, n_k)\|\leqslant C_r\text{Lip}(f) k^{1/r}.$$  This constant $C_r$ depends upon two other quantities: A constant $b_r$ such that there exists a $b_r$-equivalent norm $|\cdot|$ on $X$ which is $r$-asymptotically uniformly smooth, and a constant $a_r$ such that $\varrho_X(\sigma) \leqslant a_r\sigma^r$ for all $\sigma>0$, where $\varrho_X$ is the modulus of asymptotic uniform smoothness of $X$. If $X$ is a Banach space such that $I_X\in \mathfrak{N}_{0,p}\setminus \mathfrak{T}_{0,p}$, the constant $C_r$ cannot be taken to be bounded, and so that argument does not yield that the $C_r k^{1/r}$ can be replaced by $Ck^{1/p}$.   To motivate this quantitative improvement, we conclude by modifying a construction of Lindenstrauss \cite{Lindenstrauss} to provide for each $1<p\leqslant \infty$ and any finite dimensional Banach space $F$ some Banach space $Z_0$ such that $I_{Z_0}\in \mathfrak{N}_{0,p}\setminus \mathfrak{T}_{0,p}$ and $Z^{**}_0/Z_0$ is $2$-isomorphic to $ F$.

\section{Definitions}

\subsection{The trees $\Gamma_{\xi,n}$}

Given a set $\Lambda$, we let $\Lambda^{<\nn}=\cup_{n=1}^\infty \Lambda^n$.   We order $\Lambda^{<\nn}$ by letting $s<t$ if $s$ is a proper initial segment of $t$.  We let $\Lambda^{\leqslant n}=\cup_{i=1}^n \Lambda^i$. Given a subset $T$ of $\Lambda^{<\nn}$, we let $MAX(T)$ denote the members of $T$ which are maximal with respect to the initial segment ordering.     Given $n\in\nn$,   $t\in \Lambda^n$, and $1\leqslant i\leqslant n$, we let $t|_i$ denote the initial segment of $t$ having length $i$.     We let $\varnothing$ denote the empty sequence and agree that $\varnothing<s$ for all non-empty sequences $s$. Given two sequences $s,t\in \{\varnothing\}\cup \Lambda^{<\nn}$, we let $s\smallfrown t$ denote the concatenation of $s$ and $t$.  Given a non-empty, finite sequence $s$, we let $s^-$ denote the maximal, proper initial segment of $s$. 

Given a tree $T\subset \Lambda^{<\nn}$, we let $T'=T\setminus MAX(T)$, where $MAX(T)$ denotes the set of maximal members of $T$ with respect to the initial segment ordering.    We then define by transfinite induction $$T^0=T,$$ $$T^{\xi+1}=(T^\xi)',$$ and if $\xi$ is a limit ordinal, $$T^\xi=\bigcap_{\zeta<\xi}T^\zeta.$$  We say $T$ is \emph{well-founded} if there exists an ordinal $\xi$ such that $T^\xi=\varnothing$, and in this case we let $o(T)$ be the minimum such $\xi$.

Given a set $\Lambda$, a \emph{tree on} $\Lambda$ is a subset $T$ of $\Lambda^{<\nn}$ such that if $\varnothing<s\leqslant t\in T$, $s\in T$.   Given a set $D$ and a tree $D$, we let $$T.D=\{(\zeta_i, u_i)_{i=1}^n: n\in\nn, (\zeta_i)_{i=1}^n\in T, u_i\in D\}.$$   Given a Banach space $X$,  a tree $T$, and a directed set $D$, we say a collection $(x_t)_{t\in T.D}\subset X$ is \emph{weakly null} provided that for any $t\in \{\varnothing\}\cup T.D$ and $\zeta$ such that $t\smallfrown (\zeta, u)\in T.D$ for some (equivalently, every) $u\in D$, then $(x_{t\smallfrown (\zeta, u)})_{u\in D}$ is a weakly null net.   

The previous definition of a weakly null tree is necessary to define certain technical notions, such as eventual and inevitable sets. However, in \cite{CD}, a different definition was given which is less cumbersome but which does not afford some of the benefits of the stronger definition above.   We define this notion as well.   Given a tree $T$ and a collection $(x_t)_{t\in T}$ in some Banach space $X$, we say $(x_t)_{t\in T}$ is $\circ$-weakly null if for every $t\in \{\varnothing\}\setminus T'$, $$0\in \overline{\{x_s: s\in T, s^-=t\}}^\text{weak}.$$   

Given a set $\Lambda$ and a subset $T$ of $\Lambda^{<\nn}\cup \{\varnothing\}$, we say $T$ is a \emph{rooted tree} if $\varnothing\in T$ and $T\setminus\{\varnothing\}$ is a tree.  In this case, we let $o(T)=o(T\setminus\{\varnothing\})+1$.   Given an operator $A:X\to Y$, $\ee>0$, a tree $T$, and a collection $(z^*_t)_{t\in T}\subset Y^*$, we say $(z^*_t)_{t\in T}$ is \begin{enumerate}[(i)]\item $(\ee, A^*)$-\emph{large} if $\|A^*z^*_t\|\geqslant \ee$ for all $t\in T$, \item \emph{weak}$^*$-\emph{null} if for any $t\in \{\varnothing\}\cup T'$, $$0\in \overline{\{z^*_s: s\in T, s^-=t\}}^{\text{weak}^*}.$$  \end{enumerate}     Given a rooted tree $T$ and a collection $(y^*_t)_{t\in T}$, we say $(y^*_t)_{t\in T}$ is \begin{enumerate}[(i)]\item $(\ee, A^*)$-\emph{separated} if for any $t\in T\setminus \{\varnothing\}$, $$\|A^*y^*_t-A^*y^*_{t^-}\|\geqslant \ee,$$ \item \emph{weak}$^*$-\emph{closed} if for any $t\in T'$, $$y^*_t\in \overline{\{y^*_s: s\in T, s^-=t\}}^{\text{weak}^*}.$$ \end{enumerate}

If $\zeta $ is an ordinal and $t=(\zeta_i)_{i=1}^k$ is a sequence of ordinals, we let $\zeta+t=(\zeta+\zeta_i)_{i=1}^k$.  If $G$ is a collection of sequences of ordinals, we let $\zeta+G=\{\zeta+t: t\in G\}$.   

We now define some important trees.  For each ordinal $\xi$ and each $n\in\nn$, we will define a tree $\Gamma_{\xi,n}$ which consists of strictly decreasing sequences of ordinals in the interval $[0, \omega^\xi n)$. 

We let $\Gamma_{0,1}=\{(0)\}$, a tree with a single member.

Next, assume that for a limit ordinal $\xi$, $\Gamma_{\zeta+1, 1}$ has been defined for every $\zeta<\xi$. We then let $$\Gamma_{\xi,1}=\bigcup_{\zeta<\xi} (\omega^\zeta+\Gamma_{\zeta+1, 1}).$$   We note that $\omega^\zeta+\Gamma_{\zeta+1, 1}$ is canonically order isomorphic to $\Gamma_{\zeta+1, 1}$ via the map $\Gamma_{\zeta+1, 1}\ni t\mapsto \omega^\zeta+t\in \omega^\zeta+\Gamma_{\zeta+1, 1}$.  We also note that since $\Gamma_{\zeta+1, 1}$ consists of sequences of ordinals in the  interval $[0, \omega^{\zeta+1})$,  $\omega^\zeta+\Gamma_{\zeta+1, 1}$ consists of sequences of ordinals in the interval $[\omega^\zeta+0, \omega^\zeta+\omega^{\zeta+1})=[\omega^\zeta, \omega^{\zeta+1})$.  From this it follows that the union $\cup_{\zeta<\xi}\omega^\zeta+\Gamma_{\zeta+1, 1}$ is a totally incomparable union. Thus we think of $\Gamma_{\xi,1}$ as essentially just a totally incomparable union of the previous trees.    The purpose of using $\omega^\zeta+\Gamma_{\zeta+1, 1}$ in place of $\Gamma_{\zeta+1, 1}$ is to make the union totally incomparable.

Now suppose that for an ordinal $\xi$ and each $n\in\nn$, a tree $\Gamma_{\xi,n}$ has been defined.  Furthermore, suppose that for each $n\in\nn$, the first member of any sequence in $\Gamma_{\xi,n}$ is an ordinal in $[\omega^\xi (n-1), \omega^\xi n)$.  We then let $$\Gamma_{\xi+1, 1}=\bigcup_{n=1}^\infty \Gamma_{\xi,n}.$$  We note that this is a totally incomparable union.

Finally, suppose that for some ordinal $\xi$ and some $n\in\nn$, $\Gamma_{\xi,1}$ and $\Gamma_{\xi,n}$ have been defined.  We then let $$\Gamma_{\xi,n+1}= \{\omega^\xi n+s: s\in \Gamma_{\xi,1}\}\cup \{(\omega^\xi n+t)\smallfrown u: t\in MAX(\Gamma_{\xi,1}), u\in \Gamma_{\xi,n}\}.$$   Roughly, this is the tree obtained by taking an order isomoprhic copy of $\Gamma_{\xi,1}$ (this copy is the subset $\{\omega^\xi n+s: s\in \Gamma_{\xi,1}\}$ of $\Gamma_{\xi,n+1}$) and placing  ``above'' each maximal member of this set an order isomorphic copy of the tree $\Gamma_{\xi,n}$.

Heuristically, the tree $\Gamma_{\xi,n}$ consists of $n$ levels, each level of which consists of order isomorphic copies of $\Gamma_{\xi,1}$.  More precisely, each member $t$ of $\Gamma_{\xi,n}$ can be uniquely written as a concatenation $$t=(\omega^\xi(n-1)+t_1)\smallfrown (\omega^\xi(n-2)+t_2)\smallfrown \ldots \smallfrown (\omega^\xi (n-i)+t_i),$$ where $1\leqslant i\leqslant n$, $t_i\in \Gamma_{\xi,1}$ for each $1\leqslant i\leqslant m$, and $t_i\in MAX(\Gamma_{\xi,1})$ for each $1\leqslant i<m$.  We refer to this representation of $t$ as a concatenation as the \emph{canonical form} of $t$.   For $1\leqslant i\leqslant n$, we let $\Lambda_{\xi,n,i}$ denote those $t\in\Gamma_{\xi,n}$ such that if $(\omega^\xi(n-1)+t_1)\smallfrown \ldots \smallfrown (\omega^\xi(n-m)+t_m)$ is the canonical form of $t$, $m=i$.    We refer to the sets $\Lambda_{\xi,n,1}$, $\ldots$, $\Lambda_{\xi,n,n}$ as the \emph{levels} of $\Gamma_{\xi,n}$.    We observe that $\Lambda_{\xi,n,1}$ is canonically identifiable with $\Gamma_{\xi,1}$ via $$\Gamma_{\xi,1}\ni t \mapsto \omega^\xi(n-1)+t\in \Lambda_{\xi,n,1}.$$   Furthermore, if  $t\in MAX(\Lambda_{\xi,n,1})$, we may write $t=(\omega^\xi(n-1)+t_1)$ for some $t_1\in MAX(\Gamma_{\xi,1})$.   Then  if $n>1$, the set $$C_t:=\{s\in \Gamma_{\xi,n}: t<s\}$$ is canonically identifiable with $\Gamma_{\xi,n-1}$ via a canonical identification.    Indeed, if $t<s$, then the canonical form of $s$ must be $$(\omega^\xi (n-1)+t_1)\smallfrown (\omega^\xi(n-2)+t_2)\smallfrown \ldots (\omega^\xi(n-i)+t_i)$$ for some $t_2, \ldots, t_i$, $i>1$, and we may define the map $j_t:C_t\to \Gamma_{\xi, n-1}$ by $$j_t(s)= (\omega^\xi(n-2)+t_2)\smallfrown \ldots (\omega^\xi(n-i)+t_i).$$  Also, if $1\leqslant i<n$ and $t\in MAX(\Lambda_{\xi,n,i})$, then $E_t:=\{s\in \Lambda_{\xi,n,i+1}: t<s\}$ is canonically identifiable with $\Gamma_{\xi,1}$ via the map $$\Gamma_{\xi,1}\ni u\mapsto t\smallfrown (\omega^\xi(n-i-1)+u)\in E_t.$$    We will use these canonical identifications throughout.

We will also be interested in constructing a tree which consists of infinitely many levels, each level of which consists of order isomorphic copies of $\Gamma_{\xi,1}$.  To serve this purpose, we define the tree $\Gamma_{\xi, \infty}$ as the set of all sequences $t$ which admit a representation of the form $$t=t_1\smallfrown (\omega^\xi+t_2)\smallfrown \ldots \smallfrown (\omega^\xi (k-1)+t_k),$$ where $t_i\in \Gamma_{\xi,1}$ for all $1\leqslant i\leqslant k$ and $t_i\in MAX(\Gamma_{\xi,1})$ for each $1\leqslant i<k$.  Note that any such representation is unique.   We call this representation of $t$ the canonical form of $t$, and define $\Lambda_{\xi, \infty, i}$ to be the set of all $t\in \Gamma_{\xi, \infty}$ such that if $t_1\smallfrown (\omega^\xi +t_2)\smallfrown \ldots \smallfrown (\omega^\xi(k-1)+t_k)$ is the canonical form of $t$,  then $k=i$.  Similarly to the previous paragraph, $\Lambda_{\xi,\infty, 1}$ is canonically identifiable with (and actually equal to) $\Gamma_{\xi,1}$, and for each $i\in\nn$ and $t\in MAX(\Lambda_{\xi, \infty, i})$, $\{s\in \Lambda_{\xi, \infty, i+1}: t<s\}$ is canonically identifiable with $\Gamma_{\xi,1}$.

We also define functions $\mathbb{P}_{\xi,n}:\Gamma_{\xi,n}\to [0,1]$ and $\mathbb{P}_{\xi, \infty}:\Gamma_{\xi, \infty}\to [0,1]$.      We let $\mathbb{P}_{0,1}((0))=1$.

Suppose $\xi$ is a limit ordinal and $\mathbb{P}_{\zeta+1, 1}$ has been defined for each $\zeta<\xi$. For each $\zeta<\xi$,  we let $j_\zeta:\omega^\zeta+\Gamma_{\zeta+1, 1}\to \Gamma_{\zeta+1,1}$ be the canonical identification $j_\zeta(\omega^\zeta+t)=t$ described above.  We then let $\mathbb{P}_{\xi,1}|_{\omega^\zeta+\Gamma_{\zeta+1,1}}=\mathbb{P}_{\zeta+1, 1}\circ j_\zeta$.

We define $\mathbb{P}_{\xi+1, 1}$ by letting $\mathbb{P}_{\xi+1, 1}|_{\Gamma_{\xi,n}}=n^{-1} \mathbb{P}_{\xi, n}$.  

Assume that for some ordinal $\xi$ and $n\in\nn$, $\mathbb{P}_{\xi,n}$ and $\mathbb{P}_{\xi,1}$ have been defined.    We then define $\mathbb{P}_{\xi,n+1}$ by letting $$\mathbb{P}_{\xi, n+1}(\omega^\xi n+s)=\mathbb{P}_{\xi,1}(s)$$ for $s\in \Gamma_{\xi,1}$ and $$\mathbb{P}_{\xi, n+1}((\omega^\xi n+t)\smallfrown u)=\mathbb{P}_{\xi,n}(u)$$ for $t\in MAX(\Gamma_{\xi,1})$ and $u\in \Gamma_{\xi,n}$.   

Now if $t\in \Gamma_{\xi, \infty}$ and $$t=t_1\smallfrown (\omega^\xi+t_2)\smallfrown \ldots \smallfrown (\omega^\xi(i-1)+t_i)$$ is the canonical form  of $t$, let $\mathbb{P}_{\xi, \infty}(t)=\mathbb{P}_{\xi,1}(t_i)$.

We note that the functions $\mathbb{P}_{\xi,n}$ and $\mathbb{P}_{\xi, \infty}$ respect all of the canonical identifications described above.  For example, if $t\overset{j}{\mapsto} \omega^\xi(n-1)+t$ is the canonical identification of $\Gamma_{\xi,1}$ with $\Lambda_{\xi,n,1}$, then $$\mathbb{P}_{\xi,n}|_{\Lambda_{\xi,n,1}}=\mathbb{P}_{\xi,1}\circ j.$$   We will use this fact often.

By an abuse of notation, for a directed set $D$, we also use the symbol $\mathbb{P}_{\xi,n}$ to denote the function from $\Gamma_{\xi,n}.D$ to $[0,1]$ and $\mathbb{P}_{\xi, \infty}$ to denote the function from $\Gamma_{\xi, \infty}.D$ into $[0,1]$ given by $$\mathbb{P}_{\xi,n}((\zeta_i, u_i)_{i=1}^k)=\mathbb{P}_{\xi,n}((\zeta_i)_{i=1}^k),$$ where $(\zeta_i, u_i)_{i=1}^k$ is such that $(\zeta_i)_{i=1}^k\in \Gamma_{\xi,n}$ and $u_i\in D$ for all $1\leqslant i\leqslant k$.   We define $$\mathbb{P}_{\xi, \infty}((\zeta_i, u_i)_{i=1}^k)=\mathbb{P}_{\xi, \infty}((\zeta_i)_{i=1}^k).$$   We also define the levels of $\Gamma_{\xi,n}.D$ and $\Gamma_{\xi,\infty}.D$ in the obvious way.  We let $\Lambda_{\xi,n,i}.D$ denote the set of those $(\zeta_i, u_i)_{i=1}^k\in \Gamma_{\xi,n}.D$ such that $(\zeta_i)_{i=1}^k\in \Lambda_{\xi,n,i}$ and $\Lambda_{\xi,\infty, i}.D$ denote the set of those $(\zeta_i, u_i)_{i=1}^k\in \Gamma_{\xi,\infty}.D$ such that $(\zeta_i)_{i=1}^k\in \Lambda_{\xi, \infty, i}$.   We note that the analogous canonical identifications hold for subsets of $\Gamma_{\xi,n}.D$ as they do for $\Gamma_{\xi,n}$.  For example, $\Gamma_{\xi,1}.D$ is identifiable with $\Lambda_{\xi,n,1}.D$ via $$\Gamma_{\xi,1}.D\ni (\zeta_i, u_i)_{i=1}^k\mapsto (\omega^\xi (n-1)+\zeta_i, u_i)_{i=1}^k.$$   Furthermore, the functions $\mathbb{P}_{\xi,n}$ and $\mathbb{P}_{\xi,\infty}$ respect these canonical identifications.

In \cite{C4}, given a directed set $D$, an ordinal $\xi$, and $n\in\nn$, a subset $\Omega_{\xi,n}.D$ of the power set of $MAX(\Gamma_{\xi,n}.D)$ was defined.  There the notation $\Omega_{\xi,n}$ was used in place of $\Omega_{\xi,n}.D$, because $D$ was held fixed and unreferenced.      The members of set this set were called \emph{cofinal}.  The rough idea is that the confinal sets are those which are non-negligible.   In \cite{C4}, a subset $\mathcal{E}$ of $MAX(\Gamma_{\xi,n}.D)$ was called \emph{eventual} if $MAX(\Gamma_{\xi,n}.D)\setminus \mathcal{E}$ is not cofinal.  It was shown there that if $\mathcal{E}_1, \ldots, \mathcal{E}_k\subset MAX(\Gamma_{\xi,n}.D)$ are eventual, so is $\cap_{i=1}^k \mathcal{E}_i$, which is a fact we will need in the sequel.  It was also shown there that any superset of an eventual set is eventual.    

Also in \cite{C4}, a subset $I$ of $\cup_{i=1}^\infty MAX(\Lambda_{\xi,\infty, i}.D)$ is called \emph{inevitable} if \begin{enumerate}[(i)]\item $\{t\in MAX(\Lambda_{\xi, \infty, 1}.D): t\in I\}$ is an eventual subset of $MAX(\Gamma_{\xi,1}.D)$, and \item for any $n\in\nn$ and $t\in I\cap MAX(\Lambda_{\xi,\infty, n}.D)$, $j_t(\{s\in MAX(\Lambda_{\xi, \infty, n+1}.D): t<s\})$ is an eventual subset of $MAX(\Gamma_{\xi,1}.D)$, where  $j_t:\{s\in \Lambda_{\xi, \infty, n+1}.D: t<s\}\to \Gamma_{\xi,1}.D$ is the canonical identification.     \end{enumerate}

We say $B\subset \cup_{i=1}^\infty MAX(\Lambda_{\xi, \infty, i}.D)$ is \emph{big} if it has an inevitable subset.

\subsection{Moduli and ideals}

For an ordinal $\xi$, $n\in\nn$, and a Banach space $X$,  we let $D_X$ denote the directed set of all weakly open sets in $X$ which contain $0$, directed by reverse inclusion, and let $\mathcal{B}^X_{\xi,n}$ denote the set of all weakly null collections $(x_t)_{t\in \Gamma_{\xi,n}.D_X}\subset B_X$.    For an operator $A:X\to Y$, $s\geqslant 0$, and $y\in Y$, we let $$\varrho_\xi(s;y;A)= \sup\Big\{ \inf  \bigl\{\|y+sA\sum_{\varnothing<s\leqslant t} \mathbb{P}_{\xi,1}(s) x_s\|-1: t\in MAX(\Gamma_{\xi,1}.D)\}: (x_t)_{t\in \Gamma_{\xi,1}.D}\in \mathcal{B}_{\xi,1}^X\Bigr\}.$$ If $Y=\{0\}$, we let $\varrho_\xi(s;A)=0$ for all $s\geqslant 0$, and otherwise we let $$\varrho_\xi(s;A)= \sup \{\varrho_\xi(s;y;A): y\in B_Y\}.$$   

For $y^*\in Y^*$, $\tau\geqslant 0$, we let $$\delta^{\text{weak}^*}_\xi(\tau;y^*;A)=\inf\Bigl\{ \sup_{t\in T}\|y^*+\tau\sum_{\varnothing<s\leqslant t} z^*_s\|-1: o(T)=\omega^\xi, (z^*_t)_{t\in T}\subset Y^*\text{weak$^*$-null and\ }(1,A^*)-\text{separated}\Bigr\}.$$   If $Y=\{0\}$, we let $\delta^{\text{weak}^*}_\xi(\tau;A)=\infty$ for all $\tau>0$, $\delta^{\text{weak}^*}_\xi(0;A)=0$. Otherwise for each $\tau\geqslant 0$, we let $\delta^{\text{weak}^*}_\xi(\tau;A)=\inf_{y^*\in S_{Y^*}} \delta^{\text{weak}^*}_\xi(\tau;y^*;A)$.      We note that in \cite{CD}, the definition of $\delta^{\text{weak}^*}_\xi(\tau;A)$ involving taking the infimum of $\delta^{\text{weak}^*}_\xi(\tau;y^*;A)$ over all $y^*\in Y^*$ with $\|y^*\|\geqslant 1$ rather than simply $y^*\in S_{Y^*}$.  However, a close examination of the proofs from \cite{CD} shows that all results hold with either definition.

For $1<p<\infty$, an ordinal $\xi$, and an operator $A:X\to Y$, we say $A$ is \begin{enumerate}[(i)]\item $\xi$-\emph{asymptotically uniformly smooth} ($\xi$-\emph{AUS} for short) if $\inf_{s>0} \varrho_\xi(s;A)/s=0$,   \item $\xi$-$p$-\emph{asymptotically uniformly smooth} ($\xi$-$p$-\emph{AUS}) if $\sup_{s>0} \varrho_\xi(s;A)/s^p<\infty$, \item $\xi$-\emph{asymptotically uniformly flat} ($\xi$-\emph{AUF}) (or $\xi$-$\infty$-\emph{asymptotically uniformly smooth}, $\xi$-$\infty$-\emph{AUS}) if there exists $s>0$ such that $\varrho_\xi(s;A)=0$. \end{enumerate}  

We note that each of these properties is retained under passing to equivalent norms on $X$, but not necessarily by passing to equivalent norms on $Y$.

We say $A:X\to Y$ is $\xi$-\emph{asymptotically uniformly smoothable} if there exists an equivalent norm $|\cdot|$ on $Y$ such that $A:X\to (Y, |\cdot|)$ is $\xi$-AUS.   We define $\xi$-$p$-\emph{asymptotically uniformly smoothable} and $\xi$-\emph{asymptotically uniformly flattenable} similarly.    

For an operator $A:X\to Y$ and an ordinal $\xi$, we define $\|\cdot\|_{\xi,A}$ on $c_{00}$ by letting $$\|\sum_{i=1}^\infty a_ie_i\|_{\xi,A}= \sup\Bigl\{ \inf \{\|A\sum_{i=1}^n a_i\sum_{\Lambda_{\xi,n,i}.D_X\ni s\leqslant t}  \mathbb{P}_{\xi,n}(s) x_s\|: t\in MAX(\Gamma_{\xi,n}.D)\}: n\in\nn, (x_t)_{t\in \Gamma_{\xi,n}.D}\in \mathcal{B}^X_{\xi,n}\Bigr\}.$$

\begin{rem}\upshape It is implicitly contained in the proof of \cite[Theorem $4.3$]{C5} that $\|\sum_{i=1}^n a_ie_i\|_{\xi,A}$ is the infimum of those $C>0$ such that for every $(x_t)_{t\in \Gamma_{\xi,n}.D_X}\in \mathcal{B}^X_{\xi,n}$, $$\{t\in MAX(\Gamma_{\xi,n}.D_X): \|A\sum_{i=1}^n a_i \sum_{\Lambda_{\xi,n,i}.D_X\ni s\leqslant t} \mathbb{P}_{\xi,n}(s)x_s\|\leqslant C\}$$ is eventual.     

\end{rem}

\begin{proposition} The function $\|\cdot\|_{\xi,A}$ is a seminorm on $c_{00}$ making the canonical basis $1$-unconditional.     

\end{proposition}

\begin{proof} We need only show the triangle inequality and $1$-unconditionality.    Fix $(a_i)_{i=1}^\infty, (b_i)_{i=1}^\infty\in c_{00}$.   Fix $n\in\nn$ such that $a_i=b_i=0$ for all $i>n$.  Fix $(x_t)_{t\in \Gamma_{\xi,n}.D_X}\in \mathcal{B}^X_{\xi,n}$, $\alpha>\|\sum_{i=1}^n a_ie_i\|_{\xi,A}$, $\beta>\|\sum_{i=1}^n b_ie_i\|_{\xi,A}$. Then by our previous remark, $$\mathcal{E}:=\{t\in MAX(\Gamma_{\xi,n}.D_X): \|A\sum_{i=1}^n a_i \sum_{\Lambda_{\xi,n,i}.D_X\ni s\leqslant t}\mathbb{P}_{\xi,n}(s) x_s\|\leqslant \alpha\}$$ and $$\mathcal{F}:=\{t\in MAX(\Gamma_{\xi,n}.D_X): \|A\sum_{i=1}^n b_i \sum_{\Lambda_{\xi,n,i}.D_X\ni s\leqslant t} \mathbb{P}_{\xi,n}(s)x_s\|\leqslant \beta\}$$ are eventual, as is $\mathcal{E}\cap \mathcal{F}$.   But since $$\mathcal{E}\cap \mathcal{F} \subset \{t\in MAX(\Gamma_{\xi,n}.D_X): \|A\sum_{i=1}^n (a_i+b_i) \sum_{\Lambda_{\xi,n,i}.D_X\ni s\leqslant t} x_s\|\leqslant \alpha+\beta\},$$ another appeal to our previous remark yields the triangle inequality.  Here we are using the previously recalled fact that finite intersections of eventual sets are eventual, as are supersets of eventual sets.

Now if we fix $n\in\nn$, $(a_i)_{i=1}^\infty\in c_{00}$, $(x_t)_{t\in \Gamma_{\xi,n}.D_X}\in \mathcal{B}^X_{\xi,n}$, and unimodular scalars $\ee_1, \ldots, \ee_n$, $$\inf_{t\in MAX(\Gamma_{\xi,n}.D_X)} \|A\sum_{i=1}^n a_i \sum_{\Lambda_{\xi,n,i}.D_X\ni s\leqslant t} x_s\| = \inf_{t\in MAX(\Gamma_{\xi,n}.D_X)} \|A\sum_{i=1}^n \ee_i a_i \sum_{\Lambda_{\xi,n,i}.D_X\ni s\leqslant t} \ee_i^{-1}x_s\|\leqslant \|\sum_{i=1}^n \ee_ia_i e_i\|_{\xi,A}.$$  For the last inequality, we are using that $(\ee_{\lambda(t)} x_t)_{t\in \Gamma_{\xi,n}.D_X}\in \mathcal{B}^X_{\xi,n}$, where $\lambda(t)$ is the $1\leqslant i\leqslant n$ such that $t\in \Lambda_{\xi,n,i}.D_X$.     This yields that $\|\sum_{i=1}^n a_ie_i\|_{\xi,A}=\|\sum_{i=1}^n \ee_i a_i e_i\|_{\xi,A}$.

\end{proof}

For $1\leqslant p\leqslant \infty$, an ordinal $\xi$,   an operator $A:X\to Y$, and $n\in\nn$, we let $$\alpha_{\xi,p,n}(A)=\sup \{\|\sum_{i=1}^n a_ie_i\|_{\xi,A}: (a_i)_{i=1}^n\in B_{\ell_p^n}\},$$ $$\alpha_{\xi,p}(A)=\sup_{n\in\nn} \alpha_{\xi,p,n}(A),$$ and  $$\theta_{\xi,n}(A)= \|n^{-1}\sum_{i=1}^n e_i\|_{\xi,A}.$$

If $X$ is a Banach space, $K\subset X^*$ is weak$^*$-compact, and $\ee>0$, we let $s_\ee(K)$ denote the set of $x^*\in K$ such that for every weak$^*$-neighborhood $V$ of $x^*$, $\text{diam}(K\cap V)>\ee$.  We define $s_0(K)=K$.  We recursively define $$s_\ee^0(K)=K,$$ $$s^{\xi+1}_\ee(K)=s_\ee(s^\xi_\ee(K)),$$ and if $\xi$ is a limit ordinal, $$s^\xi_\ee(K)=\bigcap_{\zeta<\xi} s^\zeta_\ee(K).$$    If there exists an ordinal $\xi$ such that $s_\ee^\xi(K)=\varnothing$, we let $Sz(K, \ee)$ be the minimum ordinal $\xi$ such that $s_\ee^\xi(K)=\varnothing$, and otherwise we write $Sz(K, \ee)=\infty$.  For an ordinal $\xi$, if there exists an ordinal $\zeta$ such that $s^\zeta_\ee(K)=\varnothing$, then we let $Sz_\xi(K, \ee)$ be the minimum ordinal $\zeta$ such that $s^{\omega^\xi \zeta}_\ee(K)=\varnothing$, and otherwise we write $Sz_\xi(K, \ee)=\varnothing$.   If $Sz(K, \ee)=\infty$ for some $\ee>0$, then we write $Sz(K)=\infty$.   If $Sz(K, \ee)$ is an ordinal for every $\ee>0$, we let $Sz(K)=\sup_{\ee>0} Sz(K, \ee)$.    For an ordinal $\xi$, we let $$c_{\xi,\ee}(K)= \overline{\text{co}(s_\ee^{\omega^\xi}(K))}^{\text{weak}^*},$$ and we define again by transfinite induction $$c_{\xi,\ee}^0(K)=K,$$ $$c_{\xi, \ee}^{\zeta+1}(K)=c_{\xi,\ee}(c_{\xi,\ee}^\zeta(K)),$$ and if $\zeta$ is a limit ordinal, $$c_{\xi,\ee}^\zeta(K)=\bigcap_{\gamma<\zeta} c_{\xi, \ee}^\gamma(K).$$   If there exists an ordinal $\zeta$ such that $c_{\xi, \ee}^\zeta(K)=\varnothing$, we let $Cz_\xi(K, \ee)$ be the minimum ordinal $\zeta$ such that $c^\zeta_{\xi, \ee}(K)=\varnothing$, and otherwise we write $Cz_\xi(K, \ee)=\infty$. We note that by weak$^*$-compactness, $Sz(K, \ee)$, $Sz_\xi(K, \ee)$, and $Cz_\xi(K, \ee)$ cannot be a limit ordinal.    If $A:X\to Y$ is an operator, we define $Sz(A, \ee)=Sz(A^*B_{Y^*}, \ee)$, $Sz_\xi(A, \ee)=Sz_\xi(A^*B_{Y^*}, \ee)$, etc.    If $X$ is a Banach space, we let $Sz(X, \ee)=Sz(I_X, \ee)$, etc.   We define $$\textsf{p}_\xi(A)= \underset{\ee\to 0^+}{\lim\sup} \frac{\log Sz_\xi(A, \ee)}{|\log(\ee)|},$$ where we agree to the convention that $\log Sz_\xi(A, \ee)=\infty$ provided that $Sz_\xi(A, \ee)$ is infinite.

  We collect here some important facts regarding these quantities and classes.     In what follows, we let $\mathfrak{D}_\xi$ denote the class of operators $A$ with $Sz(A)\leqslant \omega^\xi$.

\begin{theorem}\cite{Brooker} For every ordinal $\xi$, $\mathfrak{D}_\xi$ is a closed operator ideal.

\end{theorem}

\begin{theorem}\cite{C5} Let $\xi$ be an ordinal. \begin{enumerate}[(i)]\item $A\in \mathfrak{D}_\xi$  if and only if for any directed set $D$ and any weakly null $(x_t)_{t\in \Gamma_{\xi,1}.D}\subset B_X$, $$\inf\|\sum_{s\leqslant t} \mathbb{P}_{\xi,1}(s)x_s\|=0$$ if and only if $\theta_{\xi,1}(A)=0$.    \item $\theta_{\xi,1}(A)\leqslant \|A\|$. \item If there exist a directed set $D$ and a weakly null collection $(x_t)_{t\in \Gamma_{\xi,1}.D}\subset B_X$ such that $$\inf_{t\in MAX(\Gamma_{\xi,1}.D)} \|A\sum_{s\leqslant t}\mathbb{P}_{\xi,1}(s)x_s\| \geqslant \ee,$$ then  $\theta_{\xi,1}(A)\geqslant \ee$.    \item The map $\omega^\xi n\mapsto \theta_{\xi,n}(A)$ is non-increasing and continuous from the class $\{\omega^\xi n:\xi \in \textbf{\emph{Ord}}, n\in\nn\}$ into $[0,\|A\|]$.       \end{enumerate}

\label{wh}

\end{theorem}

\section{Discussion of the properties}

Fix $1< p\leqslant \infty$ and an ordinal $\xi$. Let $1/p+1/q=1$.   We  let  \begin{enumerate}[(i)]\item $\mathfrak{T}_{\xi,p}$ denote the class of $\xi$-$p$-AUS-able operators, \item $\mathfrak{A}_{\xi,p}$ denote the class of operators $A$ such that $\alpha_{\xi,p}(A)<\infty$, \item $\mathfrak{N}_{\xi,p}$ denote the class of operators $A$ such that $\sup_n \theta_{\xi,n}(A)n^{1/q}<\infty$, \item $\mathfrak{P}_{\xi,p}$ denote the class of operators $A$ such that $\textsf{p}_\xi(A)\leqslant q$. \end{enumerate}   We let $\textsf{T}_{\xi,p}$ denote the class of Banach spaces $X$ such that $I_X\in \mathfrak{T}_{\xi,p}$, and $\textsf{A}_{\xi,p}$, $\textsf{N}_{\xi,p}$, and $\textsf{P}_{\xi,p}$ are defined similarly.

\begin{rem}\upshape It is the main result of \cite{C4} that $A\in \mathfrak{T}_{\xi,p}$ if and only if there exists a constant $C$ such that for any  weakly null collection $(x_t)_{t\in \Gamma_{\xi,\infty}.D_X}\subset B_X$, $$\bigcup_{n=1}^\infty \bigl\{t\in MAX(\Lambda_{\xi,\infty, n}.D_X): (\forall (a_i)_{i=1}^n\in B_{\ell_p^n})(\|A\sum_{i=1}^n a_i \sum_{\Lambda_{\xi, \infty, n}\ni s} \mathbb{P}_{\xi, \infty}(s)x_s\|\leqslant C)\bigr\}$$   is big (that is, has an inevitable subset).    We let $T_{\xi,p}(A)$ denote the infimum of $C$ having the previous property and let $\mathfrak{t}_{\xi,p}(A)=\|A\|+T_{\xi,p}(A)$.   It was shown in \cite{C4} that $(\mathfrak{T}_{\xi,p}, \mathfrak{t}_{\xi,p})$ is a Banach ideal.

It follows from the definition of $\mathfrak{A}_{\xi,p}$ that $A\in \mathfrak{A}_{\xi,p}$ if and only if the formal identity $I:(c_{00}, \|\cdot\|_{\ell_p})\to (c_{00}, \|\cdot\|_{\xi,A})$ is bounded.     In \cite{C5}, this was shown to be equivalent to: There exists $C>0$ such that for any $n\in \nn$ and $\ee_1, \ldots, \ee_n>0$ such that if $s^{\omega^\xi}_{\ee_1} \ldots s^{\omega^\xi}_{\ee_n}(A^*B_{Y^*})\neq \varnothing,$ $\sum_{i=1}^n \ee_i^q \leqslant C^q$, where $1/p+1/q=1$.    In this case, we say $A$ has $C$-$\xi$-$q$-\emph{summable Szlenk index}.    Moreover, there exist constants $a,b>0$ (independent of $\xi$ and $p$) such that if $C_1$ is the infimum of those $C>0$ such that $A$ has $C$-$\xi$-$q$-summable Szlenk index and if $C_2$ is the norm of the formal inclusion $I:(c_{00}, \|\cdot\|_{\ell_p})\to (c_{00}, \|\cdot\|_{\xi,A})$, $aC_1\leqslant C_2\leqslant bC_1$.   Furthermore, it was shown there that with $\mathfrak{a}_{\xi,p}(A)=\|A\|+\alpha_{\xi,p}(A)$, $(\mathfrak{A}_{\xi,p}, \mathfrak{a}_{\xi,p})$ is a Banach ideal.

It was shown in \cite{TAMS} that $\mathfrak{P}_{\xi,p}=\cap_{1<r<p} \mathfrak{T}_{\xi,r}$, and $\mathfrak{P}_{\xi,p}$ is therefore an operator ideal, since it is the intersection of ideals.

\end{rem}

We collect the following information regarding what is already known about the relationships among these classes.

\begin{theorem} For any ordinal $\xi$ and any $1<p\leqslant \infty$, $$\mathfrak{D}_\xi\subset \mathfrak{T}_{\xi,p}\subset \mathfrak{A}_{\xi,p}\subset \mathfrak{N}_{\xi,p}\subset \mathfrak{P}_{\xi,p}=\bigcap_{1<r<p} \mathfrak{T}_{\xi,r}\subset \mathfrak{D}_{\xi+1}.$$

\end{theorem}

It follows from Theorem \ref{wh}$(i)$ that $T_{\xi,p}(A)=0$ if $Sz(A)\leqslant \omega^\xi$.  Indeed, the condition $Sz(A)\leqslant \omega^\xi$ yields that for any $\ee>0$ and any weakly null $(x_t)_{t\in \Gamma_{\xi,1}.D_X}\in \mathcal{B}^X_{\xi,1}$, $$\{t\in MAX(\Gamma_{\xi,1}.D_X): \|A\sum_{s\leqslant t} \mathbb{P}_{\xi,1}(s)x_s\|\leqslant \ee\}$$ is eventual.    From this it follows that if  $\ee>0$ and $(x_t)_{t\in \Gamma_{\xi, \infty}.D_X}\subset B_X$ is weakly null,  then $$\bigcup_{n=1}^\infty \Bigl\{t\in MAX(\Lambda_{\xi,\infty,n}.D_X): (\forall 1\leqslant i\leqslant n) \Bigl(\|A\sum_{t_{i-1}<s\leqslant t_i} \mathbb{P}_{\xi,\infty}(s)x_s\| \leqslant \ee 2^{-i}\Bigr)\}$$ is inevitable. Here, as usual, $\varnothing=t_0<\ldots <t_n=t$ are such that $t_i\in MAX(\Lambda_{\xi,\infty, i}.D_X)$.        Since this set is a subset of $$\bigcup_{n=1}^\infty \Bigl\{t\in MAX(\Lambda_{\xi,\infty,n}.D_X): (\forall (a_i)_{i=1}^n\in B_{\ell_p^n}) \Bigl(\|A\sum_{i=1}^n a_i\sum_{t_{i-1}<s\leqslant t_i} \mathbb{P}_{\xi,\infty}(s)x_s\| \leqslant \ee \Bigr)\},$$ the latter set is big.    This yields that if $Sz(A)\leqslant \omega^\xi$, $T_{\xi,p}(A)=0$ and $A\in \mathfrak{T}_{\xi,p}$.

Now we note that $\alpha_{\xi,p}(A)\leqslant T_{\xi,p}(A)$.  Indeed, if $(x_t)_{t\in \Gamma_{\xi,n}.D_X}\in \mathcal{B}^X_{\xi,n}$, we can identify $\Gamma_{\xi,n}.D_X$ with the first $n$ levels of $\Gamma_{\xi,\infty}.D_X$ in the canonical way and fill out the rest of a weakly null collection $(x_t)_{t\in \Gamma_{\xi,\infty}.D_X}\subset B_X$ by letting $x_t=0$ for all $t\in \cup_{i=n+1}^\infty \Lambda_{\xi, \infty, i}.D_X$.   For any $\ee>0$, there exist $\varnothing=t_0<t_1<\ldots$ such that, with $y_i=\sum_{t_{i-1}<s\leqslant t_i} \mathbb{P}_{\xi,\infty}(s)x_s$, $(y_i)_{i=1}^\infty$ is $T_{\xi,p}(A)+\ee$-dominated by the $\ell_p$ (resp. $c_0$) basis.    Since the canonical identification of $\Gamma_{\xi,n}.D_X$ with the first $n$ levels of $\Gamma_{\xi, \infty}.D_X$ also preserves the values of the coefficients $\mathbb{P}_{\xi,n}$ and $\mathbb{P}_{\xi,\infty}$, for any scalar sequence $(a_i)_{i=1}^n\in B_{\ell_p^n}$, we have found $t_n \in MAX(\Gamma_{\xi,n}.D_X)$ such that $$\|A\sum_{i=1}^n a_i\sum_{t_{i-1}<s\leqslant t_i} \mathbb{P}_{\xi,n}(s)x_s\|\leqslant T_{\xi,p}(A)+\ee.$$  This yields that $\|\sum_{i=1}^n a_i e_i\|_{\xi,A} \leqslant T_{\xi,p}(A)+\ee$.  Since $\ee>0$ was arbitrary, we see that $\alpha_{\xi,p}(A)\leqslant T_{\xi,p}(A)$.    This yields that $\mathfrak{T}_{\xi,p}\subset \mathfrak{A}_{\xi,p,}$.

Next note that   $$\theta_{\xi,n}(A)n^{1/q}=\|\frac{1}{n^{1/p}}\sum_{i=1}^n e_i\|_{\xi,A} \leqslant \alpha_{\xi,p,n}(A),$$ so $\mathfrak{A}_{\xi,p}\subset \mathfrak{N}_{\xi,p}$.    

It is the main renorming theorem of \cite{TAMS} that $\mathfrak{P}_{\xi,p}=\bigcap_{1<r<p} \mathfrak{T}_{\xi,r}$, and it is implicit in the proof of that theorem that $\mathfrak{N}_{\xi,p}\subset \mathfrak{P}_{\xi,p}$, although the class $\mathfrak{N}_{\xi,p}$ was not specifically isolated or given a name there.  The actual method of proof of the main theorem of \cite{TAMS} was to argue in three stages.  One stage was to argue that if $A\in \mathfrak{N}_{\xi,s}$ for some $1<s< \infty$ and $A$ has norm not more than $1$, then there exists a constant $C$ such that for every $\sigma>0$, there exists a $C$-equivalent norm $|\cdot|$ (which depends on $\sigma$) on $Y$ such that $\varrho_\xi(\sigma;A:X\to (Y, |\cdot|))\leqslant C \sigma^s$.   The next stage was to show that for any $1<r<s$, one can ``glue'' these norms together to produce a single, equivalent norm $|\cdot|$ on $Y$ such that $\sup_{\sigma>0} \varrho_\xi(\sigma;A:X\to (Y, |\cdot|))/\sigma^r<\infty$, and therefore $A\in \mathfrak{T}_{\xi,r}$ for any $1<r<s$.    Finally, the last stage was to show that if $A\in \mathfrak{P}_{\xi,p}$, then $A\in \mathfrak{N}_{\xi,s}$ for any $1<s<p$.  Therefore if we have the \emph{a fortiori} fact that $A\in \mathfrak{N}_{\xi,p}$ (instead of $A\in \mathfrak{P}_{\xi,p}$), we can omit the third stage of this argument to deduce that $A\in \bigcap_{1<r<p}\mathfrak{T}_{\xi,r}=\mathfrak{P}_{\xi,p}$.  This yields that  $\mathfrak{N}_{\xi,p}\subset \mathfrak{P}_{\xi,p}$.

In the final section, we discuss the distinctness of these classes.

\begin{proposition} For every ordinal $\xi$ and every $1<p\leqslant \infty$, there exists an ideal norm $\mathfrak{n}_{\xi,p}$ on $\mathfrak{N}_{\xi,p}$ such that $(\mathfrak{N}_{\xi,p}, \mathfrak{n}_{\xi,p})$ is a Banach ideal.

\end{proposition}

\begin{proof} It was shown in \cite{C5} that $\theta_{\xi,n}(A)$ is an ideal norm for all ordinals $\xi$ and $n\in\nn$.    We then let $$N_{\xi,p,n}(A)= n^{1/q} \theta_{\xi,n}(A),$$ $$N_{\xi,p}(A)=\sup_n N_{\xi,p,n}(A),$$ and $$\textsf{n}_{\xi,p}(A)=\|A\|+N_{\xi,p}(A).$$    It easily follows that $\textsf{n}_{\xi,p}$ is an ideal norm and  $\mathfrak{N}_{\xi,p}$ is the class of all $A$ such that $\mathfrak{n}_{\xi,p}(A)<\infty$. It follows from standard arguments that $(\mathfrak{N}_{\xi,p}, \textsf{n}_{\xi,p})$ is a Banach ideal.    More precisely, since $N_{\xi,p,n}:\mathfrak{L}(X,Y)\to [0, \infty)$ is continuous for each $n\in\nn$ with respect to the norm topology on $\mathfrak{L}(X,Y)$, then for any $C\geqslant 0$, $\{A\in \mathfrak{L}(X,Y): N_{\xi,p}(A)\leqslant C\}$ is closed in $\mathfrak{L}(X,Y)$. 

\end{proof}

\begin{proposition} For any ordinal $\xi$, $\mathfrak{A}_{\xi,\infty}=\mathfrak{N}_{\xi,\infty}$.   

\end{proposition}

\begin{proof} By $1$-unconditionality and convexity, $$\alpha_{\xi, \infty}(A)=\sup\{\|\sum_{i=1}^\infty a_ie_i\|_{\xi,A}: (a_i)_{i=1}^\infty\in c_{00}\cap B_{\ell_\infty}\} = \sup \{\|\sum_{i=1}^n e_i\|_{\xi,A}: n\in\nn\}=N_{\xi,\infty}(A).$$   Then $A\in \mathfrak{A}_{\xi,\infty}$ if and only if $\alpha_{\xi, \infty}(A)<\infty$ if and only if $N_{\xi,\infty}(A)<\infty$ if and only if $A\in \mathfrak{N}_{\xi, \infty}$.

\end{proof}

\section{The ideal $\mathfrak{N}_{\xi,p}$}

The majority of this section is a generalization of results from \cite{GKL} from the $\xi=0$ case to the general case, and from the spatial case to the operator case.

Given a Banach space $Y$ and $b\geqslant 1$, let $\mathcal{Y}_b$ denote the set of all norms $|\cdot|$ on $Y$ such that $$b^{-1}B_Y^{|\cdot|}\subset B_Y\subset bB_Y^{|\cdot|}.$$  Given an operator $A:X\to Y$, $b\geqslant 1$,  an ordinal $\xi$, and $\sigma>0$, let $$\phi_\xi^b(\sigma;A)=\inf \{\varrho_\xi(\sigma;A:X\to (Y, |\cdot|)): |\cdot|\in \mathcal{Y}_b\}.$$   For $\tau\geqslant 0$, let $$\psi_\xi^b(\tau;A)=\sup \{\delta^{\text{weak}^*}_\xi(\tau;A:X\to (Y, |\cdot|)): |\cdot|\in \mathcal{Y}_b\}.$$

We note that $\varrho_\xi(\cdot;A)\equiv 0$ if and only for every $b\geqslant 1$, $\phi^b_\xi(\cdot;A)\equiv 0$ if and only if   $\delta^{\text{weak}^*}_\xi(\cdot;A)\equiv \infty$ if and only if for every $b\geqslant 1$, $\psi^b_\xi(\cdot;A)\equiv \infty$ if and only if  $Sz(A)\leqslant \omega^\xi$.   Thus we eliminate this trivial case.

Let $S$ denote the class of functions $f:[0,\infty)\to [0, \infty)$ which are continuous, $f(0)=0$, and the function $f(t)/t$ is non-decreasing on $(0, \infty)$.   A function which has the last of these properties we will refer to as \emph{star shaped}. 

\begin{proposition} Let $\xi$ be an ordinal,  $A:X\to Y$  an operator with $Sz(A)>\omega^\xi$, and $b\geqslant 1$.  Then $\varrho_\xi(\cdot;A), \delta^{\text{\emph{weak}}^*}_\xi(\cdot;A), \phi^b_\xi(\cdot;A), \psi^b_\xi(\cdot;A)\in S$.

\end{proposition}

\begin{proof} We note that since $Sz(A)>\omega^\xi$, there exist $\ee>0$, a tree $T$ with $o(T)=\omega^\xi+1$, and a weak$^*$-closed, $(\ee,A^*)$-separated collection $(y^*_t)_{t\in T}\subset B_{Y^*}$ such that $y^*_\varnothing=0$.  Let $z^*_t=y^*_t-y^*_{t^-}$ for $t\in T\setminus \{\varnothing\}$.    Fix $z^*\in S_{Y^*}$ and note that for any $\tau>0$, $$\sup_{t\in T} \|z^*+\frac{\tau}{\ee} \sum_{s\leqslant t} z^*_s\|\leqslant 1+\frac{\tau}{\ee}<\infty.$$  This yields that $\delta^{\text{weak}^*}_\xi(\tau)<\infty$ for all $\tau>0$.   Similarly, considering $(b^{-1}y_t^*)_{t\in T}\subset B_{Y^*}^{|\cdot|}$ for any $|\cdot|\in \mathcal{Y}_b$ yields that $\psi^b_\xi(\tau;A)\leqslant 1+\frac{\tau b}{\ee }$ for all $b\geqslant 1$ and $\tau>0$.

By definition, each of the four functions is non-negative and vanishing at $0$. Fix $0<\sigma$ and $0<\alpha<1$.    Fix $y\in B_Y$, a tree $T$ with $o(T)=\omega^\xi$, and a $\circ$-weakly null collection $(x_t)_{t\in T}\subset B_X$.  Then if $C=\{x: t\in T, x\in \text{co}(x_s: s\leqslant t)\}$, \begin{align*} \inf_{x\in C} \|y+\alpha \sigma Ax\|-1 & = \inf_{x\in C} \|(1-\alpha) y + \alpha y + \alpha \sigma Ax\| - \alpha -(1-\alpha) \\ & \leqslant (1-\alpha)[\|y\|-1] + \alpha\inf_{x\in C}\|y+\sigma Ax\|-1 \\ & \leqslant \alpha \varrho_\xi(\sigma;A). \end{align*} This yields that $\varrho_\xi(\cdot;A)$ is star shaped.  Since the same estimate holds for any $|\cdot|\in \mathcal{Y}_b$, we deduce that $\phi^b_\xi(\cdot;A)$ is star shaped.    

Now for any $y^*\in S_{Y^*}$, any tree $T$ with $o(T)=\omega^\xi$, and any weak$^*$-null, $(1, A^*)$-separated $(y^*_t)_{t\in T}$, \begin{align*} \delta^{\text{weak}^*}_\xi(\alpha \sigma;A) & \sup_{t\in T} \|y^*+\alpha \sigma\sum_{s\leqslant t}y^*_s\| -1  = \sup_{t\in T} \|(1-\alpha) y^*+\alpha y^*+\alpha \sigma \sum_{s\leqslant t}y^*_s\| - \alpha - (1-\alpha) \\ & \leqslant (1-\alpha)[\|y^*\|-1] +\alpha\Bigl( \sup_{t\in T} \|y^*+\sigma\sum_{s\leqslant t} y^*_s\|-1 \Bigr)\\ & \leqslant \alpha\sup_{t\in T}\|y^*+\sigma \sum_{s\leqslant t}y^*_s\|-1.   \end{align*} Taking the infimum over such $y^*\in S_{Y^*}$, $T$, and $(y^*_t)_{t\in T}$, we can replace the last quantity with $\alpha \delta^{\text{weak}^*}_\xi(\sigma;A)$.  This yields that $\delta^{\text{weak}^*}_\xi(\cdot;A)$ is star shaped, and similar considerations yield that $\psi^b_\xi(\cdot;A)$ is star shaped.

	It is evident that $\varrho_\xi(\cdot;A)$ is Lipschitz continuous with Lipschitz constant not more than $\|A:X\to Y\|$, whence $\phi^b_\xi(\cdot;A)$ is Lipschitz continuous with Lipschitz constant not more than $$\sup \{\|A:X\to (Y, |\cdot|)\|: |\cdot|\in \mathcal{Y}_b\}= b\|A:X\to Y\|.$$

Let $$M=\sup \{\delta^{\text{weak}^*}_\xi(1;A:X\to (Y, |\cdot|)): |\cdot|\in \mathcal{Y}_b\} \leqslant 1+ b/\ee .$$      It follows from the first paragraph that $\delta^{\text{weak}^*}_\xi(\cdot;A)$ and $\psi^b_\xi(\cdot;A)$ are continuous at $0$.   Fix $\tau>0$, $0<\delta<1$, $|\cdot|\in \mathcal{Y}_b$, $y^*\in S_{Y^*}^{|\cdot|}$, a tree $T$ with $o(T)=\omega^\xi$, and a weak$^*$-null, $(1, A^*)$-large collection $(y^*_t)_{t\in T}$ such that $$\sup_{t\in T} |y^*+\tau\sum_{s\leqslant t} y^*_s|\leqslant 1+\delta+\delta^{\text{weak}^*}_\xi(\tau;A:X\to (Y, |\cdot|)).$$  We claim that $$\sup_{t\in T} |\sum_{s\leqslant t}y^*_s|\leqslant 3+\tau(3+M).$$   Indeed, suppose $$\sup_{t\in T} |\sum_{s\leqslant t}y^*_s|> 3+\tau(3+M).$$   By the definition of $M$, there exists some $z^*\in S_{Y^*}^{|\cdot|}$, a tree $S$ with $o(S)=\omega^\xi$, and a weak$^*$-null, $(1,A^*)$-large collection $(z^*_t)_{t\in S}$ such that $$\sup_{t\in S} |z^*+\sum_{s\leqslant t}z^*_s| \leqslant 1+\delta+M,$$ so that $$\sup_{t\in S}|\sum_{s\leqslant t}z^*_s|\leqslant 2+\delta+M\leqslant 3+M,$$ $$\sup_{t\in S} |\tau \sum_{s\leqslant t}z^*_s|  \leqslant \tau(3+M),$$ and \begin{align*} 1+\delta^{\text{weak}^*}_\xi(\tau;A:X\to (Y, |\cdot|)) & \leqslant \sup_{t\in S} |z^*+\tau\sum_{s\leqslant t} z^*_s|  \leqslant 1+\tau(3+M) \\ & \leqslant \sup_{t\in T}|\sum_{s\leqslant t} y^*_s|-2\leqslant \sup_{t\in T}|y^*+\sum_{s\leqslant t}y^*_s| -1 \\ & \leqslant 1+\delta +\delta^{\text{weak}^*}_\xi(\tau;A:X\to (Y, |\cdot|))-1 \\ & <1+\delta^{\text{weak}^*}_\xi(\tau;A:X\to (Y, |\cdot|)),\end{align*} a contradiction.    This shows that $\sup_{t\in T}|y^*+\sum_{s\leqslant t} y^*_s|\leqslant 3+\tau(3+M)$, as claimed.  Now for any $\tau_1>0$, \begin{align*} 1+\delta^{\text{weak}^*}_\xi(\tau_1;A:X\to (Y, |\cdot|)) & \leqslant \sup_{t\in T}|y^*+\tau_1\sum_{s\leqslant t}y^*_s| \\ & \leqslant 1+\delta +\delta^{\text{weak}^*}_\xi(\tau;A:X\to (Y, |\cdot|)) +|\tau_1-\tau|\sup_{t\in T}|\sum_{s\leqslant t} y^*_s| \\ & \leqslant 1+\delta +\delta^{\text{weak}^*}_\xi(\tau;A:X\to (Y, |\cdot|)) +|\tau_1-\tau|(3+\tau(3+M)). \end{align*}    Since $0<\delta<1$ was arbitrary, we deduce that $$|\delta^{\text{weak}^*}_\xi(\tau_1;A:X\to (Y, |\cdot|))- \delta^{\text{weak}^*}_\xi(\tau;A:X\to (Y, |\cdot|))|\leqslant |\tau_1-\tau| (3+\tau(3+M)).$$  Since this estimate does not depend on the choice of norm $|\cdot|\in\mathcal{Y}_b$, we deduce from this continuity of both $\delta^{\text{weak}^*}_\xi(\cdot;A)$ and $\psi^b_\xi(\cdot;A)$.

\end{proof}

\begin{rem}\upshape The previous result is stated for operators.   That is, $\mathcal{Y}_b$ and $\phi^b_\xi(\cdot;A)$, $\psi^b_\xi(\cdot;A)$ are defined in terms of equivalent norms on $Y$, while the norm on $X$ is held fixed.    However, in the spatial case, we are interested in considering $I_X:(X, |\cdot|)\to (X, |\cdot|)$ as $|\cdot|$-ranges over the appropriate set of norms $\mathcal{X}_b$.  For that, we would be interested in  $$\widehat{\phi}^b_\xi(\sigma;X)=\inf\{\varrho_\xi(\sigma;I_X:(X, |\cdot|)\to (X, |\cdot|)): |\cdot|\in \mathcal{X}_b\}$$ and $$\widehat{\psi}_\xi^b(\sigma;X)=\sup\{\delta_\xi^{\text{weak}^*};I_X:(X, |\cdot|)\to (X, |\cdot|)\}$$ rather than $\phi^b_\xi(\cdot;I_X)$ and $\psi^b_\xi(\cdot;I_X)$.     The proofs above that $\phi^b_\xi(\cdot;A)$ and $\psi^b_\xi(\cdot;A)$ can be modified to show that $\widehat{\phi}^b_\xi(\cdot;X)$ and $\psi^b_\xi(\cdot;X)$ are star shaped, but they do not carry over directly. This is because the previous proof refers to Szlenk derivatives of a particular set, or trees which are $(\ee, A^*)$-separated or $(\ee, A^*)$-large. These notions  depend implicitly on the norm of $X^*$, which is held fixed in the previous proof.    However, for the spatial case, as the norms $|\cdot|$ range through $\mathcal{X}_b$, the norm on $X^*$ with respect to which we are taking Szlenk derivations will also vary. However, the inclusion of an additional factor of $b$ in the appropriate places will take care of this.

\end{rem}

Let $F$ denote the class of functions $f:[0,1]\to [0, \infty)$ which are non-decreasing  and $f(0)=0$.  We recall that the \emph{Young dual} of $f$, denoted by $f^*$, is defined by $$f^*(t)=\sup\{st-f(s): s\in [0,1]\}.$$  We recall that $f^*\in F$, and $f^*$ is convex.   Given $f,g\in F$, we write $f\preceq g$ to mean there exists $c\geqslant 1$ such that $f(t/c)\leqslant g(t)$ for all $t\in [0,1]$.   We write $f\approx g$ if $f\preceq g$ and $g\preceq f$.     We also note that if $f\preceq g$, then $g^*\preceq f^*$, and if $f$ is convex, $f^{**}=f$.

Given an operator $A:X\to Y$, an ordinal $\xi$, and $\sigma>0$, let $G_\xi(\sigma;A)$ denote the minimum $n\in\nn$, provided such an $n$ exists, such that $\sigma \theta_{\xi,n+1}>\frac{1}{n+1}$.      If no such $n$ exists, we let $G_\xi(\sigma;A)=\infty$.

\begin{proposition} If $b\geqslant 2\max\{1, \|A\|\}$ and $G_\xi(\sigma;A)\geqslant n$, then $\phi^b_\xi(\sigma;A)\leqslant 1/n$.

\label{renorm1}
\end{proposition}

\begin{proof} Let $a=\max\{1, \|A\|\}$ and let $[\cdot]$ be the norm on $Y$ given by $[y]=a^{-1}\|y\|$.  Then $\|A:X\to (Y, [\cdot])\|\leqslant 1$. Note that since $[\cdot]\leqslant \|\cdot\|$, $\theta_{\xi,n}(A:X\to (Y, [\cdot]))\leqslant \theta_{\xi,n}(A)$, and $G_\xi(\sigma;A:X\to (Y, [\cdot]))\geqslant n$.      By \cite{TAMS}, there exists a norm $|\cdot|$ on $Y$ such that $$\frac{1}{2}B_Y^{|\cdot|}\subset B_Y^{[\cdot]}\subset 2 B_Y^{|\cdot|}$$ and $\varrho_\xi(\sigma;A:X\to (Y, |\cdot|))\leqslant 1/n$.  Since $|\cdot|\in \mathcal{Y}_{2a}\subset \mathcal{Y}_b$, we are done.

\end{proof}

\begin{rem}\upshape Again, we mention the appropriate modification for the spatial case.   In this case, $a=\max\{1, \|I_X\|\}=1$, so we only need $b\geqslant 1$.  The initial passage to a norm $[\cdot]$ in the previous proof is unnecessary, and we deduce the existence of a norm $|\cdot|$ on $X$ such that $$\frac{1}{2}B_X\subset B_X^{|\cdot|}\subset 2B_X$$ such that $\varrho_\xi(\sigma;I_X:X\to (X, |\cdot|)) \leqslant 1/n$.    Then $\varrho_\xi(\sigma/2:I_X:(X, |\cdot|)\to (X, |\cdot))\leqslant 1/n$.   Therefore in the spatial case, we deduce that for any $b\geqslant 2$, if $G_\xi(\sigma;X)\geqslant 1/n$, then $\widehat{\phi}^b_\xi(\sigma/2;X)\leqslant 1/n$.

\end{rem}

\begin{proposition} If $0<\delta^{\text{\emph{weak}}^*}_\xi(\tau;A)$, then $Cz_\xi(A,2\tau) \leqslant 1+1/\delta^{\text{\emph{weak}}^*}_\xi(\tau;A).$ Therefore for any $b\geqslant 1$, $Cz(A,2\tau b) \leqslant 1+1/\psi^b_\xi(\tau;A)$.

\label{renorm2}
\end{proposition}

\begin{proof} We prove the first part by contradiction.    If $Cz_\xi(A, 2\tau)>1+1/\delta^{\text{weak}^*}_\xi(\tau;A)$, then there exists a number $0<\delta<\delta^{\text{weak}^*}_\xi(\tau;A)$ such that $Cz_\xi(A,2\tau)>1+1/\delta$.

First fix $0<r\leqslant 1$.  We will show that if $r\leqslant  \delta$, then $s^{\omega^\xi}_{2\tau}(r A^*B_{Y^*})=\varnothing$.   We show this by contradiction.  To that end, assume $x^*\in s^{\omega^\xi}_{2\tau}(rA^*B_{Y^*})$.  There exists a rooted tree $T$ with $o(T)=\omega^\xi+1$, and a weak$^*$-closed, $(\tau, A^*)$-separated collection $(y^*_t)_{t\in T}\subset rB_{Y^*}$ such that $A^*y^*_\varnothing=x^*$.  Let $z^*=y^*_\varnothing$, $S=T\setminus\{\varnothing\}$, and $z^*_t=y^*_t-y^*_{t^-}$ for $t\in S$.   Then $(z^*_t)_{t\in S}$ is weak$^*$-closed, $(\tau, A^*)$-separated, and $o(S)=\omega^\xi$. 

 First suppose that $y^*_\varnothing\neq 0$.  Then $0<\|y^*_\varnothing\|\leqslant r$, and $(\|z^*\|^{-1}z^*_t)_{t\in S}$ is $\tau/\|z^*\|$-separated.    Since $\delta^{\text{weak}^*}_\xi(\cdot;A)$ is star shaped, $$\delta^{\text{weak}^*}_\xi(\tau/\|z^*\|;A)>\delta/\|z^*\|.$$   From this it follows that $$\sup_{t\in S} \bigl\|\frac{z^*}{\|z^*\|}+\frac{1}{\|z^*\|}\sum_{s\leqslant t} z^*_s\bigr\| \geqslant 1+\delta/\|z^*\|.$$  However, since $z^*+\sum_{s\leqslant t}z^*_s=y^*_t\in rB_{Y^*}$, $$\sup_{t\in S} \bigl\|\frac{z^*}{\|z^*\|}+\frac{1}{\|z^*\|}\sum_{s\leqslant t} z^*_s\bigr\| \leqslant r/\|z^*\|.$$   From this we deduce that $$1+\delta/\|z^*\| \leqslant r/\|z^*\|.$$  Rearranging yields that $\|z^*\|\leqslant r-\delta$, which contradicts the fact that $r-\delta\leqslant 0$, while $\|z^*\|=\|y^*_\varnothing\|>0$.  This yields that if $r\leqslant \delta$, then $y^*_\varnothing=0$.    If $Y=\{0\}$, then $A$ is the zero operator and $s^{\omega^\xi}_{2\tau}(rA^*B_{Y^*})=\varnothing$. If $Y\neq \{0\}$, we may let $S=T\setminus \{\varnothing\}$ and $z^*_t=y^*_t-y^*_{t^-}$ as before, but now fix any  $z^*\in S_{Y^*}$.  Then  $$1+\delta< \sup_{t\in T} \|z^*+\sum_{s\leqslant t}z^*_s\|\leqslant  1+r,$$ a contradiction, since $r\leqslant \delta$.   This finishes the claim from the beginning of the second paragraph.

We next claim that if $r>\delta$, then $$s^{\omega^\xi}_{2\tau}(rA^*B_{Y^*})\subset (r-\delta)A^*B_{Y^*}.$$   Since $r-\delta>0$, $0=A^*0\in (r-\delta)A^*B_{Y^*}$, and we only need to prove that if $0\neq x^*\in s^{\omega^\xi}_{2\tau}(rA^*B_{Y^*})$, then there exists $z^*\in (r-\delta)B_{Y^*}$ such that $A^*z^*=x^*$.  To that end, we may fix $z^*$, $S$, $(z^*_t)_{t\in S}$ such that $0<\|z^*\|\leqslant r\leqslant 1$, $A^*z^*=x^*$,  $o(S)=\omega^\xi$, and $(z^*_t)_{t\in S}$ is weak$^*$-null and $(\tau,A^*)$-separated as in the previous paragraph.  The same computation as in that paragraph yields that $$1+\delta/\|z^*\|\leqslant r/\|z^*\|,$$ and rearranging yields that $\|z^*\|\leqslant r-\delta$.

Applying the previous claim inductively, we see  that for each $m\in\nn\cup \{0\}$ such that $1-\delta m>0$,  $c^m_{\xi, 2\tau}(A^*B_{Y^*})\subset (1-m\delta)A^*B_{Y^*}$.   If $m=0$, this is clear, while if $1-\delta(m+1)>0$ and $c^m_{\xi,2\tau}(A^*B_{Y^*})\subset (1-\delta m)A^*B_{Y^*}$, our second claim above with $r=1-\delta m$ yields that $$s^{\omega^\xi}_{2\tau}(c^m_{\xi,2\tau}(A^*B_{Y^*}))\subset  s^{\omega^\xi}_{2\tau}((1-\delta m)A^*B_{Y^*})\subset (1-\delta(m+1))A^*B_{Y^*}$$ and $$c^{m+1}_{\xi, 2\tau}(A^*B_{Y^*}) = \overline{\text{co}}^{\text{weak}^*}(s^{\omega^\xi}_{2\tau}(c^m_{\xi,2\tau}(A^*B_{Y^*})))\subset (1-\delta(m+1))A^*B_{Y^*}.$$   Finally, if $n\in\nn$ is the minimum of those $m\in\nn$ such that $\delta m\geqslant 1$, applying the first claim of the proof with $r=1-\delta (n-1)\in (0,\delta]$ yields that $$c^n_{\xi, 2\tau}(A^*B_{Y^*})= \overline{\text{co}}^{\text{weak}^*}(s^{\omega^\xi}_{2\tau}(c^{n-1}_{\xi,2\tau}(A^*B_{Y^*}))) \subset \overline{\text{co}}^{\text{weak}^*}(s^{\omega^\xi}_{2\tau}(r A^*B_{Y^*}))= \overline{\text{co}}^{\text{weak}^*}(\varnothing)=\varnothing.$$   Therefore $Cz_\xi(A, 2\tau)\leqslant n$, and $$Cz_\xi(A, 2\tau)\leqslant \min \{n\in \nn: \delta n\geqslant 1\}\leqslant 1+1/\delta.$$  This contradiction finishes the proof of the first statement.

For the second statement, if $b\geqslant 1$ and $0<\delta<\psi^b_\xi(\tau;A)$ are  such that $Cz_\xi(A;2\tau b)>1+1/\delta$, we may fix $|\cdot|\in \mathcal{Y}_b$ such that $\delta<\delta^{\text{weak}^*}_\xi(\tau;A:X\to (Y, |\cdot|))$.   It is easy to see that for any $m\in\nn$, $$Cz_\xi(A:X\to Y,2\tau b)\leqslant Cz_\xi(A:X\to (Y, |\cdot|), 2\tau)\leqslant 1+1/\delta.$$

\end{proof}

\begin{rem}\upshape For the spatial case, we can replace $\psi^b_\xi(\cdot;I_X)$ with $\widehat{\psi}^b_\xi(\cdot;X)$ as long as we replace $Cz_\xi(A:X\to Y, 2 \tau b)$ with $Cz_\xi(X,2\tau b^2)$.

\end{rem}

The next proposition was shown in \cite{TAMS}.

\begin{proposition} Let $\xi$ is an ordinal, $n\in\nn$, $a>0$, and $A:X\to Y$ are such that $n\theta_{\xi,n}(A)>a$, then there exist  $\ee_1, \ldots, \ee_n\geqslant 0$ such that $$\varnothing\neq s_{\ee_1}^{\omega^\xi}\ldots s^{\omega^\xi}_{\ee_n}(A^*B_{Y^*})$$ and $\sum_{i=1}^n \ee_i>a$.   
\label{tams}

\end{proposition}

\begin{proposition} Suppose $\xi$ is an ordinal, $K\subset X^*$ is weak$^*$-compact and convex,  $\ee\geqslant 0$, and $m\in\nn$. Then  $c_{\xi, m\ee}(K) \subset c^m_{\xi,\ee}(K)$.

In particular, if for $\ee>0$ and $m,n\in\nn$, $Cz_\xi(A, \ee m)\geqslant n+1$, then $Cz_\xi(A, \ee)\geqslant nm+1$. 

\label{jbg}
\end{proposition}

\begin{proof} We claim that for any $0\leqslant k\leqslant m$, $$\frac{k}{m}c_{\xi,\ee m}(K)+ \frac{m-k}{m}K \subset c^k_{\xi, \ee}(K),$$ which we prove by induction.  The $k=0$ case is equality.   Suppose $0<m\leqslant k$ and $$\frac{k-1}{m}c_{\xi, \ee m}(K)+\frac{m-k+1}{m} K \subset c^{k-1}_{\xi, \ee}(K).$$   Then \begin{align*} \frac{k}{m} c_{\xi, \ee m}(K) + \frac{m-k}{m}K & = \frac{k-1}{m}c_{\xi, \ee m}(K) + \frac{1}{m}\overline{\text{co}}^{\text{weak}^*}(s^{\omega^\xi}_{\ee m}(K)) + \frac{m-k}{m}K  \\ & \subset \overline{\text{co}}^{\text{weak}^*}\Bigl(\frac{k-1}{m}c_{\xi, \ee m}(K) + \frac{1}{m}s^{\omega^\xi}_{\ee m}(K) + \frac{m-k}{m}K \Bigr) \\ & = \overline{\text{co}}^{\text{weak}^*}\Bigl(\frac{k-1}{m}c_{\xi, \ee m}(K) + s^{\omega^\xi}_\ee(\frac{1}{m}K) + \frac{m-k}{m}K \Bigr) \\ & \subset \overline{\text{co}}^{\text{weak}^*}\Bigl(s^{\omega^\xi}_\ee\bigl(\frac{k-1}{m}c_{\xi, \ee m}(K) + \frac{1}{m}K + \frac{m-k}{m}K\bigr) \Bigr) \\ & = c_{\xi, \ee}\Bigl(\frac{k-1}{m}c_{\xi, \ee m}(K) + \frac{m-k+1}{m}K\Bigr) \\ & \subset c_{\xi, \ee}(c^{k-1}_{\xi, \ee}(K)) = c^k_{\xi, \ee}(K).  \end{align*}

Here we have used the obvious facts that if $L,M\subset X^*$ are weak$^*$-compact and $\ee_1, a>0$, $s_{\ee_1}(L)+M\subset s_{\ee_1}(L+M)$ and $s_{\ee_1 a}(aK)=a s_{\ee_1}(K)$.   From these obvious facts and an easy induction one can prove that for any ordinal $\zeta$, any $L,\subset X^*$ weak$^*$-compact, and $\ee_1, a>0$, $s_{\ee_1}^\zeta(L)+M\subset s_{\ee_1}^\zeta(L+M)$ and $s_{\ee_1 a}^\zeta(aK)=as_{\ee_1}^\zeta(K)$.

For the last statement, suppose that $Cz_\xi(A, \ee m)\geqslant n+1$.   Then applying the first part iteratively, we deduce that $$\varnothing\neq c_{\xi,  \ee m}^n(A^*B_{Y^*})\subset c_{\xi, m\ee}^{n-1} c^m_{\xi, \ee}(A^*B_{Y^*}) \subset c_{\xi, m\ee}^{n-2}c_{\xi, \ee}^m c_{\xi, \ee}^m(A^*B_{Y^*})= c_{\xi, \ee m}^{n-2}c^{2m}_{\xi, \ee}(A^*B_{Y^*}) \subset \ldots \subset c_{\xi, \ee}^{nm}(A^*B_{Y^*}).$$  
 
\end{proof}

\begin{proposition} If $n\in\nn$, $\ee_1, \ldots, \ee_n\geqslant 0$, $s=\sum_{i=1}^n \ee_i$,  $0<t\leqslant \frac{s}{2n}$,  $K\subset X^*$ is weak$^*$-compact and convex, and $s^{\omega^\xi}_{\ee_1}\ldots s^{\omega^\xi}_{\ee_n}(K)\neq \varnothing$, then $Cz_\xi(K, t)-1 \geqslant s/4t$.  
\label{npr}
\end{proposition}

\begin{proof} Let $m_i=\lfloor \ee_i/t\rfloor$ and note that, by Proposition \ref{jbg}, $$\varnothing\neq c^{\omega^\xi}_{\ee_1}\ldots c_{\ee_n}^{\omega^\xi}(K)\subset c^{\omega^\xi {\sum_{i=1}^n m_i}}_t(K).$$  Since $$\sum_{i: \ee_i<t} \ee_i \leqslant n t\leqslant s/2,$$   we see that  \begin{align*} \sum_{i=1}^n m_i & = \sum_{i: \ee_i\geqslant t} \lfloor \ee_i/t\rfloor  \geqslant \frac{1}{2t}\sum_{i: \ee_i\geqslant t} \ee_i\geqslant \frac{1}{2t}(s-s/2)=s/4t. \end{align*}

\end{proof}

\begin{fact} Suppose $f,g\in S$ are such that for each $0\leqslant \sigma, \tau\leqslant 1$, \begin{enumerate}[(i)]\item $1+\sigma\tau/2 \leqslant (1+f(\sigma))(1+g(\tau))$, and \item if $g(\tau)\geqslant \sigma\tau$, then $f(\sigma)\leqslant \sigma\tau$. \end{enumerate}

Then $f\approx g^*$ and $g\approx f^*$. 
\label{basics}
\end{fact}

\begin{proof} Let $C=1+f(1)$.   Then for any $0\leqslant \sigma, \tau\leqslant 1$, $$\frac{\sigma\tau}{2}\leqslant f(\sigma)+g(\tau)+f(\sigma)g(\tau) \leqslant f(\sigma)+ C g(\tau).$$  From this it follows that for any $0\leqslant  \tau\leqslant 1$, with $\tau'=\tau/C$, \begin{align*} f^*(\tau/2C) & = \max_{0\leqslant \sigma\leqslant 1} \frac{\sigma\tau}{2C} - f(\sigma) \leqslant Cg(\tau')\leqslant g(\tau).  \end{align*} Therefore $f^*\preceq g$.

Let $M=\max\{1, g(1)\}$.   Fix $0<\tau\leqslant 1$ and let $\sigma= \frac{g(\tau/2)}{M\tau/2}\in [0,1]$.   Then since $$g(\tau/2)=M \frac{\sigma\tau}{2} \geqslant \frac{\sigma\tau}{2},$$ $f(\sigma) \leqslant \sigma\tau/2$ by $(ii)$.    Therefore $$f^*(\tau) \geqslant \sigma\tau -f(\sigma) \geqslant \frac{\sigma\tau}{2} = \frac{g(\tau/2)}{M} \geqslant g(\tau/2M).$$   From this it follows that $g\preceq f^*$, and $g\approx f^*$.

Now since $f\in S$, $$F(\sigma):=\int_0^\sigma f(x)/x dx$$ defines a function $F$ on $[0,1]$ which is convex, continuous, non-decreasing, and     $F(0)=0$. Since $F$ is convex, $F=F^{**}$.  It is easily verified that $$F(\sigma)\leqslant f(\sigma)\leqslant F(2\sigma),$$ whence $f\approx F$, and $$g^*\approx f^{**}\approx F^{**}=F\approx f.$$

\end{proof}

\begin{lemma} For any $b\geqslant 1$, any ordinal $\xi$, and an operator $A:X\to Y$ with $Sz(A)>\omega^\xi$,  $$\phi^b_\xi(\cdot;A)^*\approx \psi_\xi^b(\cdot;A)$$ and $$\psi^b_\xi(\cdot;A)^*\approx \phi^b_\xi(\cdot;A).$$

\label{kc}
\end{lemma}

\begin{proof} We will argue that with $f=\phi^b_\xi(\cdot;A)$ and $g=\psi^b_\xi(\cdot;A)$, the hypotheses of Fact \ref{basics} are satisfied. 

Arguing as in the proof of Proposition $3.2$ in \cite{CD}, for $\mu>0$ and $0\leqslant \sigma,\tau\leqslant 1$, we may fix $|\cdot|\in \mathcal{Y}_b$ and $y^*\in S_{Y^*}^{|\cdot|}$, a tree $T$ with $o(T)=\omega^\xi$, a weak$^*$-null and $(\tau, A^*)$-large collection $(y^*_t)_{t\in T}$ with $\sup_{t\in T}|y^*+\sum_{s\leqslant t}y^*_s|\leqslant 1+\psi^b_\xi(\tau;A)+\mu$, $y\in S_Y^{|\cdot|}$, and a weakly null $(x_t)_{t\in T}\subset \sigma B_Y^{|\cdot|}$ such that for any $t\in T$ and any convex combination $x$ of $(x_s: s\leqslant t)$, $$\text{Re\ }(y^*+\sum_{s\leqslant t}y^*_s)(y+ Ax)\geqslant 1+\sigma\tau/2-\mu.$$  Taking the infimum over $t\in T$ and convex combinations $x\in (x_s: s\leqslant t)$, we deduce that $1+\sigma\tau/2 -\mu\leqslant (1+\phi^b_\xi(\sigma;A))(1+\psi^b_\xi(\tau;A)+\mu)$. Since $\mu>0$ was arbitrary, we obtain $(i)$ of Fact \ref{basics}.  For $(ii)$ of Fact \ref{basics}, we cite \cite[Proposition $3.2$$(ii)$]{CD}.   This is for $\rho_\xi$ and $\delta^{\text{weak}^*}_\xi$, but using this inequality for each $|\cdot|\in \mathcal{Y}_b$ gives the result for $\phi^b_\xi$ and $\psi^b_\xi$.

\end{proof}

\begin{theorem} Let $\xi$ be an ordinal and $A:X\to Y$ be an operator with $Sz(A)>\omega^\xi$.  Let $H(t)=Cz_\xi(A,t)^{-1}$, $G(t)=G_\xi(t;A)^{-1}$ for $0<t\leqslant 1$ and $H(0)=G(0)=0$.        Then for any $b\geqslant 2\max\{1, \|A\|\}$, $$G \approx \phi^b_\xi(\cdot;A)\approx \psi^b_\xi(\cdot;A)^*\approx H^*$$ and $$G^*\approx \phi^b_\xi(\cdot;A)^*\approx \psi^b_\xi(\cdot;A)\approx H.$$

\label{big theorem}
\end{theorem}

\begin{proof} For ease of notation in the proof, we  let $\phi=\phi^b_\xi(\cdot;A)$ and $\psi=\psi^b_\xi(\cdot;A)$. In the proof, we let $z,i:[0,1]\to [0,1]$ be given by $z(\sigma)=0$ for all $0\leqslant \sigma\leqslant 1$ and $i(\tau)=\tau$ for all $0\leqslant \tau\leqslant 1$. Note that $z,i$ are convex and lie in $S$.   Note also that $z=i^*$ and $i=z^*$.  We also note that for any $f\in S$, $f\preceq z$ if and only if there exists $0<\sigma\leqslant 1$ such that $f(\sigma)=0$.

By Proposition \ref{renorm1} we know that $\phi\preceq G$, and by Lemma \ref{kc}, $\phi\approx \psi^*$ and $\psi\approx \phi^*$.  

We will next show that $H\preceq i$.  Since $Sz(A)>\omega^\xi$, there exists $0<\ee\leqslant 1$ such that $s^{\omega^\xi}_\ee(A^*B_{Y^*})\neq \varnothing$, whence $Cz_\xi(A,\ee)\geqslant 1+1$.     Let $C=1/\ee$.   Then for any $0<\tau\leqslant 1$, if $n\in\nn$ is such that $$\frac{1}{n+1}\leqslant \tau\leqslant \frac{1}{n},$$ Proposition \ref{jbg} yields that $$Cz_\xi(A,\tau/C) = Cz_\xi(A, \ee\tau) \geqslant Cz_\xi(A, \ee/n) \geqslant n+1.$$ From this we deduce that $$H(\tau/C) \leqslant \frac{1}{n+1}\leqslant \tau =i(\tau).$$  This yields that $H\preceq i$.

Next, we will show that $\psi\preceq H$.  Let $M=1+\psi(1)$.    Then by Proposition \ref{renorm2}, $$Cz_\xi(A, 2b\tau) \leqslant 1+1/\psi(\tau)= \frac{M}{\psi(\tau)}.$$ Therefore $$\psi(\tau/2bM) \leqslant \psi(\tau/2b)/M \leqslant H(\tau)$$ for any $0<\tau\leqslant 1$. 

Next, fix  $4\leqslant r\in \nn$ such that $s^{\omega^\xi}_{1/r}(A^*B_{Y^*})\neq \varnothing$ and  define $k:[0,1]\to [0, \infty)$ by letting $$k(\tau)=\sup_{0< \alpha \leqslant 1} H(\alpha \tau/r)/\alpha.$$   We first show that $k\leqslant H$.           Fix $0<\alpha, \tau\leqslant 1$ and $n\in \nn$ such that $\frac{1}{n+1}\leqslant \alpha \leqslant \frac{1}{n}$. Let $m+1=Cz_\xi(A, \tau)$. If $m=0$, then  $$Cz_\xi(A, \alpha \tau/r)\geqslant Cz_\xi(A, 1/nr)\geqslant n+1,$$ and $$H(\alpha\tau/r)\leqslant \frac{1}{n+1} = \frac{1}{n+1}\cdot H(\tau) \leqslant \alpha H(\tau).$$   If $m>0$, then $$Cz_\xi(A, \alpha\tau/r) \geqslant Cz_\xi(A, \tau/4n) \geqslant 4nm+1,$$ and $$H(\alpha\tau/r) \leqslant \frac{1}{4nm+1} \leqslant \frac{1}{n+1}\cdot \frac{1}{m+1} \leqslant \alpha H(\tau).$$   From this it follows that $k \leqslant H$.  Of course, $H(\tau/r)\leqslant k(\tau)$, so $H\approx k$.   Now note that for any $0\leqslant \beta, \tau\leqslant 1$, $k(\beta\tau) \leqslant \beta k(\tau)$. Using this fact, one easily checks that $k\approx K$, where $K(\tau)=\int_0^\tau k(x)/xdx$, and that $K$ is convex.  From this it follows that $k^{**}\approx K^{**}=K\approx k$.   Since $k\approx H$, $$H\approx k \approx k^{**}\approx H^{**}.$$

The remainder of the proof will be split into two cases.   

Case $1$, $G\preceq z$:  This means there exists $0<\sigma\leqslant 1$ such that $G(\sigma)=0$. By Proposition \ref{renorm1}, we deduce that $$z\preceq \phi\preceq G\preceq z,$$ $$\psi\approx \phi^*\approx z^*=i,$$ and $$i\approx \psi \preceq H\preceq i.$$

Case $2$, for all $0<\sigma$, $G_\xi(\sigma;A)<\infty$:  We note that if $|\cdot|$ is any equivalent norm on $Y$, and if $G_0=G_\xi(\cdot;A:X\to (Y, |\cdot|))$ and $H_0=1/Cz_\xi(A:X\to (Y, |\cdot|))$, then $G_0\approx G$ and $H_0\approx H$.  Let us fix a norm $|\cdot|$ on $Y$ such that $A:X\to (Y, |\cdot|)$ has norm $1$. Let $G_0$, $H_0$ be as above.  We will show that $G_0\preceq H_0^*$, which will yield that $G\preceq H^*$.   Fix $0<\sigma\leqslant 1/17^2$ and let $1/n=G_0(\sigma)$.   Let $\sigma_1=2\sigma$.   Then by definition of $G_\xi$, $$(n+1)\theta_{\xi, n+1}(A:X\to (Y, |\cdot|))\geqslant 1/\sigma>1/\sigma_1.$$  Thus there exist $\ee_1, \ldots, \ee_{n+1}\geqslant 0$ such that $\sum_{i=1}^{n+1}\ee_i >1/\sigma_1$ and  $$s^{\omega^\xi}_{\ee_1}\ldots s^{\omega^\xi}_{\ee_{n+1}}(A^*B_{Y^*}^{|\cdot|})\neq \varnothing.$$   Now since $\theta_{\xi, n+1}(A:X\to (Y, |\cdot|))\leqslant 1$, it follows that $$1\leqslant \sigma (n+1)\theta_{\xi, n+1}(A:X\to (Y, |\cdot|)) \leqslant  \sigma(n+1)$$ and $$\tau:= \frac{1}{4(n+1)}\cdot \frac{1}{\sigma_1} <1.$$   Moreover, $$Cz_\xi(A:X\to (Y, |\cdot|), \tau) \geqslant \frac{\sum_{i=1}^{n+1}\ee_i}{8\tau}> \frac{1}{16\sigma \tau},$$ so that $$H_0(\tau) \leqslant 16\sigma\tau.$$  From this it follows that $$H^*_0(17\sigma) \geqslant 17\sigma\tau - H_0(\tau) \geqslant \sigma\tau>\frac{1}{17n}=G_0(\sigma)/17.$$  From this and convexity of $H^*_0$, it follows that for any $0<\sigma\leqslant 1$, $$G_0(\sigma/17^2) =17 G_0(\sigma/17^2)/17 \leqslant 17 H^*_0(\sigma/17) \leqslant H^*_0(\sigma).$$   This yields that $G\preceq H^*$.   Now we finally deduce that $$\phi\preceq G\preceq H^*\preceq \psi^*\approx\phi,$$ so that $$\phi\approx G\approx H^*\approx \psi^*$$ and $$\phi^*\approx G^*\approx H^{**}\approx H\approx \psi.$$

\end{proof}

\begin{rem}\upshape The modification for the spatial case follows as usual, with the inclusion of an additional factor of $b$ or $1/2$ in the appropriate places.

\end{rem}

\begin{corollary} Fix an ordinal $\xi$, $1<p< \infty$, and an operator $A:X\to Y$.  \begin{enumerate}\item The following are equivalent.  \begin{enumerate}[(i)]\item $A\in \mathfrak{N}_{\xi,p}$, \item With $b=2\max\{1, \|A\|\}$, $\sup_{\sigma>0}\phi^b_\xi(\sigma;A)/\sigma^p<\infty$. \item There exists $b\geqslant 1$ such that $\sup_{\sigma>0}\phi^b_\xi(\sigma;A)/\sigma^p<\infty$.   \end{enumerate}

\item The following are equivalent. \begin{enumerate}[(i)]\item $A\in \mathfrak{N}_{\xi,\infty}$. \item With $b=2\max\{1, \|A\|\}$, there exists $\sigma>0$ such that  $\phi^b_\xi(\sigma;A)=0$.  \item There exist $b\geqslant 1$ and $\sigma>0$ such that $\phi^b_\xi(\sigma;A)=0$.  \end{enumerate} \end{enumerate}

\end{corollary}

\begin{proof} Note that $\phi^b_\xi(\cdot;A)\approx G_\xi(\cdot;A)$ whenever $b\geqslant 2\max\{1, \|A\|\}$ and  $\phi^b_\xi(\sigma;A)$ is  decreasing as a funciton of $b$.  Therefore  for $1<p<\infty$, there exists $b\geqslant 1$ such that $\sup_{\sigma>0}\phi^b_\xi(\sigma;A)/\sigma^p<\infty$ if and only if there exists $b\geqslant 2\max\{1, \|A\|\}$ such that $\sup_{\sigma>0}\phi^b_\xi(\sigma;A)/\sigma^p$ and

In light of Theorem \ref{big theorem}, it is straightforward to verify that for $1<p<\infty$ and $b\geqslant 2\max\{1, \|A\|\}$, \begin{align*} A\in \mathfrak{N}_{\xi, p}(A) & \Leftrightarrow (\exists C>0)(\forall 0<\sigma<1)(G_\xi(\sigma;A)\geqslant 1/C\sigma^p) \\ & \Leftrightarrow (\exists C>0)(\forall \sigma>0)(\phi^b_\xi(\sigma;A)\leqslant C\sigma^p)\end{align*}

and \begin{align*} A\in \mathfrak{N}_{\xi,\infty}(A) & \Leftrightarrow (\exists \sigma>0)(G_\xi(\sigma;A)=\infty) \Leftrightarrow (\exists \sigma>0)(\psi^b_\xi(\sigma;A)=0).\end{align*}

\end{proof}

We similarly deduce the following. 

\begin{corollary} Fix an ordinal $\xi$, $1<p< \infty$, and a Banach space $X$.   \begin{enumerate}\item The following are equivalent.  \begin{enumerate}[(i)]\item $X\in \textsf{\emph{N}}_{\xi,p}$, \item  $\sup_{\sigma>0}\widehat{\phi}^2_\xi(\sigma;X)/\sigma^p<\infty$. \item There exists $b\geqslant 1$ such that $\sup_{\sigma>0}\widehat{\phi}^b_\xi(\sigma;X)/\sigma^p<\infty$.   \end{enumerate}

\item The following are equivalent. \begin{enumerate}[(i)]\item $X\in \textsf{\emph{N}}_{\xi,\infty}$. \item There there exists $\sigma>0$ such that  $ \widehat{\phi}^2_\xi(\sigma;X)=0$.  \item There exist $b\geqslant 1$ and $\sigma>0$ such that $\widehat{\phi}^b_\xi(\sigma;X)=0$.  \end{enumerate} \end{enumerate}

\end{corollary}

\section{An application}

Fix $n\in\nn$, $a_1, \ldots, a_n,C$ non-negative real numbers, and $A:X\to Y$ an operator.   We define a two player game (referred to as the $A, a_1, \ldots, a_n,C$-\emph{game}).  For the game, recall that $D_X$ denotes the set of all weakly open sets in $X$ containing $0$, directed by reverse inclusion.    Player I chooses $U_1\in D_X$, Player II chooses $x_1\in U_1\cap a_1 B_X$, Player I chooses $U_2\in D_X$, Player II chooses $x_2\in U_2\cap a_2 B_X$, $\ldots$, Player I chooses $U_n\in D_X$, and Player II chooses $x_n\in U_n\cap a_n B_X$.  We say that Player I wins provided that $\|A\sum_{i=1}^n x_i\|\leqslant C$, and Player II wins otherwise.

We let $$S_1=\{\varnothing\}\cup \{(x_i)_{i=1}^j: 1\leqslant j<n\}$$ and $$S_2=\{((U_j, x_j)_{j=1}^k, U):1\leqslant k<n, U_1, \ldots, U_k, U\in D_X\}.$$   A \emph{strategy for Player I in the} $A, a_1, \ldots, a_n, C$ -\emph{game} is a function $\phi:S_1\to D_X$, and a \emph{strategy for Player II in the} $A, a_1, \ldots, a_n, C$-\emph{game} is a function $\psi:S_2\to X$ such that $\psi((U_j, x_j)_{j=1}^k, U)\in U\cap a_{k+1}B_X$.    We say a strategy $\phi:S_1\to D_X$ for Player I in the $A, a_1, \ldots, a_n, C$-game is a \emph{winning strategy in the} $A, a_1, \ldots, a_n, C$-\emph{game} provided that if $(x_i)_{i=1}^n$ is any sequence such that $x_i\in \phi((x_j)_{j=1}^{i-1})\cap a_i B_X$ for each $1\leqslant i\leqslant n$, then $\|A\sum_{i=1}^n x_i\|\leqslant C$.     

Standard facts about such games yields that this game is determined. That is, either Player I or Player II has a winning strategy in the game.  Furthermore, it is evident that for any $n\in \nn$ and scalars $a_1, \ldots, a_n$, $\|\sum_{i=1}^n a_i e_i\|_{0, A}$ is the infimum of those $C$ such that Player I has a winning strategy in the $A, |a_1|, \ldots, |a_n|, C$-game.

For a natural number $k\in\nn$, a sequence $w=(w_i)_{i=1}^k$ of positive numbers, and an infinite subset $M$ of $\nn$, we let $G_k^w(M)$ denote the set of sequences $(n_i)_{i=1}^k$ such that $n_i\in M$ for each $1\leqslant i\leqslant k$.   We endow $G_k^w=G_k^w(\nn)$, and therefore the subsets $G_k^w(M)$ of $G_k^w$, with the metric $$d^w_k((m_i)_{i=1}^k, (n_i)_{i=1}^k)=\sum_{i: m_i\neq n_i} w_i.$$   We let $G_k$ denote the metric space $G_k^w$, where $w=(w_i)_{i=1}^k$ and $w_1=\ldots=w_k=1$.

We recall the following theorem from \cite{LR}, which was a generalization of a result from \cite{KR}.    

\begin{theorem} If $1<p\leqslant \infty$ and if  $X$ is a quasi-reflexive, $p$-AUS Banach space, then there exists a constant $C$ such that for any $k\in\nn$, any positive numbers $w_1, \ldots, w_k$,  any Lipschitz function $f:G_k^w\to X$, and any $\ee>0$, there exist an infinite subset $M$ of $\nn$ such that for any $m_1<n_1<\ldots <m_k<n_k$, $m_i, n_i\in M$, $$\|f(m_1, \ldots, m_k)-f(n_1, \ldots, n_k)\|\leqslant C\text{\emph{Lip}}(f) \|(w_i)_{i=1}^k\|_{\ell_p^k}.$$

\end{theorem}

Our application is the following strengthening of this result.

\begin{theorem} Fix $1<p\leqslant \infty$, a quasi-reflexive Banach space $X$, and an operator $A:X\to Y$.  \begin{enumerate}[(i)]\item For any $k\in\nn$, any positive numbers $w_1, \ldots, w_k$,  any $\vartheta>2\|\sum_{i=1}^k w_ie_i\|_{0, A}$, and any non-constant Lipschitz map $f:G_k^w\to X$, there exists an infinite subset $M$ of $\nn$ such that for any $m_1<n_1<\ldots <m_k<n_k$, $m_i, n_i\in M$, $$\|Af(m_1, \ldots, m_k)-Af(n_1, \ldots, n_k)\|\leqslant \text{\emph{Lip}}(f)\vartheta.$$  \item If $A\in \mathfrak{A}_{0,p}$, then for any $\vartheta>2$, any $k\in\nn$, any positive numbers $w_1, \ldots, w_k$, and any non-constant Lipschitz map $f:G_k^w\to X$, there exists an infinite subset $M$ of $\nn$ such that for any $m_1<n_1<\ldots <m_k<n_k$, $m_i, n_i\in M$, $$\|Af(m_1, \ldots, m_k)-Af(n_1, \ldots, n_k)\|\leqslant \vartheta\text{\emph{Lip}}(f)\alpha_{0,p}(A) \|\sum_{i=1}^k w_ie_i\|_{\ell_p}.$$ \item If $A\in \mathfrak{N}_{0,p}$, then for any $\vartheta>2$, any $k\in\nn$, and any non-constant Lipschitz map $f:G_k\to X$, there exists an infinite subset $M$ of $\nn$ such that for any $m_1<n_1<\ldots <m_k<n_k$, $m_i, n_i\in M$,  $$\|Af(m_1, \ldots, m_k)-Af(n_1, \ldots, n_k)\|\leqslant \vartheta \text{\emph{Lip}}(f)N_{0,p}(A)k^{1/p}.$$\end{enumerate}

\label{app}
\end{theorem}

\begin{rem}\upshape Suppose that for $1<p\leqslant \infty$, $X$ is the dual $T^*_q$ of the dual of the $q$-convexification of the Figiel-Johnson Tsirelson space $T_q$.  This is a reflexive Banach space with $T^*_q\in\textsf{A}_{0,p}\setminus \textsf{T}_{0,p}$, so Theorem \ref{app} produces new results applicable to $T^*_q$, as was explained in the introduction.  In particular, if $T^*$ is Tsirelson's original space, then for any $\vartheta>4$, any positive numbers $w_1, \ldots, w_k\leqslant 1$, and any non-constant Lipschitz map $f:G_k^w\to T^*$, there exists $M\in[\nn]$ such that for any $m_1<n_1<\ldots <m_k<n_k$, $$\|f(m_1, \ldots, m_k)-f(n_1, \ldots, n_k)\|\leqslant \vartheta.$$

\end{rem}

\begin{fact} If $X$ is a Banach space and $(u_\lambda)\subset X\subset X^{**}$ is weak$^*$-convergent to some $x^{**}$ with $\|x^{**}\|\leqslant \ee$, then for any $\ee_1>\ee$, there exists a convex combination $u$ of $(u_\lambda)$ such that $\|u\|\leqslant \ee_1$.    
\label{needed fact}
\end{fact}

\begin{proof} Let $C$ denote the closed, convex hull of $(u_\lambda)$ in $X$. Seeking a contradiction, fix $\ee<\ee_2<\ee_1$ and assume that $C\cap  \ee_1 B_X=\varnothing$.   Then by the Hahn-Banach theorem, there exist $ x^*\in X^*$ and a real number $\alpha$ such that $$\underset{x\in \ee_2 B_X}{\sup} \text{Re\ }x^*(x) < \alpha \leqslant \underset{x\in C}{\inf} \text{Re\ }x^*(x).$$   Now since $x^*\neq 0$, by scaling $x^*$ and $\alpha$, we may assume $\|x^*\|=1$. It follows that $$\alpha > \underset{x\in \ee_2 B_X}{\sup} \text{Re\ }x^*(x)=\ee_2.$$    Now $$\ee \geqslant \|x^{**}\|\geqslant \text{Re\ }x^{**}(x^*) = \underset{\lambda}{\lim} \text{Re\ }x^*(u_\lambda)\geqslant \alpha >\ee_2,$$ a contradiction.

\end{proof}

\begin{lemma} Fix $R, \delta>0$, and a Banach space $X$.   Then for any weak$^*$-neighborhood $U$ of $0$ in $X^{**}$, there exists a weak$^*$-neighborhood $V$ of $0$ in $X^{**}$ such that for any $x^{**}\in V\cap RB_{X^{**}}$ with $\|x^{**}\|_{X^{**}/X}<\delta$, there exists $x\in U\cap (R+3\delta)B_X$ such that $\|x^{**}-x\|<3\delta$.

\end{lemma}

\begin{proof} By replacing $U$ with a subset, we may assume $U$ is convex and symmetric. Let $D$ denote the set of convex, symmetric, weak$^*$ open sets in $X$ containing $0$, directed by reverse inclusion.    If the result is not true, then we could find for each $V\in D$ some $x^{**}_V\in V\cap RB_{X^{**}}$ and $u_V\in X$ such that $\|x^{**}_V-u_V\|<\delta$, but such that for each $V\in D$, there does not exist $x\in U\cap (R+3\delta)B_X$ such that $\|x^{**}_V-x\|<3\delta$.  By passing to a subnet $(u_V)_{V\in D_1}$, we may assume $(u_V)_{V\in D_1}$ is weak$^*$-convergent to some $x^{**}\in (R+\delta)B_{X^{**}}$ and $u_{V_1}-u_{V_2}\in U$ for all $V_1, V_2\in D_1$.  Note that $$\|x^{**}\| \leqslant \underset{V\in D_1}{\lim\inf} \|u_V-x^{**}_V\|\leqslant \delta.$$  By Fact \ref{needed fact}, there exists a convex combination $u$ of $(u_V)_{V\in D_1}$ such that $\|u\|<2\delta$.  Fix any $V_1\in D_1$ and note that $u_{V_1}-u\in U\cap (R+3\delta)B_X$ and $\|x^{**}_{V_1}-(u_{V_1}-u)\| < 3\delta$.  This contradiction finishes the proof.

\end{proof}

\begin{theorem} Fix $k\in \nn$,  a sequence  $w=(w_i)_{i=1}^k$ of positive numbers, a quasi-reflexive Banach space $X$, and an operator $A:X\to Y$.   If $f:G_k^w\to X$ is $1$-Lipschitz, then  for any $\vartheta> 2\|\sum_{i=1}^k w_ie_i\|_{0,A}$, there exists $M\in [\nn]$ such that for any $m_1<n_1<\ldots <m_k<n_k$, $m_i,n_i\in M$, $$\|f(m_1, \ldots, m_k)-f(n_1, \ldots, n_k)\|\leqslant \vartheta.$$   

\end{theorem}

\begin{proof} Fix $b>\max\{1,\|A\|\}$.    Fix $\vartheta> 2\|\sum_{i=1}^k w_ie_i\|_{0,A}$. Let $\mathcal{V}$ denote the set of those $(m_1, n_1, \ldots, m_k, n_k)\in [\nn]^{2k}$ such that $\|Af(m_1, \ldots, m_k)-Af(n_1, \ldots, n_k)\|\leqslant \vartheta$.   By the finite Ramsey theorem, there exists $M\in [\nn]$ such that either $[M]^{2k}\subset \mathcal{V}$ or $[M]^{2k}\cap \mathcal{V}=\varnothing$.     We will show that the latter cannot hold, which will finish the proof.  Seeking a contradiction, assume that $[M]^{2k}\cap \mathcal{V}=\varnothing$.    By relabeling, we may assume $M=\nn$.

Now fix $\delta>0$ such that $2\|\sum_{i=1}^k w_i e_i\|_{0,A}+6kb\delta<\vartheta$.   Note that since the formal identity $I:(c_{00}, \ell_1)\to (c_{00}, \|\cdot\|_{0,A})$ has norm at most $\|A\|$, it follows that $$\|\sum_{i=1}^k (2w_i+3\delta)e_i\|_{0,A}+ 3kb\delta \leqslant 2\|\sum_{i=1}^k w_ie_i\|_{0,A} + 6kb\delta <\vartheta.$$      Fix a winning strategy $\phi$ for Player I in the $A, 2w_1+3\delta, \ldots, 2w_k+3\delta, \vartheta-3kb\delta$-game.  We may do this, since $\|\sum_{i=1}^k (2w_i+3\delta)w_i\|_{0,A}<\vartheta-3kb\delta$.       Fix a free ultrafilter $\mathcal{U}$ on $\nn$ and define $F:\cup_{i=0}^k G_i\to X^{**}$ by  \begin{displaymath}
   F(s) = \left\{
     \begin{array}{lr}
       f(s) & : |s|=k\\
       \underset{n_{j+1}\in \mathcal{U}}{\lim}\ldots \underset{n_k\in \mathcal{U}}{\lim} f(s\smallfrown (n_{j+1}, \ldots, n_k)) & : |s|=j<k,
     \end{array}
   \right.
\end{displaymath} 

where all limits are taken with respect to the weak$^*$-topology on $X^{**}$.  Now define $g:\cup_{i=1}^k G_i\to X^{**}$ by $g(s)=F(s)-F(s^-)$ and for $i=1, \ldots, k$, let $e_i:\cup_{j=0}^i G_j\to X^{**}/X$ be given by \begin{displaymath}
   e_i(s)= \left\{
     \begin{array}{lr}
       g(s) & : |s|=j\\
       \underset{n_{i+1}\in \mathcal{U}}{\lim}\ldots \underset{n_k\in \mathcal{U}}{\lim} g(s\smallfrown (n_{i+1}, \ldots, n_k))+X & : |s|=i<j,
     \end{array}
   \right.
\end{displaymath} 

We claim that for all $1\leqslant j\leqslant k$ and $s\in [\nn]^j$, $\|g(s)\|\leqslant w_j$.  Indeed let $s=(n_1, \ldots, n_j)$ and note that \begin{align*} \|g(s)\| & = \|F(s)-F(s^-)\| \\ & \leqslant \underset{m\in \mathcal{U}}{\lim} \underset{n_{j+1}\in \mathcal{U}}{\lim} \ldots \underset{n_k\in \mathcal{U}}{\lim} \|f(n_1, \ldots,n_{j-1}, n_j, n_{j+1}, \ldots, n_k)- f(n_1, \ldots, n_{j-1}, m, n_{j+1}, \ldots, n_k)\| \\ & \leqslant w_j. \end{align*}

Now fix weak$^*$-neighborhoods $W_1, V_1$ of $0$ in $X^{**}$ such that $W_1-W_1\subset V_1$ and if $x^{**}\in V_1\cap 2w_1 B_{X^{**}}$ and $\|x^{**}\|_{X^{**}/X}<\delta$, there exists $x\in \phi(\varnothing) \cap (2w_1+3\delta)B_X$ such that $\|x^{**}-x\|<3\delta$. Since $$\text{weak}^*\text{-}\underset{n\in \mathcal{U}}{\lim} g((n))=0$$ and for each $1\leqslant j\leqslant k$, $$\underset{n\in \mathcal{U}}{\lim} \|e_j(\varnothing)-e_j((n))\|_{X^{**}/X}=0,$$ we may fix $N_1\in \mathcal{U}$ such that for any $n\in N_1$, $g((n))\in W_1$ and $\|e_j(\varnothing)-e_j((n))\|_{X^{**}/X}<\delta/2k$.  Now fix any $m_1<n_1$, $m_1, n_1\in N_1$.   Note that $$g((m_1))-g((n_1))\in (W_1-W_1)\cap 2w_1 B_{X^{**}}\subset V_1 \cap 2w_1B_{X^{**}}$$ and \begin{align*} \|g((m_1))-g((n_1))\|_{X^{**}/X} & = \|e_1((m_1))-e_1((n_1))\|_{X^{**}/X} \\ & \leqslant \|e_1((m_1))-e_1(\varnothing)\|_{X^{**}/X}+\|e_1(\varnothing)-e_1((n_1))\|_{X^{**}/X}  <\delta,\end{align*} whence there exists $x_1\in \phi(\varnothing)\cap (2w_1+3\delta)B_X$ such that $\|g((m_1))-g((n_1))-x_1\|<3\delta$.

Now suppose that $m_1<n_1<\ldots <m_{j-1}<n_{j-1}$,  $N_1\supset N_{j-1}\in \mathcal{U}$, $x_1, \ldots, x_{j-1}\in X$ have been chosen such that for each $1\leqslant i<j$,  \begin{enumerate}[(i)]\item $\|x_i\|\leqslant 2w_i+3\delta$, \item $x_i\in \phi((x_1, \ldots, x_{i-1}))$, \item $m_i, n_i\in N_i$, \item if $s=(m_1, \ldots, m_{i-1})$ and $t=(n_1, \ldots, n_{i-1})$, then for any $n\in N_i$, and $i\leqslant l\leqslant k$, $$\|e_l(s\smallfrown (n))-e_l(s)\|_{X^{**}/X}, \|e_l(t\smallfrown (n))-e_l(t)\|_{X^{**}/X}<\delta/2k.$$  \end{enumerate}

Now fix weak$^*$-neighborhoods $W_j, V_j$ of $0$ in $X^{**}$ such that $W_j-W_j\subset V_j$ and if $x^{**}\in V_j \cap 2w_j B_{X^{**}}$ and $\|x^{**}\|_{X^{**}/X}<\delta$, then there exists $x\in \phi((x_1, \ldots, x_{j-1}))\cap (2w_j+3\delta)B_X$ such that $\|x^{**}-x\|<3\delta$.  Let $s=(m_1, \ldots, m_{j-1})$ and $t=(n_1, \ldots, n_{j-1})$.     Since $$\text{weak}^*\text{-}\underset{n\in \mathcal{U}}{\lim} g(s\smallfrown (n))= \text{weak}^*\text{-}\underset{n\in \mathcal{U}}{\lim} g(t\smallfrown (n))=0$$ and for each $j\leqslant l\leqslant k$, $$\underset{n\in \mathcal{U}}{\lim} \|e_l(s\smallfrown (n))-e_l(s)\|_{X^{**}/X}= \underset{n\in \mathcal{U}}{\lim} \|e_l(t\smallfrown (n))-e_l(t)\|_{X^{**}/X}=0,$$ there exists $N'_j\in \mathcal{U}$ such that for every $n\in N_j'$, $g(s\smallfrown (n)), g(t\smallfrown (n))\in W_j$ and for each $n\in N_j'$ and $j\leqslant l\leqslant k$, $\|e_l(s\smallfrown (n))-e_l(s)\|_{X^{**}/X}, \|e_l(t\smallfrown (n))-e_l(t)\|_{X^{**}/X}<\delta/2k$.    Now let $N_j=N_j'\cap N_{j-1}\cap (n_{j-1}, \infty)\in \mathcal{U}$.   Fix $m_j<n_j$, $m_j, n_j\in N_j$.    Now $$g(s\smallfrown (m_j))-g(t\smallfrown (n_j))\in (W_j-W_j)\cap 2w_j B_{X^{**}}\subset V_j\cap 2w_j B_{X^{**}}$$ and \begin{align*} \|g(s\smallfrown (m_j))-g(t \smallfrown (n_j))\|_{X^{**}/X}  & \leqslant \sum_{l=1}^j \|e_l((m_1, \ldots, m_l))-e_l((m_1, \ldots, m_{l-1}))\|_{X^{**}/X}  \\ & + \sum_{l=1}^j  \|e_l((n_1, \ldots, n_l))-e_l((n_1, \ldots, n_{l-1}))\|_{X^{**}/X} \\ &  < 2j \cdot \frac{\delta}{2k}\leqslant \delta,\end{align*} whence there exists $x_j\in \phi((x_1, \ldots, x_{j-1}))\cap (2w_j+3\delta)\cap B_X$ such that $\|g((m_1, \ldots, m_j))-g((n_1, \ldots, n_j))-x_j\|<3\delta$.   This completes the recursive construction.

Now we note that \begin{align*} \vartheta & < \|Af((m_1, \ldots, m_k))-Af((n_1, \ldots, n_k))\| = \|A\sum_{j=1}^k g((m_1, \ldots, m_j))-g((n_1, \ldots, n_j))\| \\ & \leqslant \|A\sum_{i=1}^k x_i\| + 3 kb\delta \leqslant  (\vartheta - 3 kb \delta)+3 kb \delta= \vartheta. \end{align*}

This contradiction finishes the proof.

\end{proof}

\section{Distinctness of the classes}

We already know that if $\xi$ is any ordinal and $1<p\leqslant \infty$, $$\mathfrak{T}_{\xi,p}\subset \mathfrak{A}_{\xi,p}\subset \mathfrak{N}_{\xi,p}\subset \mathfrak{P}_{\xi,p}.$$   Furthermore, for any ordinals $\xi<\zeta$ and any $1<p\leqslant \infty$, $$\mathfrak{P}_{\xi,p}\subset \mathfrak{D}_{\xi+1}\subset \mathfrak{D}_\zeta\subset \mathfrak{T}_{\zeta, \infty},$$ and for any ordinal $\xi$ and $1<p<r\leqslant \infty$, $$\mathfrak{P}_{\xi,r}\subset \mathfrak{T}_{\xi,p}.$$   We will also show the following, which completely solves the question of containment and distinctness of these classes in the case of different indices.

\begin{theorem} Let $(\xi, p), (\zeta, r)\in \textbf{\emph{Ord}}\times (1, \infty]$ be distinct.  If $\mathfrak{I}$ is one of the four classes $\mathfrak{T}_{\xi,p}, \mathfrak{A}_{\xi,p}, \mathfrak{N}_{\xi,p}, \mathfrak{P}_{\xi,p}$ and if $\mathfrak{J}$ is one of the four classes $\mathfrak{T}_{\zeta, r}, \mathfrak{A}_{\zeta, r}, \mathfrak{N}_{\zeta, r}, \mathfrak{P}_{\zeta, r}$, then $\mathfrak{I}\subsetneq \mathfrak{J}$ if and only if $(\xi, 1/p)<(\zeta,1/r)$ in the lexicographical ordering.   

\end{theorem}

Beyond this result, it still remains to decide upon the distinctness of the four classes $\mathfrak{T}_{\xi,p}$, $\mathfrak{A}_{\xi,p}$, $\mathfrak{N}_{\xi,p}$, and $\mathfrak{P}_{\xi,p}$ for $\xi$ and $p$ fixed.   To that end, we have the following.

\begin{theorem} For any $1<p<\infty$ and any ordinal $\xi$, we have the following containments. \begin{align*} \mathfrak{D}_\xi & \subsetneq \bigcup_{p<r<\infty} \mathfrak{T}_{\xi,r} = \bigcup_{p<r<\infty} \mathfrak{P}_{\xi,r} \subsetneq \mathfrak{T}_{\xi,p} \subsetneq \mathfrak{N}_{\xi,p} \subsetneq \mathfrak{P}_{\xi,p} = \bigcap_{1<r<p} \mathfrak{T}_{\xi,r} = \bigcap_{1<r<p} \mathfrak{P}_{\xi,r} \subsetneq \mathfrak{D}_{\xi+1}. \end{align*}  

If $\xi$ has countable cofinality, and in particular if $\xi$ is countable, we also have the following containments. $$\mathfrak{T}_{\xi,p}\subsetneq \mathfrak{A}_{\xi,p} \subsetneq \mathfrak{N}_{\xi,p} .$$   

\label{tm1}
\end{theorem}

\begin{theorem} For any ordinal $\xi$,   we have the following containments. $$\mathfrak{D}_\xi  \subsetneq \mathfrak{T}_{\xi,\infty}\subsetneq  \mathfrak{P}_{\xi, \infty}  = \bigcap_{1<r<\infty} \mathfrak{T}_{\xi,r} = \bigcap_{1<r<\infty} \mathfrak{P}_{\xi,r} \subsetneq \mathfrak{D}_{\xi+1}.  $$

If $\xi$ has countable cofinality, and in particular if $\xi$ is countable, we also have the following containments.  \begin{align*}  \mathfrak{T}_{\xi,\infty}\subsetneq \mathfrak{A}_{\xi,\infty}= \mathfrak{N}_{\xi,\infty}\subsetneq \mathfrak{P}_{\xi, \infty}. \end{align*}

\label{tm2}

\end{theorem}

We have already shown all of the containments in Theorems \ref{tm1} and \ref{tm2}. It remains only to show that the indicated sets are not equal. We first argue the outermost differences.   It was argued in \cite{C4} that $C_0(\omega^{\omega^\xi})$, the space of all continuous, scalar-valued funtions $f$ on $\omega^{\omega^\xi}+=[0, \omega^{\omega^\xi}]$ such that $f(\omega^{\omega^\xi})=0$, is $\xi$-AUF. Brooker \cite{Brooker2} showed that  $Sz(C_0(\omega^{\omega^\xi}+))=\omega^{\xi+1}$, so $\mathfrak{D}_\xi\neq \mathfrak{T}_{\xi, \infty},$ and therefore $\mathfrak{D}_\xi$ does not contain any of the other classes mentioned in Theorems \ref{tm1} or \ref{tm2}.

Now for any $0<\xi$, in \cite{C3} it was shown that for any sequence $1> \vartheta_1>\vartheta_2\ldots$ with $\lim_n \vartheta_n=0$ and any sequence $(m_n)_{n=1}^\infty$ of natural numbers with $\sup_n m_n=\omega$, there exists a Banach space $X$ with $Sz(X)=\omega^{\xi+1}$, while for all $n\in\nn$, $Sz(X, \vartheta_n)>\omega^\xi m_n$.   An appropriate choice of $(\vartheta_n)_{n=1}^\infty$ and $(m_n)_{n=1}^\infty$, for example $\vartheta_n=\frac{1}{n+1}$ and $m_n=2^n$, we arrive at a Banach space in $\mathfrak{D}_{\xi+1}$ such that $Sz_\xi(X, \ee)$ has no non-trivial power type bound on its growth as $\ee\to 0$.  Thus for any $0<\xi$, we can exhibit some Banach space $X$ which lies in $\mathfrak{D}_{\xi+1}$ but not in any of the other classes in Theorems \ref{tm1} or \ref{tm2}.

This is in contrast to the $\xi=0$ case.  Lancien \cite{Lancien} showed that $Sz(X, \delta\ee)\leqslant Sz(X ,\delta)Sz(X, \ee)$ for any Banach space and any $0<\delta, \ee<1$, and therefore any Banach space with Szlenk index not more than $\omega$ lies in $\cup_{1<p<\infty} \mathfrak{P}_{0,p}$.   However, one can easily construct an operator which lies in $\mathfrak{D}_1$ but not $\cup_{1<p<\infty} \mathfrak{P}_{0,p}$ by taking the operator $A:(\oplus_{n=1}^\infty \ell_{1+1/n})_{\ell_2}\to (\oplus_{n=1}^\infty \ell_{1+1/n})_{\ell_2}$ such that $A|_{\ell_{1+1/n}}\equiv \frac{1}{\log_2(n+1)}I_{\ell_{1+1/n}}$.  Since $\ell_{1+1/n}\in \textsf{D}_1$ and $\mathfrak{D}_1$ is a closed ideal, we  deduce that $A\in \mathfrak{D}_1$. Moreover, for each $n\in\nn$,  $$\theta_{0, n}(\ell_{1+1/n})=n^{-\frac{1}{n+1}}\underset{n\to \infty}{\to}1.$$  From this we deduce the existence of some $c\in (0,1)$ such that $\theta_{0,n}(A) \geqslant c/\log_2(n+1)$ for all $n\in\nn$, and $A$ does not lie in $\cup_{1<p<\infty}\mathfrak{P}_{0,p}$.

We will recall some important Banach spaces for our examples below, but first we must recall some notation. We let $(e_n)_{n=1}^\infty$ and $(e_n^*)_{n=1}^\infty$ denote the canonical $c_{00}$ basis and the corresponding coordinate functionals, respectively.  Given $x\in c_{00}$ and $E\subset \nn$, $Ex$ is the sequence in $c_{00}$ such that $e^*_n(Ex)= 1_E(n)e^*_n(x)$.  Given two subsets $E,F$ of $\nn$ and $n\in\nn$, we write $n\leqslant E$ to mean $E= \varnothing$ or $n\leqslant \min E$, and we write   $E<F$ to mean that either $E=\varnothing$, $F=\varnothing$, or $\max E<\min F$.  Let $\text{supp}(x)=\{n\in\nn: e^*_n(x)\neq 0\}$ and $x<y$ will denote that $\text{supp}(x)<\text{supp}(y)$.  For $n\in\nn$ and $x\in c_{00}$, $n\leqslant x$ will denote that $n\leqslant \text{supp}(x)$.

We say that a (possibly uncountable, unordered) collection $(e_\gamma:\gamma\in \Gamma)$ in a Banach space $E$ is a $1$-\emph{unconditional basis} for $E$ if $(e_\gamma:\gamma\in \Gamma)$ has dense span in $E$ and for any finite subset $F$ of $\Gamma$, any scalars $(a_\gamma)_{\gamma\in F}$, and any unimomdular scalars $(\ee_\gamma)_{\gamma\in \Gamma}$, $$\|\sum_{\gamma\in F} a_\gamma e_\gamma\|= \|\sum_{\gamma\in F} \ee_\gamma a_\gamma e_\gamma\|.$$   We recall that for a Banach space $E$ with $1$-unconditional basis $(e_\gamma:\gamma\in \Gamma)$ and for a collection of Banach spaces $X_\gamma$, $\gamma\in \Gamma$, the direct sum $(\oplus_{\gamma\in\Gamma} X_\gamma)_E$ is the Banach space consisting of all $(x_\gamma)_{\gamma\in \Gamma}\in \prod_{\gamma\in \Gamma} X_\gamma$ such that $\sum_{\gamma\in \Gamma}\|x_\gamma\|_{X_\gamma}e_\gamma\in E$, endowed with the norm $$\|(x_\gamma)_{\gamma\in \Gamma}\|_{(\oplus_{\gamma\in \Gamma}X_\gamma)_E}=\|\sum_{\gamma\in \Gamma}\|x_\gamma\|_{X_\gamma}e_\gamma\|_E.$$   For $1<r<\infty$, we say $E$ satisfies an upper $\ell_r$ estimate provided there exists a constant $C$ such that for any disjointly supported vectors $x_1, \ldots, x_n\in E$, $$\|\sum_{i=1}^n x_i\|_E\leqslant C\bigl(\sum_{i=1}^r \|x_i\|_E^r\bigr)^{1/r}.$$    Formally speaking, this terminology should reference the particular basis rather than simply the space $E$, but each of the spaces we consider below will have a specific, canonical basis, so no confusion will arise.

 If $(e_i)_{i=1}^\infty$ is a Schauder basis for the Banach space $E$, we say $E$ satisfies an $\ell_r$-\emph{upper block estimate} provided that there exists a constant $C$ such that  for any successively supported vectors $x_1, \ldots, x_n$ in $E$, $$\|\sum_{i=1}^n x_i\|\leqslant C\bigl(\sum_{i=1}^n \|x_i\|^r\bigr)^{1/r}.$$   As in the last paragraph, this terminology should reference the specific basis of $E$, but this terminology will cause no confusion.   We recall that for $1\leqslant p\leqslant \infty$, a Schauder basis $(e_i)_{i=1}^\infty$ is said to be \emph{asymptotic} $\ell_p$ (resp. $c_0$ if $p=\infty$) in $E$ provided that there exists a constant $C\geqslant 1$ such that for any $n\in\nn$ and $n\leqslant x_1<\ldots <x_n$, $$C^{-1}\|\sum_{i=1}^n \|x_i\|_Ee_i\|_{\ell_p^n} \leqslant \|\sum_{i=1}^n x_i\|_E \leqslant C\|\sum_{i=1}^n \|x_i\|_E e_i\|_{\ell_p^n}.$$  Once more, we will say a Banach space $E$ is asymptotic $\ell_p$ (or $c_0$) provided that the canonical basis of that space is asymptotic $\ell_p$ (resp. $c_0$) in $E$. 

We recall that a basis (either an unordered, possibly uncountable $1$-unconditional basis or a Schauder basis) is called \emph{shrinking} if any bounded, coordinate-wise null net is weakly null.  This is equivalent to the coordinate functionals having dense span in the dual space.

\begin{lemma} \begin{enumerate}[(i)]\item Suppose $1<r,s<\infty$ are such that $1/r+1/s=1$, $(m_n)_{n=1}^\infty$ is a sequence of natural numbers such that $\sum_{j=1}^\infty 2^{-j} m_j^{1/s}<\infty$,  $E$ is a Banach space with $1$-unconditional basis such that for any $n\in\nn$, if $x_1, \ldots, x_k\subset B_E$ are disjointly supported and $x^*\in B_{E^*}$ is such that $\text{\emph{Re\ }}x^*(x_i)\geqslant 2^{-n}$ for all $1\leqslant i\leqslant k$, then $k\leqslant m_n$.  Then $E$ satisfies an upper $\ell_r$  estimate. \item Suppose $1<r,s<\infty$ are such that $1/r+1/s=1$, $(m_n)_{n=1}^\infty$ is a sequence of natural numbers such that $\sum_{j=1}^\infty 2^{-j} m_j^{1/s}<\infty$,  $E$ is a Banach space with $1$-unconditional basis $(e_i)_{i=1}^\infty$ such that for any $n\in\nn$, if $x_1, \ldots, x_k\subset B_E$ are successively supported, and $x^*\in B_{E^*}$ is such that $\text{\emph{Re\ }}x^*(x_i)\geqslant 2^{-n}$ for all $1\leqslant i\leqslant k$, then $k\leqslant m_n$.  Then $E$ satisfies an  $\ell_r$ upper block  estimate.  \item If $1<r<\infty$ and $E$ is a Banach space with $1$-unconditional basis satisfying an upper $\ell_r$ estimate (resp. $\ell_r$ upper block estiate), then there exists an equivalent norm $|\cdot|$ on $E$ such that $(E, |\cdot|)$ is $r$-AUS and the basis of $E$ is $1$-unconditional in $(E, |\cdot|)$. \item If $m_n=\min \{k\in\nn: 2^{-n}k>\log_2(k+1)\}$, then for any $1<s<\infty$, $\sum_{n=1}^\infty 2^{-n}m_n^{1/s}<\infty$. \end{enumerate}

\label{tech}

\end{lemma}

\begin{proof}$(i)$ Let $C=\sum_{j=1}^\infty 2^{1-j} m_j^{1/s}$.  Fix a sequence $(x_i)_{i=1}^n\subset B_E$ of disjointly supported vectors, scalars $(a_i)_{i=1}^n$, and $x^*\in B_{E^*}$ such that $$x^*(\sum_{i=1}^n a_i x_i)=\|\sum_{i=1}^n a_ix_i\|.$$  For each $j\in\nn$, let $$T_j=\{i\leqslant n: |x^*(x_i)|\in (2^{-j}, 2^{1-j}]\}.$$    By hypothesis, $|T_j|\leqslant m_j$ for all $j\in\nn$, whence $y^*=\sum_{j=1}^\infty \sum_{i\in T_j} 2^{1-j}e_i^*\in \ell_s$ satisfies $\|x^*\|_{\ell_s}\leqslant C$.     Moreover, \begin{align*} \|\sum_{i=1}^n a_ix_i\| & = x^*(\sum_{i=1}^n a_i x_i) \leqslant \sum_{j=1}^\infty \sum_{i\in T_j}2^{1-j}|a_i| \\ & = y^*\bigl(\sum_{i=1}^n a_i e_i\bigr) \leqslant C\bigl(\sum_{i=1}^n |a_i|^r\bigr)^{1/r}. \end{align*}    

$(ii)$ This is similar to $(i)$.

$(iii)$ Define the two quantities $[\cdot]$ and $|\cdot|$ on the span of $(e_\gamma: \gamma\in \Gamma)$ by $$[x]=\inf\Bigl\{\bigl(\sum_{i=1}^n \|x_i\|_E^r\bigr)^{1/r}: n\in\nn, x=\sum_{i=1}^n x_i, x_1, \ldots x_n\text{\ disjointly supported}\Bigr\}$$ and $$|x|=\Bigl\{\sum_{i=1}^n [x_i]: n\in\nn, x=\sum_{i=1}^n x_i\Bigr\}.$$  Then $|\cdot|$ extends to an equivalent norm on $E$ satisfying $|x+y|^r\leqslant |x|^r+|y|^r$ for any disjointly supported $x,y$, and $(e_\gamma:\gamma\in \Gamma)$ is still $1$-unconditional with respect to the norm $|\cdot|$.  More precisely, if  $C$ is such that $\|\sum_{i=1}^n x_i\|_E \leqslant C\bigl(\sum_{i=1}^n \|x_i\|_E^r\bigr)^{1/r}$ for any disjointly supported $x_1, \ldots, x_n$, then $|\cdot|\leqslant \|\cdot\|\leqslant C|\cdot|$.  It is obvious that $(e_\gamma:\gamma\in \Gamma)$ is still $1$-unconditional with respect to $|\cdot|$.  To see that $|x+y|^r \leqslant |x|^r+|y|^r$ for disjointly supported vectors, first suppose that $x,y$ have finite, disjoint supports.  Let $E,F$ be the supports of $x$ and $y$, respectively.    Write $x=\sum_{i=1}^n x_i$ for disjointly supported $x_1, \ldots, x_n$.  Note that $\|Ex_i\|\leqslant \|x_i\|$, so that $$x=Ex=\sum_{i=1}^n Ex_i$$ and $$\bigl(\sum_{i=1}^n \|Ex_i\|^r\bigr)^{1/r} \leqslant \bigl(\sum_{i=1}^n \|x_i\|^r\bigr)^{1/r}.$$   From this it follows that to compute $[x]$, it is sufficient to take the infimum of $\bigl(\sum_{i=1}^n \|x_i\|^r\bigr)^{1/r}$ for disjointly supported $x_1, \ldots, x_n$ with $\text{supp}(x_i)\subset E$.   A similar argument holds for $y$.  Furthermore, by compactness of the order interval $[0, |x|]$ in the lattice structure of $E$, $[x]$ is a minimum.  Thus if $$[x]=\bigl(\sum_{i=1}^n \|x_i\|^r\bigr)^{1/r}$$ and $$[y]=\bigl(\sum_{i=1}^m \|y_i\|^r\bigr)^{1/r}$$ for some $x_i$, $y_i$ with $\text{supp}(x_i)\subset E$ and $\text{supp}(y_i)\subset F$, then $$|x+y|^r\leqslant [x+y]^r\leqslant \sum_{i=1}^n \|x_i\|^r_E+\sum_{i=1}^m \|y_i\|_E^r\leqslant [x]^r+[y]^r.$$   This argument yields that for any finite, disjoint subsets $E,F$ of $\Gamma$ and any non-negative numbers $a,b$, $$\{u+v: u\in \text{span}(e_\gamma:\gamma\in E), [u]\leqslant a, v\in \text{span}(e_\gamma:\gamma\in F), [v]\leqslant b\} \subset  (a^r+b^r)^{1/r}B_E^{|\cdot|}.$$

Now we note that $|x|$ can similarly be computed by $$|x|=\inf\{\sum_{i=1}^n [x_i]: n\in \nn, x=\sum_{i=1}^n x_i, \text{supp}(x_i)\subset E\},$$ since for any $n\in\nn$ and $x_1, \ldots, x_n$ with $x=\sum_{i=1}^n x_i$, $x=Ex=\sum_{i=1}^n Ex_i$ and $$\sum_{i=1}^n [Ex_i] \leqslant \sum_{i=1}^n [x_i].$$   From this, the previous paragraph, and standard arguments it follows that for $x,y$ with finite, disjoint supports,  $$x+y\in \overline{\text{co}}\{u+v: u\in \text{span}(e_\gamma:\gamma\in \text{supp}(x)), [u]\leqslant |x|, v\in \text{span}(e_\gamma:\gamma\in \text{supp}(y)), [v]\leqslant |y|\} \subset  (|x|^r+|y|^r)^{1/r}B_E^{|\cdot|} .$$

The case of block estimates is similar, except we define the quantity $$[x]=\inf\Bigl\{\bigl(\sum_{i=1}^n \|x_i\|_E^r\bigr)^{1/r}:n\in\nn, x_1<\ldots <x_n, x=\sum_{i=1}^n x_i\Bigr\}.$$

$(iv)$ Fix $0<\alpha <1$ such that $\frac{1+\alpha}{s}<1$. Let $\theta=\frac{1+\alpha}{s}-1<0$. Fix $n_0\in \nn$ such that $2^{\alpha n}>(1+\alpha)n+1$ for all $n\geqslant n_0$.   Then if $n\geqslant n_0$, $$2^{-n}2^{(1+\alpha)n} = 2^{\alpha n}> (1+\alpha)n+1 \geqslant \log_2(2^{(1+\alpha)n}+1),$$ so that $m_n\leqslant 2^{(1+\alpha)n}$ for all $n\geqslant n_0$.   Then  from this it follows that $m_n^{1/s}\leqslant 2^{\frac{1+\alpha}{s}n}$ for all $n\geqslant n_0$, and $2^{-n}m_n \leqslant (2^\theta)^n$ for all $n\geqslant n_0$.  Since $\theta<0$, $\sum_{n=n_0}^\infty 2^{\theta n}<\infty$.

\end{proof}

We  recall Schlumprecht space $S$ and modified Schlumprecht space, $SM$.  The space $S$ is the completion of $c_{00}$ with respect to the norm  $$\|x\|_S=\max\Biggl\{ \|x\|_{c_0}, \max\Bigl\{\frac{1}{\log_2(l+1)}\sum_{i=1}^l \|E_ix\|_S: 1<l\in \nn, E_1<\ldots<  E_l\Bigr\}\Biggr\}.$$$SM$ is the completion of $c_{00}$ with respect to the norm $$\|x\|_{SM}=\max\Biggl\{ \|x\|_{c_0}, \max\Bigl\{\frac{1}{\log_2(l+1)}\sum_{i=1}^l \|E_ix\|_{SM}: 1<l\in \nn, E_1,\ldots,  E_l\text{\ pairwise disjoint}\Bigr\}\Biggr\}.$$   For $1\leqslant q<\infty$, we let $S_q$ and  $SM_q$ denote the completions of $c_{00}$ with respect to the norms $\|x\|_{S_q}=\||x|^q\|_S^{1/q}$ and $\|x\|_{SM_q}= \||x|^q\|_{SM}^{1/q},$ respectively.  Here $|\sum_{i=1}^\infty a_i e_i|^q=\sum_{i=1}^\infty |a_i|^qe_i$ is the usual lattice operation. Note that $(e_i)_{i=1}^\infty$ is a normalized, $1$-subsymmetric basis for $S_q$ and a normalized,   $1$-symmetric basis for $SM_q$. For an infinite set $\Gamma$ and $1\leqslant q<\infty$, we let $\mathfrak{X}_q(\Gamma)$ denote the completion of $c_{00}(\Gamma)$ with respect to the norm $$\|\sum_{\gamma\in \Gamma}a_\gamma e_\gamma\|_{\mathfrak{X}_q(\Gamma)}= \sup\Bigl\{\|\sum_{i=1}^\infty a_{f(i)} e_i\|_{SM_q^*}: f:\nn\to \Gamma\text{\ is an injection}\Bigr\}.$$  We let $\mathfrak{X}_q$ denote $\mathfrak{X}_q(\nn)=[e^*_i:i\in\nn]\subset SM_q^*$.  Furthermore, by $1$-symmetry, for any sequence of distinct elements $(\gamma_i)_{i=1}^\infty$ in $\Gamma$, $(e_{\gamma_i})_{i=1}^\infty\subset \mathfrak{X}_q(\Gamma)$ is isometrically equivalent to $(e^*_i)_{i=1}^\infty\subset SM^*_q$.

Next, we recall a space introduced by Tzafriri.   Fix  such that $1/p+1/q=1$.  We then let $V_q$ (denoted by Tzafriri as $V_{1/3^{1/p}, q}$) be the completion of $c_{00}$ with respect to the norm  $$\|x\|_{V_q}=\max\Biggl\{\|x\|_{c_0}, \max\Bigl\{\frac{1}{3^{1/p} l^{1/p}}\sum_{i=1}^l \|E_ix\|_{V_q}: 1<l\in \nn, E_1, \ldots, E_l\text{\ pairwise disjoint}\Bigr\}\Biggr\}.$$    Note that $(e_i)_{i=1}^\infty$ is a normalized, $1$-symmetric basis for $V_q$.  For an infinite set $\Gamma$, we let $\mathfrak{Y}_q(\Gamma)$ denote the completion of $c_{00}(\Gamma)$ with respect to the norm $$\|\sum_{\gamma\in \Gamma} a_\gamma e_\gamma\|_{\mathfrak{Y}_q(\Gamma)} = \sup\Bigl\{\|\sum_{i=1}^\infty a_{f(i)} e_i\|_{V_q^*}: f:\nn\to \Gamma\text{\ is an injection}\Bigr\}.$$  We let $\mathfrak{Y}_q$ denote $\mathfrak{Y}_q(\nn)=[e^*_i:i\in\nn]\subset V_q^*$. We note that for any sequence $(\gamma_i)_{i=1}^\infty$ of distinct elements of $\Gamma$, $(e_{\gamma_i})_{i=1}^\infty$ is isometrically equivalent to $(e^*_i)_{i=1}^\infty\subset V_q^*$.     We also recall the Tirilman space $\text{Ti}_q$ defined in \cite{CS}, the norm of which is given by $$\|x\|_{\text{Ti}_q}=\max\Biggl\{\|x\|_{c_0}, \max\Bigl\{\frac{1}{3^{1/p} l^{1/p}}\sum_{i=1}^l \|E_ix\|_{V_q}: 1<l\in \nn, E_1< \ldots< E_l\Bigr\}\Biggr\}.$$  

Last, we recall the Figiel-Johnson Tsirelson space, $T$, which is the dual of the space constructed by Tsirelson. The space $T$ is the completion of $c_{00}$ with respect to the norm $$\|x\|_T= \max\Biggl\{\|x\|_{c_0}, \max \Bigl\{\frac{1}{2}\sum_{i=1}^l \|E_i x\|_T: 1<l\in\nn, l\leqslant E_1<\ldots <E_l\Bigr\}\Biggr\}.$$   We let $T_q$ denote the completion of $c_{00}$ with respect to the norm $$\|x\|_{T_q}= \||x|^q\|_T^{1/q}.$$   Note that $(e_i)_{i=1}^\infty$ is a normalized, $1$-unconditional basis for $T_q$.

We  recall the essential facts regarding these spaces. 

\begin{proposition} Fix $1\leqslant q<\infty$ and let $1<p<\infty$ be such that $1/p+1/q=1$. In what follows, each assertion about $V_q$ and $\text{\emph{Ti}}_q$ will include the implicit assumption that $1<q$.  \begin{enumerate}[(i)]\item For any $n\in\nn$, $\|\sum_{i=1}^n e_i\|_S=\|\sum_{i=1}^n e_i\|_{SM}=\frac{n}{\log_2(n+1)}$.

\item If $G$ is any of the spaces $S_q, SM_q$, $\text{\emph{Ti}}_q$, $V_q$, or $T_q$, the formal identity $I:\ell_q\to G$ is bounded and strictly singular.

  \item For any $l\in\nn$ and  disjointly (resp. successively) supported $x^*_1, \ldots , x^*_l\in [e^*_i:i\in\nn]\cap B_{SM^*}$ (resp. $B_{S^*}$, $\frac{1}{\log_2(l+1)}\sum_{i=1}^l x^*_i\in B_{SM^*}$ (resp. $B_{S^*}$).

	\item For any finite, non-empty subset $E$ of $\nn$, $\|\sum_{i\in E}e_i\|_{SM_q^*}, \|\sum_{i\in E}e_i\|_{S_q^*}\geqslant |E|^{1/p}\log_2(|E|+1)^{1/q}$. 
	
	\item For any $l\in\nn$ and any disjointly (resp. successively) supported $x^*_1, \ldots , x^*_l$, $x^*_i\in B_{[e^*_j:j\in\nn]}\subset B_{V_q^*}$ (resp. $B_{\text{\emph{Ti}}_q^*}$), $\sum_{i=1}^l x^*_i\in  3^{1/p} n^{1/p}B_{V^*_q}$ (resp. $\sum_{i=1}^l x^*_i \in  3^{1/p}n^{1/p} B_{\text{\emph{Ti}}_q^*}$).

	\item For any $l\in \nn$ and $l\leqslant x^*_1<\ldots <x^*_l$, $x^*_i\in T_q^*$, $$\bigl(\sum_{i=1}^l \|x_i^*\|_{T_q^*}^p\bigr)^{1/p} \leqslant \|\sum_{i=1}^l x^*_i\|_{T^*_q} \leqslant \bigl(2^{1/q}\sum_{i=1}^l \|x_i^*\|_{T^*_q}^p\bigr)^{1/p}$$ $$\Bigl(\text{resp.\ \ }\max_{1\leqslant i\leqslant l}\|x^*_i\|_{T^*} \leqslant \|\sum_{i=1}^l x^*_t\|_{T^*} \leqslant 2\max_{1\leqslant i\leqslant l}\|x^*_i\|_{T^*}\text{ \ \ if\ }p=\infty\Bigr).$$    \end{enumerate}

\label{1994}
\end{proposition}

\begin{proof}$(i)$ Schlumprecht showed that $\|\sum_{i=1}^n e_i\|_S=\frac{n}{\log_2(n+1)}$. The same proof holds for $SM$.

$(ii)$ For $S_q$, $SM_q$ and $T_q$, boundedness of $I$ follows from the fact that we have taken a $q$-convexification. Item $(i)$ yields that $I:\ell_q\to S_q$ and $I:\ell_q\to SM_q$ are strictly singular.  It is well known that $T$ contains no isomorph of $\ell_1$, so the $q$-convexification $T_q$ cannot contain an isomorph of $\ell_q$, from which strict singularity of $I:\ell_q\to T_q$ follows.    For $V_q$, boundedness and strict singularity of the formal identity $I:\ell_q\to V_q$ were shown by Tzafriri. It is evident that $\|\cdot\|_{\text{Ti}_q} \leqslant \|\cdot\|_{V_q}$, so the formal inclusion of $\ell_q$ into $\text{Ti}_q$ factors through the formal inclusion of $\ell_q$ into $V_q$, and is therefore strictly singular.

$(iii)$ This follows easily from the definition of the norms of $SM$ and $S$.

$(iv)$ It follows from $(i)$ that for any finite, non-empty subset $E$ of $\nn$ that $$\|\sum_{i\in E} e_i\|_{SM_q} = \frac{|E|^{1/q}}{\log_2(|E|+1)^{1/q}}.$$   From this it follows that $$\|\sum_{i\in E} e_i\|_{SM_q^*} \geqslant \frac{|E|\log_2(|E|+1)^{1/q}}{|E|^{1/q}} = |E|^{1/p}\log_2(|E|+1)^{1/q}.$$

$(v)$ and $(vi)$ follow immediately from the definitions.

\end{proof}

\begin{proposition} Fix $1<p\leqslant \infty$ and let $1/p+1/q=1$.  \begin{enumerate}[(i)]\item If $G$ is any of the spaces $S_q^*$, $\text{\emph{Ti}}_q^*$, $T^*_q$, then for any $1<r<p$, $G$ satisfies $\ell_r$ upper block estimates.   \item  For any $q<s<\infty$ and $1<r<p$, $[e^*_i:i\in\nn]\subset SM_q^*$ satisfies an upper $\ell_r$ estimate and $SM_q$ satisfies a lower $\ell_s$ estimate.   \end{enumerate}

\label{blocks}
\end{proposition}

\begin{proof}$(i)$  First fix $1<r<p$ and let $1/r+1/s=1$.     First note that by Proposition \ref{1994}$(v)$ and $(vi)$, if $\theta= 1/3^{1/p}$ if $G=\text{Ti}_q^*$ and $\theta=1/2^{q/p}$ if $G=T^*_q$, then for any $k\in\nn$ and any $k\leqslant x^*_1<\ldots <x^*_k$, $x^*_i\in B_G$, then $\frac{\theta}{k^{1/p}}\sum_{i=1}^k x_i^*\in B_G$.    Now for each $n\in\nn$, let $m_n= 2 \cdot 2^{qn}/\theta^q$. Note that $\sum_{n=1}^\infty 2^{-n} m_n^{1/s}<\infty$.      We will show that if $x^*_1<\ldots <x^*_k$, $x^*_i\in B_G$, and if $x\in B_{G^*}$ is such that $\text{Re\ }x^*_i(x)\geqslant 1/2^n$ for each $1\leqslant i\leqslant k$, then $k\leqslant m_n$.    Indeed, suppose that $k>m_n$ and $x^*_1<\ldots <x^*_k$, $x$ are as above.    If $k$ is even, let $j=k/2$ and otherwise let $j=\frac{k-1}{2}$.  In either case, let $l=k-j>2^{qn}/\theta^q$ and note that $l\leqslant x^*_{j+1}<\ldots <x^*_{j+l}$.     From this it follows that $$\frac{\theta}{l^{1/p}}\cdot \frac{1}{2^n}\cdot l \leqslant \Bigl(\frac{\theta}{l^{1/p}}\sum_{i=1}^l x^*_{j+i}\Bigr)(x) \leqslant 1.$$   From this it follows that $$l\leqslant 2^{qn}/\theta^q,$$ a contradiction.   We now appeal to Lemma \ref{tech}$(ii)$.

Now fix $1<r<\infty$ and let $1/r+1/s=1$.   Arguing as in the previous paragraph with $m_n$ as in Lemma \ref{tech}$(iv)$, we deduce that $S$ satisfies an $\ell_r$ upper block estimate.  By duality, $S$ satisfies an $\ell_s$ lower block estimate for every $1<s<\infty$.   From this it follows that $S_q$ satisfies an $\ell_s$ lower block estimate for every $q<s<\infty$, and using duality again we deduce that $S_q^*$ satisfies an upper $\ell_r$ estimate for every $1<r<p$.

$(ii)$  This is similar to the argument in $(i)$, using Lemma \ref{tech}$(i)$ in place of Lemma \ref{tech}$(ii)$.

\end{proof}

We are now prepared to prove Theorems \ref{tm1} and \ref{tm2} in the $\xi=0$ case. 

\begin{proof}[Proof of Theorems \ref{tm1}, \ref{tm2}, $\xi=0$ case]

First note that for $1<p<\infty$, $\ell_p\in \textsf{T}_{0,p}\setminus \cup_{p<r<\infty}\textsf{T}_{0,r}$.  The fact that $\ell_p\notin \cup_{p<r<\infty}\textsf{T}_{0,r}$ follows from the fact that $n\theta_{0,n}(\ell_p)= n^{1/p}$, and so $\ell_p\notin \textsf{N}_{0,r}$ for any $p<r<\infty$. Note that $c_0\in \textsf{T}_{0,\infty}$.

Now we argue that $T_q^*\in \textsf{A}_{0,p}\setminus \textsf{T}_{0,p}$.  The fact that $T^*_q\in \textsf{A}_{0,p}$ follows from Proposition \ref{1994}$(vi)$.   Since the canonical basis of $T^*_q$ is weakly null (which also follows from Proposition \ref{1994}$(vi)$, we can consider the weakly null tree $(x_t)_{t\in \nn^{<\nn}}\subset S_{T^*_q}$ given by $x_{(n_1, \ldots, n_k)}=e_{\sum_{i=1}^k n_i}$.  Then if $T^*_q\in \textsf{T}_{0,p}$, there must exist some subsequence of the $T^*_q$ basis which is dominated by the $\ell_p$ (resp. $c_0$ if $p=\infty$) basis.  By duality, this would yield the existence of a subsequence of the $T_q$ basis which dominates the $\ell_q$ basis.  Since this subsequence is also dominated by the $\ell_q$ basis, it must be equivalent.  But this contradicts the strict singularity of the formal inclusion $I:\ell_q\to T_q$.

Now we argue that if $1<q$, $V_q^*\supset [e^*_i:i\in\nn]\in \textsf{N}_{0,p}\setminus \textsf{A}_{0,p}$.    The fact that $[e^*_i:i\in\nn]\in \textsf{N}_{0,p}$ follows from Proposition \ref{1994}$(v)$.  Seeking a contradiction, suppose that $\alpha:=\alpha_{0,p}([e^*_i:i\in\nn])$ is finite.   Since the canonical basis is weakly null (which also follows from Proposition \ref{1994}$(v)$), we can using $1$-symmetry  to deduce that for any $n\in\nn$ and any scalar sequence $(a_i)_{i=1}^n\in B_{\ell_p^n}$, $$\|\sum_{i=1}^n a_i e_i^*\| = \inf \Bigl\{\|\sum_{i=1}^n a_i e_{\sum_{j=1}^i n_j}^*\|: (n_j)_{j=1}^n\in \nn^n\Bigr\} \leqslant \alpha.$$  From this we decuce that the sequence of coordinate functionals $(e^*_i)_{i=1}^\infty \subset V_q^*$ is dominated by the $\ell_p$ basis. By duality, the $V_q$ basis dominates, and is therefore equivalent to, the $\ell_q$ basis, contradicting the strict singularity of the formal inclusion $I:\ell_q\to V_q$.

Finally we argue that $SM_q^*\supset [e_i^*:i\in\nn]\in \textsf{P}_{0,p}\setminus \textsf{N}_{0,p}$. For the remainder of this paragraph, let $E_q$ denote the closed span in $SM_q^*$ of the canonical basis $(e^*_i)_{i=1}^\infty$.    Since the basis $(e^*_i)_{i=1}^\infty\subset SM_q^*$ is weakly null, we may argue as in the previous paragraph to deduce that for any $n\in\nn$ and any $(n_i)_{i=1}^n\in \nn^n$, $$\|\sum_{i=1}^n e_{\sum_{j=1}^i n_j}\|_{SM_q^*}/n^{1/p} = \log_2(n+1)^{1/q}\underset{n\to \infty}{\to}\infty,$$ whence $E_q\notin \mathfrak{N}_{0,p}$.      By Lemma \ref{tech},  $E_1\subset SM^*$ satisfies an upper $\ell_r$ estimate for any $1<r<\infty$. If $q=1$, then Lemma \ref{tech} yields that $E_1\in \cap_{1<r<\infty}\textsf{T}_{0,r}=\textsf{P}_{0,p}$.    Suppose now that $1<p<\infty$. By duality, $SM$ satisfies a lower $\ell_s$ estimate for every $1<s<\infty$.   From this it follows that  $SM_q$ satisfies a lower $\ell_s$ estimate for every $q<s<\infty$, and $E_q$ satisfies an upper $\ell_r$ estimate  for every $1<r<p$. Then Lemma \ref{tech} yields that $E_q\in \cap_{1<r<p}\textsf{T}_{0,r}=\textsf{P}_{0,p}$.  

\end{proof}

\begin{rem}\upshape Note that we could have used $S_q$ in place of $SM_q$ and $\text{Ti}_q$ in place of $V_q$ above, using the analogous statements from Propositions \ref{1994} and \ref{blocks}.

\end{rem}

\begin{lemma} Let $E$ be a Banach space with infinite, shrinking, $1$-unconditional basis $(e_\gamma:\gamma\in \Gamma)$.  Suppose that $\xi$ is an ordinal and for every $\gamma\in \Gamma$, $X_\gamma$ is a Banach space with $Sz(X_\gamma)\leqslant \omega^\xi$.   Then with $F=(\oplus_{\gamma\in \Gamma} X_\gamma)_E$, $$\|\cdot\|_{\xi, F}\leqslant \|\cdot\|_{0,E}.$$   

\label{envelope}

\end{lemma}

\begin{proof} In the proof, let us recall that if $t=(S_1, \ldots, S_k)$, $k\in\nn$, is some finite sequence, then $t^-=(S_1, \ldots, S_{k-1})$.

Let $\pi:F\to E$ be given by $\pi((x_\gamma)_{\gamma\in \Gamma})=\sum_{\gamma\in \Gamma} \|x_\gamma\|_{X_\gamma} e_\gamma$.    For each finite subset $G$ of $\Gamma$, let $P_G(x_\gamma)_{\gamma\in \Gamma}= (1_G(\gamma)x_\gamma)_{\gamma\in \Gamma}$.   Then since $Sz(X_\gamma)\leqslant \omega^\xi$, $Sz(P_G)\leqslant \omega^\xi$ for any finite subset $G$ of $\Gamma$.    From this it follows that for any directed set $D$, any $n\in\nn$, any weakly null collection $(x_t)_{t\in \Gamma_{\xi,n}.D}\subset B_F$,  any finite  subset $G$ of $\Gamma$, and any $\delta>0$, there exist $t\in MAX(\Lambda_{\xi,n,1}.D)$ (resp. for any $1\leqslant i<n$ and $u\in MAX(\Lambda_{\xi,n,i}.D)$, there exist $u<t\in MAX(\Lambda_{\xi, n, i+1}.D)$)  such that $\|P_G\sum_{s\leqslant t}\mathbb{P}_{\xi,n}(s)x_s\|<\ee$ (resp. $\|P_G\sum_{u<s\leqslant t}\mathbb{P}_{\xi,n}(s)x_s\|<\delta$).

Fix $n\in\nn$, a non-zero scalar sequence $(a_i)_{i=1}^n$, and $$\lambda<\|\sum_{i=1}^n a_i e_i\|_{\xi,F}.$$  Fix $\ee>0$ such that $$2\ee n \max_{1\leqslant i\leqslant n}|a_i|+\lambda< \|\sum_{i=1}^n a_i e_i\|_{\xi,F}.$$  Then we may fix a directed set $D$ and a weakly null collection $(x_t)_{t\in \Gamma_{\xi,n}.D}$ such that $$2\ee n \max_{1\leqslant i\leqslant n}|a_i|+ \lambda <\inf_{t\in MAX(\Gamma_{\xi,n}.D)} \|\sum_{i=1}^n a_i \sum_{t_{i-1}<s\leqslant t_i} \mathbb{P}_{\xi,n}(s)x_s\|,$$ where for $t\in MAX(\Gamma_{\xi,n}.D)$, $\varnothing=t_0<\ldots <t_n=t$ is the unique sequence such that $t_i\in MAX(\Lambda_{\xi,n,i}.D)$ for each $1\leqslant i\leqslant n$.      Let $\mathcal{F}$ denote the set of finite, non-empty subsets of $\Gamma$, directed by inclusion.   We then define a function $f:\mathcal{F}^{\leqslant n}\to \cup_{i=1}^n MAX(\Lambda_{\xi,n,i}.D)$, sets $(T_t)_{t\in \mathcal{F}^{\leqslant n}}\subset \Gamma$,   such that, with $T_\varnothing=\varnothing$ and $f(\varnothing)=\varnothing$,  \begin{enumerate}[(i)]\item if $t\in \mathcal{F}^k$, $f(t)\in MAX(\Lambda_{\xi,n,k}.D)$, \item if $s,t\in \mathcal{F}^{\leqslant n}$ and $s<t$, then $f(s)<f(t)$, \item if $t=(S_1, \ldots, S_k)\in \mathcal{F}^{\leqslant n}$, then $$\|P_{T_{t^-}\cup S_k} \sum_{f(t^-)<s\leqslant f(t)} \mathbb{P}_{\xi,n}(s) x_s\| <\ee,$$  \item if $s,t\in \mathcal{F}^{\leqslant n}$ and $s<t$, $T_s\subset T_t$, \item for any $t\in \mathcal{F}^{\leqslant n}$, $$\|P_{\Gamma\setminus T_t} \sum_{f(t^-)<s\leqslant f(t)} \mathbb{P}_{\xi,n}(s)x_s\|<\ee.$$   \end{enumerate} We define $f(t)$, $T_t$ for $t\in \mathcal{F}^{\leqslant n}$ under the assumption that $f(s), T_s$ have already been defined and have the stated properties for each $\varnothing\leqslant s<t$.    Write $t=(S_1, \ldots, S_k)\in \mathcal{F}^k$ and let $G=T_{t^-}\cup S_k$, $\delta=\ee$.    By hypothesis, $f(t^-)=\varnothing$ if $k=1$ and $f(t^-)\in MAX(\Lambda_{\xi, n, k-1}.D)$ otherwise. By the previous paragraph, we may fix $f(t)\in MAX(\Lambda_{\xi, n,k}.D)$ such that $f(t^-)<f(t)$ and $$\|P_G\sum_{f(t^-)<s\leqslant f(t)} \mathbb{P}_{\xi,n}(s)x_s\|<\delta.$$   We then fix $T_{t^-}\subset T_t\subset \Gamma$ finite such that $$\|P_{\Gamma\setminus T_t} \sum_{f(t^-)<s\leqslant f(t)} \mathbb{P}_{\xi,n}(s)x_s\|<\ee.$$  This completes the recursive construction.

Now let $y_t=\pi P_{T_t\setminus T_{t^-}}\sum_{f(t^-)<s\leqslant f(t)} \mathbb{P}_{\xi,n}(s)x_s\in B_E$.   Since the basis is shrinking and $P_{S_k}y_{(S_1, \ldots, S_k)}=0$, $(y_t)_{t\in \mathcal{F}^{\leqslant n}}$ is weakly null and for any $t\in \mathcal{F}^n$, \begin{align*} \|\sum_{i=1}^n a_iy_{t|_i} \| & = \|\sum_{i=1}^n a_iP_{T_t\setminus T_{t^-}}\sum_{f(t^-)<s\leqslant f(t)} \mathbb{P}_{\xi,n}(s) x_s\| \\ & \geqslant \|\sum_{i=1}^n a_i\sum_{f(t^-)<s\leqslant f(t)} \mathbb{P}_{\xi,n}(s) x_s\| -2\ee n \max_{1\leqslant i\leqslant n}|a_i| >\lambda. \end{align*} From this it follows that $\|\sum_{i=1}^n a_i e_i\|_{0,E}\geqslant \lambda$.

\end{proof}

\begin{corollary} Let $\xi$ be an ordinal and suppose $1\leqslant q<\infty$ and $1/p+1/q=1$.     Let $\Gamma$ be an infinite set and suppose that for every $\gamma\in \Gamma$, $X_\gamma$ is a Banach space with $Sz(X_\gamma)\leqslant \omega^\xi$.    \begin{enumerate}[(i)]\item If $p<\infty$, $(\oplus_{\gamma\in \Gamma} X_\gamma)_{\ell_p(\Gamma)}\in \textsf{\emph{T}}_{\xi,p}$, and  $(\oplus_{\gamma\in \Gamma} X_\gamma)_{c_0(\Gamma)} \in \textsf{\emph{T}}_{\xi,\infty}$. \item If $\Gamma$ is countable, say $\Gamma=(\gamma_n)_{n=1}^\infty$, then $(\oplus_{n=1}^\infty X_{\gamma_n})_{T_q^*}\in \textsf{\emph{A}}_{\xi,p}$.  \item If $1<q$, $(\oplus_{\gamma\in \Gamma}X_\gamma)_{\mathfrak{Y}_q(\Gamma)}\in \textsf{\emph{N}}_{\xi,p}$. \item $(\oplus_{\gamma\in \Gamma}X_\gamma)_{\textsf{X}_q(\Gamma)}\in \textsf{\emph{P}}_{\xi,p}$.    \end{enumerate}

\label{one side}
\end{corollary}

\begin{proof} We note that the canonical basis of each space $\ell_p(\Gamma)$, $1<p<\infty$, $c_0(\Gamma)$, $T^*_q$, $\mathfrak{X}_q(\Gamma)$, $\mathfrak{Y}_q(\Gamma)$ is shrinking. More precisely, the basis of each space except $T_q^*$ basis has some form of $\ell_p$ or $c_0$ upper estimate on flat linear combinations of vectors which are  disjointly supported, and $T_q^*$ has an asymptotic $\ell_p$ (resp. $c_0$ if $p=\infty$) basis.    

Fix $\sigma>0$, a finite subset $G$ of $\Gamma$, and $y=(y_\gamma)_{\gamma\in \Gamma}\in (\oplus_{\gamma\in \Gamma}X_\gamma)_{\ell_p(\Gamma)}$ such that $y_\lambda=0$ for all $\lambda\in \Gamma\setminus G$ and a weakly null $(x_t)_{t\in \Gamma_{\xi,1}.D}\subset \sigma B_{(\oplus_{\gamma\in\Gamma}X_\gamma)_{\ell_p(\Gamma)}}$.   Then since $Sz(P_G)\leqslant \omega^\xi$,  $$\inf_{t\in MAX(\Gamma_{\xi,1}.D)} \|P_G\sum_{s\leqslant t}\mathbb{P}_{\xi,1}(s)x_s\|=0.$$  Therefore for $p<\infty$,  \begin{align*} \inf_{t\in MAX(\Gamma_{\xi,1}.D)} \|y+\sum_{s\leqslant t} \mathbb{P}_{\xi,1}(s)x_s\|^p & \leqslant \inf_{t\in MAX(\Gamma_{\xi,1}.D)} \|y+P_G\sum_{s\leqslant t}\mathbb{P}_{\xi,1}(s)x_s\|^p + \|P_{\nn\setminus G}\sum_{s\leqslant t}\mathbb{P}_{\xi,1}(s)x_s\|^p \\ & \leqslant \|y\|^p+\sigma^p.  \end{align*}  Since $G$ was arbitrary and the set of all $y=(y_\gamma)_{\gamma\in \Gamma}$ such that $\{\gamma: y_\gamma\neq 0\}$ is finite is dense in $(\oplus_{\gamma\in \Gamma} X_\gamma)_{\ell_p(\Gamma)}$, we deduce that $\varrho_\xi(\sigma;(\oplus_{\gamma\in \Gamma}X_\gamma)_{\ell_p(\Gamma)}) \leqslant (1+\sigma^p)^{1/p}-1$ for all $\sigma>0$.     The proof for $c_0(\Gamma)$ is similar, and we deduce that $\varrho_\xi(1;(\oplus_{\gamma\in \Gamma}X_\gamma)_{c_0(\Gamma)})=0$.

Note that by Lemma \ref{envelope} and by the properties of $T^*_q$, there exists a constant $C_q$ such that  $$\|\cdot\|_{\xi, (\oplus_{n=1}^\infty X_{\gamma_n})_{T_q^*}}\leqslant \|\cdot\|_{0, T_q^*}\leqslant C_q\|\cdot\|_{\ell_p},$$ so $(\oplus_{n=1}^\infty X_{\gamma_n})_{T_q^*}\in \mathfrak{A}_{\xi,p}$.

Similarly, if $1<q$, $$\|\sum_{i=1}^n e_i\|_{\xi, (\oplus_{\gamma\in \Gamma} X_\gamma)_{\mathfrak{Y}_q(\Gamma)}}\leqslant \|\sum_{i=1}^n e_i\|_{0, \mathfrak{Y}_q(\Gamma)} \leqslant 3^{1/p}n^{1/p},$$   whence $(\oplus_{\gamma\in \Gamma} X_\gamma)_{\mathfrak{Y}_q(\Gamma)}\in \mathfrak{N}_{\xi,p}$.   Here we are using the fact that any disjointly supported vectors $x_1, \ldots, x_n\in B_{\mathfrak{Y}_q(\Gamma)}$ satisfy $\frac{1}{3^{1/p}n^{1/p}}\sum_{i=1}^n x_i\in B_{\mathfrak{Y}_q(\Gamma)}$, because these vectors are isometrically equivalent to some disjointly supported vectors in $[e^*_i:i\in\nn]\subset V_q^*$.

Finally, by Lemma \ref{tech} and Proposition \ref{blocks}, for any $1<r<p$, there exists an equivalent norm $|\cdot|_r$ on $\mathfrak{X}_q(\Gamma)$ such that $|x+y|^r_r\leqslant |x|^r_r+|y|^r_r$ for any disjointly supported $x,y\in \mathfrak{X}_q(\Gamma)$, and such that the canonical $\mathfrak{X}_q(\Gamma)$ basis is $1$-unconditional with respect to the norm $|\cdot|_r$.  Here we are using the fact that a finite sequence of disjointly supported vectors in $\mathfrak{X}_q(\Gamma)$ is isometrically equivalent to a finite sequence of disjointly supported vectors in $[e_i^*:i\in\nn]\subset SM_q^*$, whence $\mathfrak{X}_q(\Gamma)$ satisfies an upper $\ell_r$ estimate for every $1<r<p$.   Then arguing as in the second paragraph, we deduce that, if $E=(\mathfrak{X}_q(\Gamma), |\cdot|_r)$, $$\varrho_\xi(\sigma, (\oplus_{\gamma\in \Gamma}X_\gamma)_E)\leqslant (1+\sigma^r)^{1/r}-1.$$  Since $(\oplus_{\gamma\in \Gamma}X_\gamma)_E$ is isomorphic to $(\oplus_{\gamma\in \Gamma} X_\gamma)_{\mathfrak{X}_q(\Gamma)}$ and since $1<r<p$ was arbitrary, $$(\oplus_{\gamma\in \Gamma}X_\gamma)_{\mathfrak{X}_q(\Gamma)}\in \bigcap_{1<r<p} \textsf{T}_{\xi,r}=\textsf{P}_{\xi,p}.$$

\end{proof}

In what follows, we use the notation that for any $0<\xi$, $\xi=[0, \xi)$.

\begin{lemma} Let $\xi$ be an ordinal. \begin{enumerate}[(i)]\item Suppose that $\xi=\zeta+1$ and  $(X_n)_{n=1}^\infty$ is a sequence of Banach spaces such that $\theta_{\zeta,n}(X_n)=1$ for all $n\in\nn$.  Let $E$ be a Banach space with normalized, $1$-unconditional basis and let $D$ be a weak neighborhood basis at $0$ in $F:=(\oplus_{n=1}^\infty X_n)_E$.  Then for any finite subset $J$ of $\nn$ and any $0<\vartheta<1$, there exists a weakly null collection $(x_t)_{t\in\Gamma_{\xi,1}.D}\subset B_F$ such that $$\vartheta \leqslant\inf_{t\in MAX(\Gamma_{\xi,1}.D)} \|\sum_{s\leqslant t}\mathbb{P}_{\xi,1}(s)x_s\| $$ and for any $t\in MAX(\Gamma_{\xi,1}.D)$, there exists $n_t\in \nn\setminus J$ such that $x_s\in X_{n_t}$ for all $s\leqslant t$. 

In particular, for any $0<\vartheta<1$, there exists $(x_t)_{t\in \Gamma_{\xi, \infty}.D}\subset B_F$ such that for any $\varnothing=t_0<t_1<\ldots$, $t_i\in MAX(\Lambda_{\xi, \infty,i}.D)$, there exist $n_1<n_2<\ldots$ such that $\Bigl(\sum_{t_{i-1}<s\leqslant t_i}\mathbb{P}_{\xi, \infty}(s)x_s\Bigr)_{i=1}^\infty$ is $\vartheta^{-1}$-equivalent to $(e_{n_i})_{i=1}^\infty$.   Furthermore, for any $n\in\nn$, there exists $(x_t)_{t\in \Gamma_{\xi, n}.D}\subset B_F$ such that for any $\varnothing=t_0<t_1<\ldots<t_n$, $t_i\in MAX(\Lambda_{\xi, \infty,i}.D)$, there exist $k_1<k_2<\ldots<k_n$ such that $\Bigl(\sum_{t_{i-1}<s\leqslant t_i}\mathbb{P}_{\xi, n}(s)x_s\Bigr)_{i=1}^n$ is $\vartheta^{-1}$-equivalent to $(e_{k_i})_{i=1}^n$.

\item Suppose that $\Gamma \subset \xi$ is a set with $\sup \Gamma=\xi$ such that $\theta_{\gamma, 1}(X_\gamma)=1$ for each $\gamma\in \Gamma$.   Let $E$ be a Banach space with normalized, $1$-unconditional basis $(e_\gamma: \gamma\in \Gamma)$ and let $D$ be a weak neighborhood basis at $0$ in $F:=(\oplus_{\gamma\in \Gamma} X_\gamma)_E$.    Then for any finite subset $J$ of $\xi$ and any $0<\vartheta<1$, there exists a weakly null collection $(x_t)_{t\in \Gamma_{\xi,1}.D}\subset B_F$ such that $$\vartheta\leqslant \inf_{t\in MAX(\Gamma_{\xi,1}.D)} \|\sum_{s\leqslant t}\mathbb{P}_{\xi,1}(s)x_s\|$$ and for any $t\in MAX(\Gamma_{\xi,1}.D)$, there exists $\gamma_t\in \Gamma\setminus J$ such that $x_s\in X_{\gamma_t}$ for all $s\leqslant t$.    

In particular, for any $0<\vartheta<1$, there exists $(x_t)_{t\in \Gamma_{\xi, \infty}.D}\subset B_F$ such that for any $\varnothing=t_0<t_1<\ldots$, $t_i\in MAX(\Lambda_{\xi, \infty,i}.D)$, there exists $\gamma_1<\gamma_2<\ldots$, $\gamma_i\in \Gamma$, such that $\Bigl(\sum_{t_{i-1}<s\leqslant t_i}\mathbb{P}_{\xi, \infty}(s)x_s\Bigr)_{i=1}^\infty$ is $\vartheta^{-1}$-equivalent to $(e_{\gamma_i})_{i=1}^\infty$.   Furthermore, for any $n\in\nn$, there exists $(x_t)_{t\in \Gamma_{\xi, n}.D}\subset B_F$ such that for any $\varnothing=t_0<t_1<\ldots<t_n$, $t_i\in MAX(\Lambda_{\xi, \infty,i}.D)$, there exists $\gamma_1<\gamma_2<\ldots<\gamma_n$, $\gamma_i\in\Gamma$, such that $\Bigl(\sum_{t_{i-1}<s\leqslant t_i}\mathbb{P}_{\xi, n}(s)x_s\Bigr)_{i=1}^n$ is $\vartheta^{-1}$-equivalent to $(e_{\gamma_i})_{i=1}^n$.

\end{enumerate}
\label{ravenous}
\end{lemma}

\begin{proof}$(i)$ Let $m=\max J$ if $J\neq \varnothing$, and otherwise let $m\in\nn$ be arbitrary.    For each $n\in\nn$, we may fix $k_n\geqslant \max\{1+m, n\}$.  We identify $X_{k_n}$ in the obvious way with a subspace of $F$.  Since $1=\theta_{\zeta,k_n}(X_{k_n})\leqslant \theta_{\zeta, n}(X_{k_n})$, there exists a weakly null collection $(x_t)_{t\in \Gamma_{\zeta, n}.D}\subset B_{X_{k_n}}\subset B_F$ such that $$\vartheta\leqslant \inf_{t\in MAX(\Gamma_{\xi,n}.D)} \|n^{-1}\sum_{s\leqslant t}\mathbb{P}_{\xi,n}(s)x_s\|.$$     Taking the union of these collections gives the first statement, since $\Gamma_{\xi,1}.D=\cup_{n=1}^\infty \Gamma_{\zeta, n}.D$.

For the second statement, we define a collection $(x_t)_{t\in \Gamma_{\xi,\infty}.D}\subset B_F$ for $t\in \Lambda_{\xi, \infty ,n}.D$ by induction on $n$.   Since $\Lambda_{\xi,\infty, 1}.D=\Gamma_{\xi,1}.D$, we use the previous paragraph with $J=\varnothing$. Now suppose that $(x_t)_{t\in \cup_{i=1}^k \Lambda_{\xi, \infty, i}.D}$ has been defined in a such a way that for each $t\in MAX(\Lambda_{\xi, \infty, k}.D)$, if $\varnothing=t_0<\ldots <t_k$, $t_i\in MAX(\Lambda_{\xi, \infty, i}.D)$, then there exist $n_1<\ldots <n_k$ such that $x_s\in X_{n_i}$ for each $1\leqslant i\leqslant k$ and $t_{i-1}<s\leqslant t_i$.    Now fix $t\in MAX(\Lambda_{\xi, \infty, k+1}.D)$.  By identifying $U_t:=\{s\in \Lambda_{\xi, \infty, k+1}: t<s\}$ with $\Gamma_{\xi, 1}.D$ and applying the previous paragraph with $J=\{1, \ldots, n_k\}$, we may select $(x_s)_{s\in U_t}\subset B_F$ such that for each $u\in MAX(\Lambda_{\xi, \infty,k+1}.D)$ with $t<u$, there exists $n_u>n_k$ such that $x_s\in X_{n_u}$ for all $t<s\leqslant u$.    This completes the definition of $(x_t)_{t\in \Gamma_{\xi, \infty}.D}\subset B_F$.

Now for any $\varnothing=t_0<t_1<\ldots$ with $t_i\in MAX(\Gamma_{\xi, \infty}.D)$, there exist $n_1<n_2<\ldots $ such that $x_s\in X_{n_i}$ for each $t_{i-1}<s\leqslant t_i$.   Let $$w_i=\|\sum_{t_{i-1}<s\leqslant t_i} \mathbb{P}_{\xi, \infty}(s)x_s\|\in [\vartheta, 1].$$   Then $\Bigl(\sum_{i=1}^\infty \sum_{t_{i-1}<s\leqslant t_i} \mathbb{P}_{\xi, \infty}(s)x_s\Bigr)_{i=1}^\infty$ is isometrically equivalent to $(w_i e_{n_i})_{i=1}^\infty$, and by $1$-unconditionality, $\vartheta^{-1}$ equivalent to $(e_{n_i})_{i=1}^\infty$. 

The proof for $\Gamma_{\xi,n}.D$ is similar.

$(ii)$ Let $\gamma=\max J<\xi$ if $J\neq \varnothing$, and otherwise let $\gamma$ be arbitrary.  Now for any $\zeta<\xi$, since $\sup \Gamma=\xi$, there exists $\beta\in \Gamma$ such that $\gamma, \zeta<\beta$.      Now since $1=\theta_{\beta, 1}(X_\beta)\leqslant \theta_{\zeta+1, 1}(X_\beta)$, using the canonical identification of $(\omega^\zeta+\Gamma_{\zeta+1, 1}).D$ with $\Gamma_{\zeta+1, 1}.D$, there exists a weakly null collection $(x_t)_{t\in (\omega^\zeta+\Gamma_{\zeta+1, 1}).D}\subset B_{X_\beta}\subset B_F$ such that $$\vartheta\leqslant \inf_{t\in MAX((\omega^\zeta+\Gamma_{\zeta+1, 1}).D} \|\sum_{s\leqslant t} \mathbb{P}_{\xi, 1}(s)x_s\|.$$    Taking the union over all $\zeta<\xi$ gives the first part.    The last two statements follow as in $(i)$.

\end{proof}

\begin{proposition} Suppose that $\xi=\zeta+1$ and that $(X_n)_{n=1}^\infty$ is a sequence of Banach spaces such that $\theta_{\zeta, n}(X_n)=1$ for all $n\in\nn$.  Fix $1\leqslant q<\infty$ and let $1/p+1/q=1$.    \begin{enumerate}[(i)]\item  If $1<q$, $(\oplus_{n=1}^\infty X_n)_{\ell_p}\notin \bigcup_{p<r<\infty}\textsf{\emph{T}}_{\xi,r}.$  \item $(\oplus_{n=1}^\infty X_n)_{T_q^*}\notin \textsf{\emph{T}}_{\xi,p}$. \item If $1<q$, $(\oplus_{n=1}^\infty X_n)_{\mathfrak{Y}_q(\omega)}\notin \textsf{\emph{A}}_{\xi,p}$. \item $(\oplus_{n=1}^\infty X_n)_{\mathfrak{X}_q(\omega)}\notin \textsf{\emph{N}}_{\xi,p}$.   \end{enumerate}

\label{p1}
\end{proposition}

\begin{proof} By Lemma \ref{ravenous}$(i)$,  there exists a weakly null collection $(x_t)_{t\in \Gamma_{\xi, \infty}.D}$ of norm at most $1$ vectors and for each $n\in\nn$, there exists a weakly null collection $(x_t^n)_{t\in \Gamma_{\xi,n}.D}$ of norm at most $1$ vectors such that for each $\varnothing=t_0<t_1<\ldots$, $t_i\in MAX(\Lambda_{\xi, \infty, n}.D)$ (resp. for each $n\in \nn$ and each $\varnothing=t_0<\ldots <t_n$, $t_i\in MAX(\Lambda_{\xi, n,i}.D)$), there exist $k_1<k_2<\ldots $ (resp. $k_1<\ldots <k_n$) such that $(\sum_{i=1}^\infty \sum_{t_{i-1}< s\leqslant t_i}\mathbb{P}_{\xi, \infty}(s)x_s)_{i=1}^\infty$ (resp. $(\sum_{i=1}^n \sum_{t_{i-1}<s\leqslant t_i}\mathbb{P}_{\xi,n}(s)x_s^n)_{i=1}^n$) is $2$-equivalent to $(e_{k_i})_{i=1}^\infty$ (resp. $(e_{k_i})_{i=1}^n$) in the space $E$, where $E=\ell_p$ in case $(i)$, $E=T_q^*$ in case $(ii)$, $E=\mathfrak{Y}_q(\omega)$ in case $(iii)$, and $E=\mathfrak{X}_q(\omega)$ in case $(iv)$.    

In case $(i)$, we can use the collection $(x^n_t)_{t\in \Gamma_{\xi,n}.D}$ to see that $n\vartheta_{\xi,n}((\oplus_{n=1}^\infty X_n)_{\ell_p}) \geqslant n^{1/p}/2$, whence $(\oplus_{n=1}^\infty X_n)_{\ell_p}\notin \bigcup_{p<r<\infty}\textsf{N}_{\xi,r}=\bigcup_{p<r<\infty}\textsf{T}_{\xi,r}$.

In case $(ii)$, we can use the collection $(x_t)_{t\in \Gamma_{\xi, \infty}.D}$ to see that $(\oplus_{n=1}^\infty X_n)_{T_q^*}\notin \textsf{T}_{\xi,p}$. In order to see this, note that if $(\oplus_{n=1}^\infty X_n)_{T_q^*}\in \textsf{T}_{\xi,p}$, there exists some $\varnothing=t_0<t_1<\ldots$ such that $(\sum_{i=1}^\infty \sum_{t_{i-1}<s\leqslant t_i} \mathbb{P}_{\xi, \infty}(s)x_s)_{i=1}^\infty$ is dominated by the $\ell_p$ (resp. $c_0$ if $p=\infty$) basis, which by $2$-equivalence would yield the existence of a subsequence of the $T_q^*$ basis which is dominated by the $\ell_p$ (resp. $c_0$) basis, a contradiction.

In case $(iii)$, we can use the collections $(x^n_t)_{t\in \Gamma_{\xi,n}.D}$, $n=1, 2, \ldots$ together with the $1$- symmetry of the $\mathfrak{Y}_q(\omega)$ basis and the fact that the $\mathfrak{Y}_q(\omega)$ basis is isometrically equivalent to the functionals $(e^*_i)_{i=1}^\infty\subset V_q^*$   to see that $$\alpha_{\xi,p}\Bigl((\oplus_{n=1}^\infty X_n)_{\mathfrak{Y}_q(\omega)}\Bigr) \geqslant \sup_{n\in\nn} \sup\Bigl\{\frac{1}{2}\|\sum_{i=1}^n a_ie^*_i\|_{V_q^*}: (a_i)_{i=1}^n\in B_{\ell_p^n}\Bigr\}=\infty.$$   This yields that $(\oplus_{n=1}^\infty X_n)_{\mathfrak{Y}_q(\omega)}\notin \textsf{A}_{\xi,p}$.

In case $(iv)$, we can use the collection $(x^n_t)_{t\in \Gamma_{\xi,n}.D}$ together with $1$- symmetry of the $\mathfrak{X}_q(\omega)$ basis and its equivalence to the functionals $(e^*_i)_{i=1}^\infty \subset SM_q$ to see that $$\theta_{\xi,n}\Bigl((\oplus_{n=1}^\infty X_n)_{\mathfrak{X}_q(\omega)}\Bigr)n^{1/q} \geqslant \frac{1}{2}\cdot \log_2(n+1)^{1/p},$$ so $(\oplus_{n=1}^\infty X_n)_{\mathfrak{X}_q(\omega)}\notin \textsf{N}_{\xi,p}$.

\end{proof}

The proof of the following proposition is an inessential modification of the previous proposition, using Lemma \ref{ravenous}$(ii)$ in place of Lemma \ref{ravenous}$(i)$.  We therefore omit the proof.

\begin{proposition} Suppose $\xi$ is a limit ordinal and $\Gamma\subset \xi$ is a subset with $\sup\Gamma=\xi$.   Suppose also that $(X_\gamma)_{\gamma\in \Gamma}$ is a collection of Banach spaces such that $\theta_{\gamma, 1}(X_\gamma)=1$ for each $\gamma\in \Gamma$.  Fix $1\leqslant q<\infty$ and let $1/p+1/q=1$.    \begin{enumerate}[(i)]\item  If $1<q$, $(\oplus_{\gamma\in \Gamma} X_\gamma)_{\ell_p(\Gamma)}\notin \bigcup_{p<r<\infty}\textsf{\emph{T}}_{\xi,r}.$  \item If $\Gamma=(\gamma_n)_{n=1}^\infty$, $(\oplus_{n=1}^\infty X_{\gamma_n})_{T_q^*}\notin \textsf{\emph{T}}_{\xi,p}$. \item If $1<q$, $(\oplus_{\gamma\in \Gamma} X_\gamma)_{\mathfrak{Y}_q(\Gamma)}\notin \textsf{\emph{A}}_{\xi,p}$. \item $(\oplus_{\gamma\in \Gamma} X_\gamma)_{\mathfrak{X}_q(\Gamma)}\notin \textsf{\emph{N}}_{\xi,p}$.   \end{enumerate}

\label{p2}
\end{proposition}

\begin{fact} For every ordinal $\zeta$ and $n\in\nn$, there exists a Banach space $X_{\zeta, n}$ such that $Sz(X_{\zeta,n})=\omega^{\zeta+1}$ and $\theta_{\zeta, n}(X_{\zeta, n})=1$.    Moreover, this space $X_{\zeta,n}$ may be taken to be reflexive,  have a $1$-unconditional basis, and may be taken to be separable if $\zeta$ is countable.   Indeed, we may take $$X_{0,1}=\ell_2,$$ $$X_{\xi,n}=\ell_1^n(X_{\xi,1}),$$ $$X_{\xi+1, 1}=\bigl(\oplus_{n=1}^\infty X_{\xi,n})_{\ell_2},$$ and if $\xi$ is a limit ordinal, $$X_{\xi,1}=\bigl(\oplus_{\zeta<\xi} X_{\zeta+1,1}\bigr)_{\ell_2(\xi)}.$$

\label{good}
\end{fact}

\begin{proof}[Proof of Theorems \ref{tm1} and \ref{tm2}, $0<\xi$] Throughout the proof, let $X_{\zeta, n}$ be a Banach space with $Sz(X_{\zeta, n})=\omega^{\zeta+1}$ and $\theta_{\zeta, n}(X_{\zeta, n})=1$.

First suppose that $\xi$ is a successor, say $\xi=\zeta+1$.   Then for $1\leqslant q<\infty$ and $1/p+1/q=1$, combining Corollary \ref{one side} with Proposition \ref{p1}, \begin{enumerate}[(i)]\item if $1<q$, $\bigl(\oplus_{\zeta<\xi} X_{\zeta, n}\bigr)_{\ell_p}\in \textsf{T}_{\xi,p}\setminus \bigcup_{p<r<\infty}\textsf{T}_{\xi,r}$, \item $\bigl(\oplus_{n=1}^\infty X_{\zeta,n}\bigr)_{T_q^*}\in \textsf{A}_{\xi,p}\setminus \textsf{T}_{\xi,p}$, \item if $1<q$, $\bigl(\sum_{n=1}^\infty X_{\zeta, n}\bigr)_{\mathfrak{Y}_q(\omega)}\in \textsf{N}_{\xi,p}\setminus \textsf{A}_{\xi,p}$, and \item $\bigl(\oplus_{n=1}^\infty X_{\zeta, n}\bigr)_{\mathfrak{X}_q(\omega)}\in \textsf{P}_{\xi,p}\setminus \textsf{N}_{\xi,p}$.  \end{enumerate}

Now suppose that $\xi$ is a limit ordinal. Combining Corollary \ref{one side} with Proposition \ref{p2}, \begin{enumerate}[(i)]\item if $1<q$, $\bigl(\oplus_{\gamma<\xi} X_{\gamma,1}\bigr)_{\ell_p(\xi)}\in \textsf{T}_{\xi,p}\setminus \bigcup_{p<r<\infty}\textsf{T}_{\xi,r}$, \item if $1<q$, $\bigl(\sum_{\gamma<\xi} X_{\gamma, 1}\bigr)_{\mathfrak{Y}_q(\xi)}\in \textsf{N}_{\xi,p}\setminus \textsf{A}_{\xi,p}$, and \item $\bigl(\oplus_{\gamma<\xi} X_{\gamma, 1}\bigr)_{\mathfrak{X}_q(\xi)}\in \textsf{P}_{\xi,p}\setminus \textsf{N}_{\xi,p}$.  \end{enumerate}

If $\xi$ is a limit ordinal with countable cofinality, we may fix $\xi_n\uparrow \xi$ and combine Corollary \ref{one side} with Proposition \ref{p2} to deduce that $(\oplus_{n=1}^\infty X_{\xi_n, 1})_{T_q^*}\in \mathfrak{A}_{\xi,p}\setminus \mathfrak{T}_{\xi,p}$.

\end{proof}

\begin{rem}\upshape The class of limit ordinals with countable cofinality is a proper class.  Therefore Theorems \ref{tm1} and \ref{tm2} produce results for uncountable ordinals. For example, they yield the distinctness of $\mathfrak{T}_{\omega_1 \omega, p}$ and $\mathfrak{A}_{\omega_1 \omega, p}$, since $\omega_1 \omega=\lim_n \omega_1 n$.  Here $\omega_1$ denotes the first uncountable ordinal.

\end{rem}

\begin{question} Are $\mathfrak{T}_{\xi,p}$ and $\mathfrak{A}_{\xi,p}$ distinct for all ordinals? 

\end{question}

We last conclude with the promised modification of the example of Lindenstrauss.  Let $E$ be a  Banach space such that the canonical $c_{00}$ basis is a  normalized, $1$-unconditional, boundedly-complete basis for $E$.  Let $F$ be a finite dimensional, non-zero Banach space and let $(x_n)_{n=1}^\infty$ be a dense sequence in $S_F$.     Let us define the norm $\|\cdot\|_{Y^*_E}$ on $c_{00}$ by $$\|\sum_{i=1}^\infty a_ie_i\|_{Y^*_E} = \sup \Bigl\{\bigl\|\sum_{j=1}^\infty \|\sum_{i=p_{j-1}+1}^{p_j}a_ix_i\|_F e_{p_j}\|_E: 0=p_0<p_1<\ldots \Bigr\}.$$   Let $Y^*_E$ denote the completion of $c_{00}$ with respect to this norm.  When no confusion can arise, we will write $Y^*$ in place of $Y^*_E$.   The following facts are standard, and follow as in \cite{Lindenstrauss}.

\begin{proposition} \begin{enumerate}[(i)]\item The canonical $c_{00}$ basis is a normalized, bimonotone, boundedly-complete basis for $Y^*$. \item The map $Q:Y^*\to F$ given by $Q\sum_{i=1}^\infty a_ie_i=\sum_{i=1}^\infty a_ix_i$ is a well-defined quotient map with norm $1$. \end{enumerate}

\end{proposition}

Since the canonical $c_{00}$ basis is a normalized, bimonotone, boundedly-complete for $Y^*$, then the coordinate functionals $(e^*_i)_{i=1}^\infty$ form a normalized, bimonotone, shrinking basis for a predual of $Y^*$, which we call $Y$ (or $Y_E$ if we wish to specify the choice of $E$).

\begin{lemma} Suppose $1<r<\infty$, $C\geqslant 1$ are such that the basis of $E$ satisfies a $C$-$\ell_r$ upper block estimate.  \begin{enumerate}[(i)]\item If $(y_i^*)_{i=1}^n$ is a block sequence in $Y^*$ and $(b_i)_{i=1}^n$ are any scalars, $$\|\sum_{i=1}^n b_iy_i^*\|_{Y^*} \leqslant \sum_{i=1}^n |b_i|\|Qy_i^*\| + 5C\bigl(\sum_{i=1}^n |b_i|^r \|y_i^*\|^r\bigr)^{1/r}.$$ \item If $1<p\leqslant \infty$ and there exists a constant $C_1$ such that for any $n\in\nn$ and any block sequence $(z_i)_{i=1}^n$ in $B_E$, then $\|\sum_{i=1}^n z_i\|_E\leqslant C_1 n^{1/p}$, then for any $n\in\nn$ and any  block sequence $(y_i^*)_{i=1}^n$ in $B_{Y^*}$, $$\|\sum_{i=1}^n y_i^*\|\leqslant \sum_{i=1}^n \|Qy_i^*\|+ 5C_1 n^{1/p}.$$   \item If $1< p\leqslant \infty$ and if the basis of $E$ is asymptotic $\ell_p$ (resp. $c_0$ if $p= \infty$), then there exists a constant $C_2$ such that for any $n\in\nn$ and any block sequence $(y_i^*)_{i=1}^n $ in $B_{Y^*}$ with $n\leqslant y_1^*$ and any scalars $(b_i)_{i=1}^n$, $$\|\sum_{i=1}^n b_iy_i^*\|_{Y^*}\leqslant \sum_{i=1}^n |b_i|\ee_i + C_2\|\sum_{i=1}^n b_i e_i\|_{\ell_p^n}.$$    \item A bounded block sequence $(y_i^*)_{i=1}^\infty $ in $Y^*$ is weakly null if and only if $(Qy_i^*)_{i=1}^\infty$ is norm null.  \item For any $y^{**}\in Y^{**}$, there exists a unique $x^*\in F^*$ with $\|x^*\|\leqslant \|y^{**}\|$ such that if $(y_i^*)_{i=1}^\infty\subset B_{Y^*}$ is a bounded block sequence such that $(Qy_i^*)_{i=1}^\infty$ is convergent to $x\in F$, then $\lim_i y^{**}(y_i^*)=\lim_i x^*(x)$.   Furthermore, $y^{**}-Q^*x^*\in JY$.    \end{enumerate}

\label{pete}

\end{lemma}

\begin{proof}$(i)$ Fix $n\in\nn$, a block sequence $(y^*_i)_{i=1}^n$, scalars  $(b_i)_{i=1}^n$,  and intervals $I_1<I_2<\ldots$ with $\cup_{i=1}^\infty I_i=\nn$.     We may separate $\{1, \ldots, n\}$ into two sets $S,T$ such that for each $i\in S$, there exists an interval $I_{j_i}$ such that $\text{supp}(y^*_i)\subset I_{j_i}$ and for each $i\in T$, there exist at least two values of $j$ such that $I_j\cap \text{supp}(y^*_i)\neq \varnothing$. Of course, $$\Bigl\|\sum_{j=1}^\infty \|QI_j \sum_{i\in S} b_i y^*_i\|_F e_{\max I_j}\|_E \leqslant \sum_{i\in S}|b_i|\|Qy_i^*\| \leqslant \sum_{i=1}^n |b_i|\|Qy^*_i\|.$$

Now for each $i\in T$, we may let $$p_i=\min\{j: I_j\cap \text{supp}(y^*_i)\neq \varnothing\},$$ $$q_i= \max \{j: I_j \cap \text{supp}(y^*_i)\neq \varnothing\},$$ and note that $p_i<q_i\leqslant p_{i+1}$.  Note that $p_i\geqslant \min \text{supp}(y^*_1)\geqslant n$ for all $i\in T$.    For each $i\in T$, let $$v_i=\sum_{j=p_i+1}^{q_i-1} \|QI_j y^*_i\|e_{\max I_j}\in \|y^*_i\|B_E,$$ where $v_i=0$ if $p_i+1=q_1$.    Let $M=\{p_i, q_i: i\in T\}$, so that \begin{align*} \|\sum_{j\in \nn\setminus M} \|QI_j\sum_{i\in T} b_iy^*_i\|e_{\max I_j}\|= \|\sum_{j\in \nn\setminus M}\sum_{i\in T} b_i \|QI_jy^*_i\|e_{\max I_j}\| = \|\sum_{i\in T} b_i v_i\|_E \leqslant C\bigl(\sum_{i=1}^n|b_i|^r\|y_i^*\|^r)^{1/r}.\end{align*}    Now if $T\neq \varnothing$, write  $T=(t_i)_{i=1}^{|T|}$ and $z^*_i=y^*_{t_i}$, $a_i=b_{t_i}$.     For convenience, let $z^*_0=z^*_{|T|+1}=0$ and $a_0=a_{|T|+1}=0$. Note that for each $1\leqslant l\leqslant |T|$, $$\|QI_{p_{t_l}}\sum_{i\in T}b_iy^*_i\|\leqslant |a_{l-1}|\|z^*_{l-1}\|+|a_l|\|z^*_l\|$$ and $$\|QI_{q_{t_l}}\sum_{i\in T} b_iy^*_i\|\leqslant |a_l|\|z^*_l\|+|a_{l+1}|\|z^*_{l+1}\|.$$   Then \begin{align*} \|\sum_{j\in M} \|QI_j\sum_{i\in T}b_iy^*_i\| e_{\max I_j}\| & \leqslant \|\sum_{l=1}^{|T|} \|QI_{p_{t_l}} \sum_{i\in T}b_iy^*_i\|e_{\max I_{p_{t_l}}} \| + \|\sum_{l=1}^{|T|} \|QI_{q_{t_l}}\sum_{i\in T} b_i y^*_i\|e_{\max I_{q_{t_l}}} \| \\ & \leqslant \|\sum_{l=2}^{|T|} b_{l-1}\|z^*_{l-1}\|e_{\max I_{p_{t_{l-1}}}} \| + \|\sum_{l=1}^{|T|-1} b_{l+1} \|z^*_{l+1}\| e_{\max I_{q_{t_{l+1}}}}\| \\ & + \|\sum_{l=1}^{|T|} |b_l|\|z^*_l\|e_{\max I_{p_{t_l}}}\|+\|\sum_{l=1}^{|T|}|b_l|\|z^*_l\|e_{\max I_{q_{t_l}}}\| \\ & \leqslant 4C\bigl(\sum_{l=1}^{|T|} |b_l|^r\|z^*_l\|^r\bigr)^{1/r} \leqslant 4C\bigl(\sum_{i=1}^n  |a_i|^r\|y^*_i\|^r\bigr)^{1/r}.  \end{align*}    Since this holds for any sequence of intervals $I_1<I_2<\ldots $ with $\cup_{i=1}^\infty I_i=\nn$, $$\|\sum_{i=1}^n b_i y^*_i\|\leqslant \sum_{i=1}^n |b_i|\|Qy^*_i\|+ 5C\bigl(\sum_{i=1}^n |b_i|^r\|y^*_i\|^r\bigr)^{1/r}.$$

$(ii)$  This is similar to $(i)$, noting that $|T|\leqslant n$.

$(iii)$ We fix $C_1$ such that for any $n\leqslant z_1<\ldots <z_n$, $\|\sum_{i=1}^n z_i\|_E \leqslant C_1\|\sum_{i=1}^n \|z_i\|_E e_i\|_{\ell_p^n}$.   Then we can take $C_2=5C_1$.    The argument that the desired inequality holds is similar to the argument in $(i)$, noting that, with $I_1<I_2<\ldots$, $S,T$ as defined in $(i)$, $n\leqslant \min\{p_i, q_i, \min \text{supp}(v_i): i\in T\}$.

$(iv)$ If $(y_i^*)_{i=1}^\infty$ is weakly null, then by compactness of $Q$, $(Qy_i^*)_{i=1}^\infty$ is norm null.  Seeking a contradiction, suppose $(Qy_i^*)_{i=1}^\infty$ is norm null but $(y_i^*)_{i=1}^\infty$ is not weakly null. Fix $b>\sup_i \|y_i^*\|$.    Then after passing to a subsequence, we may assume there exists $\ee>0$ such that $$\inf\{\|y^*\|: y^*\in \text{co}(y_i^*:i\in\nn)\}\geqslant \ee$$ and such that $\|Qy_i^*\|<\ee/2$ for all $i\in\nn$.   Let $1/r+1/s=1$ and note that $s<\infty$.    Fix $n\in\nn$ such that $5bC/n^{1/s}<\ee/2$ and note that by $(i)$, $$\|\frac{1}{n}\sum_{i=1}^n y_i^*\| \leqslant \frac{1}{n}\sum_{i=1}^n \|Qy_i^*\|+ 5bC/n^{1/s}<\ee,$$ a contradiction.

$(v)$    Suppose $(y_i^*)_{i=1}^\infty\subset B_{Y^*}$ is such that $Qy_i^*\to x\in B_F$ in norm. We first prove that $y^{**}(y_i^*)$ converges. If it did not, then since $(y^{**}(y_i^*))_{i=1}^\infty$ is bounded, we could, after passing to a subsequence and relabeling, assume there exist distinct scalars $a,b$ such that $y^{**}(y_{2i}^*)\to a$ and $y^{**}(y_{2i+1}^*)\to b$.    Then $(y_{2i}^*-y_{2i+1}^*)_{i=1}^\infty$ is a bounded block sequence in $Y^*$ with $Q(y_{2i}^*-y_{2i+1}^*)\to 0$ in norm.  But this means $(y_{2i}^*-y_{2i+1}^*)_{i=1}^\infty$ is weakly null by $(iv)$, but $y^{**}(y_{2i}^*-y_{2i+1}^*)\to a-b\neq 0$.   This yields that $(y^{**}(y_i^*))_{i=1}^\infty$ must be convergent.

For an arbitrary $x\in X$, we now define $x^*(x)=\lim_i y^{**}(y_i^*)$, where $(y_i^*)_{i=1}^\infty$ is a bounded block sequence in $Y^*$ such that $Qy_i^*\to x$.  We note that this definition is independent of the particular choice of $(y_i^*)_{i=1}^\infty$.  Indeed, if $(y_i^*)_{i=1}^\infty, (z_i^*)_{i=1}^\infty$ are bounded block sequences with $Qy_i^*, Qz_i^*\to x$, then we may select $r_1<s_1<r_2<s_2<\ldots$ such that $(y_{r_i}^*-z_{s_i}^*)_{i=1}^\infty$ is a block sequence in $Y^*$.  Since this sequence is bounded and $Q(y_{r_i}^*-z_{s_i}^*)\to 0$, $(y_{r_i}^*-z_{s_i}^*)_{i=1}^\infty$ is weakly null, so $$\lim_i y^{**}(y_i^*)=\lim_i y^{**}(y^*_{r_i})= \lim_i y^{**}(z^*_{s_i})= \lim_i y^{**}(z_i^*).$$     

Now for any $x\in X$, we may fix $k_1<k_2<\ldots$ such that $\|\|x\|x_{k_i}-x\|\to 0$ and let $y^*_i= \|x\|e_{k_i}$.    Then $$|x^*(x)|=\lim_i |y^{**}(y^*_i)| \leqslant \|x\|\|y^{**}\|.$$  Thus $x^*$ is continuous and $\|x^*\|\leqslant \|y^{**}\|$.    Uniqueness of $x^*$ is obvious.  

We last need to show that $y^{**}-Q^*x^*\in JY$, for which it is sufficient to prove that $y^{**}-Q^*x^*$ is in the closed span of $(e^*_i)_{i=1}^\infty$.   If $y^{**}-Q^*x^*$ is not in the closed span of $(e_i^*)_{i=1}^\infty$, there would exist $\ee>0$ and a block sequence $(y^*_i)_{i=1}^\infty \subset B_{Y^*}$ such that $|(y^{**}-Q^*x^*)(y^*_i)|\geqslant \ee$ for all $i\in\nn$.   Since $B_F$ is compact, we may pass to a subseuqence and assume $Qy_i\to x\in B_F$.    Then \begin{align*} \ee & \leqslant \underset{i}{\lim}|y^{**}(y_i^*)-Qy^*_i(x)|=0,\end{align*}  and this contradiction concludes $(v)$.

\end{proof}

\begin{corollary} $Y^{**}=JY\oplus Q^*F^*$.

\end{corollary}

\begin{proof} For $y^{**}\in Y^{**}$, let $x^*$ be the functional guaranteed to exist by Lemma \ref{pete}$(v)$ and let $Py^{**}=Q^*x^*$.   Observe that if $y\in JY$, then for any bounded block sequence $(y_i^*)_{i=1}^\infty$ in $Y^*$, $y^*_i(y)\to 0$, whence $Py=0$.    Furthermore, if $y^{**}\in Y^{**}$, then $$y^{**}= y^{**}-Py^{**} + Py^{**}\in JT+Q^*Y^*.$$    From this it follows that $Y^{**}=JT\oplus Q^*Y^*$.   

\end{proof}

The next result follows without modification as in \cite[Corollary 1]{Lindenstrauss}. 

\begin{corollary} If $E$ is a Banach space with normalized, $1$-unconditional, boundedly-complete basis satisfying $\ell_r$ upper block estimates for some $1<r<\infty$, then then $Z:=\ker(Q:Y^*_E\to F)$ is such that $Z^{**}/Z$ is $2$-isomorphic to $F$. 

\label{linden}

\end{corollary}

\begin{theorem} Let $F$ be a finite dimensional Banach space and $1<p\leqslant \infty$. Then there exist Banach spaces $Z_1, Z_2, Z_3, Z_4$ such that for each $j=1,2,3,4$, $Z_j^{**}/Z_j$ is $2$-isomorphic to $F$, and   \begin{enumerate}[(i)]\item If $p<\infty$, $Z_1\in \textsf{\emph{T}}_{0, p}\setminus \bigcup_{p<r<\infty} \textsf{\emph{P}}_{0,p}$.    \item $Z_2\in \textsf{\emph{A}}_{0,p}\setminus \textsf{\emph{T}}_{0,p}$, \item if $p<\infty$, $Z_3\in \textsf{\emph{N}}_{0,p}\setminus \textsf{\emph{A}}_{0,p}$, \item $Z_4\in \textsf{\emph{P}}_{0,p}\setminus \textsf{\emph{N}}_{0,p}$. \end{enumerate}

\end{theorem}

\begin{proof} We have already shown the $F=\{0\}$ case (with $\ell_p, T^*_q, \text{Ti}_q^*$, and $S^*_q$ for the four cases, respectively). So assume $F\neq \{0\}$.     For $(i)$, $(ii)$, $(iii)$, and $(iv)$, respectively, we let $E=\ell_p, T^*_q, \text{Ti}_q^*, S^*_q$. We let $Y^*_E$ and  $Z_j=\ker(Q:Y^*_E\to F)$ be chosen as in Corollary \ref{linden} so that $Z_j^{**}/Z_j$ is $2$-isomorphic to $F$. We remark that in each case, $E$ is reflexive with normalized, $1$-unconditional basis satisfying an $\ell_r$ upper block estimate for each $1<r<p$.

We will next show that in each case, $Z_j$ is in the appropriate class.  Fix $j\in \{1, 2,3,4\}$.     Fix a sequence $(\ee_n)_{n=1}^\infty$ of positive numbers such that $\sum_{n=1}^\infty \ee_n\leqslant 1$.      Now suppose that $E$ is a weak neighborhood basis at $0$ in $Z_j$ and suppose that  $(u_t)_{t\in D^{\leqslant n} }\subset B_{Z_j}$ is weakly null.    We may fix $t_1<\ldots <t_n \in D^n$ and a block sequence $n\leqslant z_1<\ldots <z_n$, $z_i\in B_{Z_j}$. such that for each $1\leqslant i\leqslant n$, $\|z_i-x_{t|_i}\|<\ee_i$.  Since $Qx_{t|_i}=0$ and $Q$ is norm $1$, it follows that $$\|Qz_i\|\leqslant \|Qx_{t|_i}\|+\|z_i-x_{t|_i}\|\leqslant \ee_i.$$     In the case $j=2$,  using Lemma \ref{pete} and the properties of $E=T^*_q$, we have that for any scalar sequence $(a_i)_{i=1}^n\in B_{\ell_p^n}$, \begin{align*} \|\sum_{i=1}^n a_i x_{t|_i}\| & \leqslant 1+\|\sum_{i=1}^n a_i z_i\| \leqslant 2+C_2. \end{align*}   From this it follows that $\alpha_{0,p}(Z_2)\leqslant 2+C_2<\infty$, and $Z_2\in \textsf{A}_{0,p}$.    Now in the case $j=3$, we deduce that $$\|\sum_{i=1}^n x_{t|_i}\|\leqslant 1+\|\sum_{i=1}^n z_i\|+ 2+C_2n^{1/p} \leqslant (2+C_2)n^{1/q}.$$    This yields that $Z_3\in \textsf{N}_{\xi,p}$.      

Now suppose that $(x_t)_{t\in D^{<\nn }}\subset B_X$ is weakly null.  Let $r=p$ in the case $j=1$ and let $1<r<p$ in the case $j=4$. In either case, there exists a constant $C$ such that the basis of $E$ satisfies $C$-$\ell_r$ upper block estimates.    We may fix $t_1<t_2<\ldots$ such that $|t_i|=i$ and a block sequence $(z_i)_{i=1}^\infty \subset B_{Z_j}$ such that $\|z_i-x_{t|_i}\|<\ee_i$ and $\|Qz_i\|<\ee_i$.   Then by Lemma \ref{pete}, there exists a constant $C_2$ (depending only on the constant $C$) such that for any $(a_i)_{i=1}^\infty \in c_{00}\cap B_{\ell_r}$, $$\|\sum_{i=1}^\infty a_ix_{t|_i}\|\leqslant 1+ \|\sum_{i=1}^\infty a_i z_i\|\leqslant 2+C_2.$$   This yields that $Z_j$ satisfies $\ell_r$ upper tree estimates and therefore lies in $Z_j\in \textsf{T}_{0,r}$.  This completes the case $j=1$.  For the case $j=4$, we note that since $1<r<p$ was arbitrary, $$Z_4\in \bigcap_{1<r<p}\textsf{T}_{0,r} = \textsf{P}_{0,p}.$$

We last conclude that in each case, $Z_j$ is not in the appropriate class.    Fix $j\in \{1, 2, 3, 4\}$.  Now fix any $x\in S_F$ and select $p_1<q_1<p_2<q_2<\ldots$ such that $\|x-x_{p_n}\|, \|x-x_{q_n}\|<1/2^{n+1}$.    Let $y^*_n=e_{p_n}-e_{q_n}$  and note that $\|Q(e_{p_n}-e_{q_n})\|<1/2^n$. From this it follows that there exists $v_n^*\in Y^*_E$ with $\|v_n^*\|<1/2^n$ such that $Q(y^*_n-v^*_n)=0$.  Therefore $z^*_n=y^*_n-v^*_n\in 3B_{Z_j}$, and $(z^*_n)_{n=1}^\infty$ is equivalent to $(y^*_n)_{n=1}^\infty$.   Furthermore, it is evident that $(y^*_n)_{n=1}^\infty$ $1$-dominates, and therefore $(z^*_n)_{n=1}^\infty$ $C$-dominates for some $C$, the sequence  $(e_{p_n})_{n=1}^\infty$ in $E$.   Indeed, with $q_0=0$, for any $(a_n)_{n=1}^\infty\in c_{00}$, with $b_{p_n}=b_{q_n}=a_n$ and $b_n=0$ for all $n\in \nn\setminus \{p_1, q_1, \ldots\}$,  \begin{align*} \|\sum_{n=1}^\infty a_n y_n^*\|_{Y^*_E} &  \geqslant \Bigl\|\sum_{i=1}^\infty \Bigl[\|\sum_{j=q_{i-1}+1}^{p_i} b_ix_i\|_Fe_{p_i} + \|\sum_{i=p_i+1}^{q_i} b_ix_i\|_F e_{q_i}\Bigr]\Bigr\|_E  \\ & \geqslant \Bigl\|\sum_{i=1}^\infty \|\sum_{j=q_{i-1}+1}^{p_i} b_ix_i\|_Fe_{p_i} \Bigr\|_E  = \| \sum_{i=1}^\infty a_ie_{p_i}\|_E.\end{align*}  

Now define $u_{(n_1, \ldots, n_k)}=z_{\sum_{j=1}^k n_i}^*\in 3B_{Z_j}$ and note that any  branch of the tree $(u_t)_{t\in \nn^{\leqslant n}}$ (resp. $(u_t)_{t\in \nn^{<\nn}}$) is simply a subsequence of $(z^*_n)_{n=1}^\infty$.   Now in the case $(i)$, it follows that $$\inf_{t\in \nn^n} \|\sum_{i=1}^n u_{t|_i}\|\geqslant n^{1/p}/3C,$$ yielding that $Z_1\notin \bigcup_{p<r<\infty} \textsf{P}_{0,r}$.     In the case $(ii)$, we can use $(u_t)_{t\in \nn^{<\nn}}$ to deduce that $Z_2\notin \textsf{T}_{0,p}$, since otherwise some branch of $(u_t)_{t\in \nn^{<\nn}}$, and therefore some subsequence of $(z^*_n)_{n=1}^\infty$ and therefore some subsequence of the basis of $T^*_q$, is dominated by the $\ell_p$ basis. But since no subsequence of the $T^*_q$ basis is dominated by the $\ell_p$ basis.    In case $(iii)$, we can use the collections $(u_t)_{t\in \nn^{\leqslant n}}$ to witness that $Z_3\notin \textsf{A}_{0,p}$. Indeed, seeking a contradiction, assume $\alpha_{0,p}(Z_3)<\infty$.    By subsymmetry of the $\text{Ti}_q^*$ basis, for any scalar sequence $(a_i)_{i=1}^n\in B_{\ell_p^n}$, $$\|\sum_{i=1}^n a_ie_i\|_{\text{Ti}^*_q} \leqslant C \inf_{t\in \nn^n} \|\sum_{i=1}^n a_i u_{t|_i}\|\leqslant 3C \alpha_{0,p}(Z_3),$$ yielding that the canonical  $\text{Ti}_q^*$ basis is equivalent to the $\ell_p$ basis.  But this contradicts the strict singularity of the formal inclusion $I:\ell_q\to \text{Ti}_q$.      Finally, in case $(iv)$, we can use the collections $(u_t)_{t\in \nn^{\leqslant n}}$ to see that $$\theta_{0,n}(Z_4)n^{1/q} \geqslant \log_2(n+1)^{1/q}/3C.$$

\end{proof}

\begin{rem}\upshape It is not possible to have a quasi-reflexive, infinite dimensional space $X\in \textsf{T}_{0,\infty}$. Indeed, if $X\in \textsf{T}_{0,\infty}$ has $\dim X=\infty$, then by passing to a subspace, we could assume this space is separable. But it is known that a separable space in $\textsf{T}_{0,\infty}$ is isomorphic to a subspace of $c_0$, whence this $X$ cannot be quasi-reflexive.

\end{rem}

\end{document}